\documentclass[11pt, a4paper]{article}

\usepackage[arxiv]{mymacros}
\usepackage[normalem]{ulem}

\usepackage{xr-hyper} 
\usepackage{hyperref}
\makeatletter
\newcommand*{\addFileDependency}[1]{
  \typeout{(#1)}
  \@addtofilelist{#1}
  \IfFileExists{#1}{}{\typeout{No file #1.}}
}
\makeatother
\newcommand*{\myexternaldocument}[2]{%
    \externaldocument{#1}[#2]
    \addFileDependency{#1.tex}%
    \addFileDependency{#1.aux}%
}
\myexternaldocument{dgauss}{https://arxiv.org/pdf/2202.02286v2.pdf}

\makeatletter
\newcommand{\customlabel}[2]{%
   \protected@write \@auxout {}{\string \newlabel {#1}{{#2}{\thepage}{#2}{#1}{}} }%
   \hypertarget{#1}{#2}
}
\makeatother

\makeatletter
\newcommand{\subjclass}[2][1991]{%
  \let\@oldtitle\@title%
  \gdef\@title{\@oldtitle\footnotetext{#1 \emph{Mathematics subject classification.} #2}}%
}
\makeatother

\usepackage{tikz}
\usepackage{tikz-cd} 

\title{The Discrete Gaussian model, II.\\ Infinite-volume scaling limit at high temperature}

\author{Roland Bauerschmidt\footnote{University of Cambridge, Statistical Laboratory, DPMMS. E-mail: {\tt rb812@cam.ac.uk}.}
  \and
  Jiwoon Park\footnote{University of Cambridge, Statistical Laboratory, DPMMS. E-mail: {\tt jp711@cam.ac.uk}.}
  \and
  Pierre-Fran\c{c}ois Rodriguez\footnote{Imperial College London, Department of Mathematics. E-mail: {\tt p.rodriguez@imperial.ac.uk}.}}

\date{\vspace*{-2em}} 


\usepackage[titles]{tocloft}
\newcommand{\betaeff}{\beta_{\operatorname{eff}}}

\definecolor{darkmagenta}{rgb}{0.55, 0.0, 0.55}
\definecolor{magenta}{rgb}{0.85, 0.0, 0.55}
\definecolor{grey}{rgb}{0.55, 0.55, 0.55}

\newcommand{\Rem}{\operatorname{Rem}}
\newcommand{\Loc}{\operatorname{Loc}}
\newcommand{\KK}{\vec{K}}
\newcommand{\MM}{\mathfrak{M}}
\newcommand{\Eplus}{\mathbb{E}}%_+}
\newcommand{\assumpu}{(\textbf{A}_u)} 

\renewcommand{\bar}[1]{\overline{#1}}
\newcommand{\alphaLoc}{\alpha_{\Loc}}

\begin{document}
\maketitle

\begin{abstract}
  The Discrete Gaussian model is the lattice Gaussian free field conditioned to be integer-valued.
  In two dimensions, at sufficiently high temperature,
  we show that the scaling limit of the infinite-volume gradient Gibbs state with zero mean
  is a multiple of the Gaussian free field.

  This article is the second in a series on the Discrete Gaussian model,
  extending the methods of the first paper by the analysis of general external fields
  (rather than macroscopic test functions on the torus). 
  As a byproduct, we also obtain a scaling limit for mesoscopic test functions on the torus.
\end{abstract}

\setcounter{tocdepth}{1}
\tableofcontents

\section{Introduction and main results}
\label{sec:intro-dgauss2}

This is the second article in a series on the Discrete Gaussian model, which builds on the foundation
provided by the first paper \cite{dgauss1}.
The Discrete Gaussian model is the Gaussian free field conditioned to be integer-valued.
Its two-dimensional version is a model for a crystal interface (in 2+1 dimensions) undergoing a roughening transition, see \cite[Section 6]{MR1446000} for a textbook treatment.
We refer to our first paper \cite{dgauss1} for a more extensive introduction and discussion of the literature.

\subsection{Discrete Gaussian model in infinite volume}

In our first paper \cite{dgauss1}, we studied the scaling limit of the Discrete Gaussian model
for macroscopic test functions on the torus.
In the present article, we derive the scaling limit of its infinite-volume gradient Gibbs state,
as well as the scaling limit for mesoscopic test functions on the torus, which is a byproduct
of the proof of the infinite-volume result. These scaling limit results are the objects of Theorems~\ref{thm:highbeta-Z2} and~\ref{thm:highbeta-meso} below.

The infinite-volume limit of the two-dimensional Discrete Gaussian model
will be taken through weak limits with periodic boundary conditions, cf.~\eqref{eq:TDlimit},
and we permit a general finite-range interaction $J$ in the definition of the model.
To be precise, let $J \subset \Z^d \setminus \{0\}$ be finite and symmetric under reflections and lattice rotations,
and define the associated normalised range-$J$ Laplacian $\Delta_J$ by
\begin{equation}
  (\Delta_J f)(x) = \frac{1}{|J|}\sum_{y\in J} (f(x+y)-f(x)), \label{eq:Delta_J_definition-dgauss2}
\end{equation}
for $f:\Z^d\to \R$, 
where $|J|$ denotes the number of elements of $J$.
Acting on test functions having mean zero, $ (-\Delta_J)^{-1}$ has kernel
\begin{equation}
  (-\Delta_J)^{-1}(x,y)  \sim -\frac{1}{2\pi v_J^2} \log|x-y|,  \text{ as $|x-y|
\to \infty$,} \qquad \text{ where } \qquad
  v_J^2 = \frac{1}{2|J|} \sum_{x \in J} x_1^2.
\end{equation}
We now introduce the relevant finite-volume states. Let $\Lambda_N$ be a two-dimensional discrete torus of side length $L^N$ for integers $L>1,N \geq 1$, and fix an origin $0\in \Lambda_N$.
Given the above step distribution $J$,
the \emph{Discrete Gaussian model} on $\Lambda_N$ at temperature $\beta \in (0,\infty)$
has expectation, for any $F: (2\pi\Z)^{\Lambda_N} \to \R$ with $F(\sigma) = F(\sigma+c)$ for any constant $c\in 2\pi\Z$ and such that the following series converges, defined by
\begin{equation}
  \avg{F}_{J,\beta}^{\Lambda_N}
  \propto \sum_{\sigma \in \Omega^{\Lambda_N}} e^{-\frac{1}{2\beta} (\sigma,-\Delta_J\sigma)} \, F(\sigma)
  = \sum_{\sigma \in \Omega^{\Lambda_N}} e^{-\frac{1}{4\beta |J|} \sum_{x-y \in J} (\sigma_x-\sigma_y)^2} \, F(\sigma)
  \label{eq:DG_model_1_external_field}
\end{equation}
where the sum over $x-y\in J$ counts every undirected edge $\{x,y\}$ twice and
\begin{equation} \label{e:Omega-def_external_field}
  \Omega^{\Lambda_N} = \{\sigma \in (2\pi\Z)^{\Lambda_N}: \sigma_{x=0} =0 \}.
\end{equation}
Note that, as in our first paper \cite{dgauss1},
the factors of $2\pi$ in the spacing of the integers in \eqref{e:Omega-def_external_field} are convenient
(but could be absorbed by rescaling $\beta$), and,
to relate better to the Coulomb gas literature (cf.~references below),
we use $\frac 1\beta$ rather than $\beta$ to denote the inverse temperature of the Discrete Gaussian model.
Equivalent to considering $\sigma$ modulo constants,
one can consider  the
gradient field $\eta = (\eta_e)_{e \in E}$ where 
$E$ are the directed nearest-neighbour edges of $\Z^2$ and
$\eta_e = \sigma_x-\sigma_y$ when $e=(x,y)$.
Known correlation inequalities imply that, for any integer $L>1$ and any finite-range distribution  $J$,
the weak limit of
$\avg{\cdot}_{J,\beta}^{\Lambda_N}$ as $N\to\infty$ exists (modulo constants or as a gradient field),
see Appendix~\ref{app:corrineq}.
For concreteness, we define the infinite-volume limit in terms tori of side lengths $2^N$, i.e., when $\Lambda_N$ has side length $2^N$,
\begin{equation}\label{eq:TDlimit}
  \avg{\cdot}_{J,\beta}^{\Z^2}
  := \lim_{N\to\infty} \avg{\cdot}_{J,\beta}^{\Lambda_N} .
\end{equation}
This limit $\avg{\cdot}^{\Z^2}_{J,\beta}$
is a translation-invariant gradient Gibbs measure and every ergodic measure $\avg{\cdot}$ in its extremal decomposition has zero mean, i.e., $\avg{\eta_e}=0$ for all $e \in E$,
also see Appendix~\ref{app:corrineq}.
For $J=J_{\rm nn}$ the usual nearest-neighbour interaction, $\avg{\cdot}^{\Z^2}_{J,\beta}$ is \textit{the} unique ergodic gradient Gibbs measure with zero mean on account of Theorem~9.1.1 in \cite{MR2251117}.
For general $J$, such a characterisation has not been proved.

As is well-known (see refs.~below for an overview over the existing literature on the subject),
in the Discrete Gaussian model, the discreteness of the spins is responsible for a phase transition between a rough (or delocalised) high-temperature phase and an ordered (or localised) low-temperature phase.
Our results apply to large temperatures $\beta$.
In contrast, in the regime of small $\beta$, a Peierls expansion yields that the Discrete Gaussian field is localised (or `smooth'), e.g., there actually exists an (ordinary nongradient) Gibbs measure $\avg{\cdot}^{\Z^2}_{J,\beta}$ satisfying
\begin{equation}
\label{eq:DG-localized}
\avg{\sigma_x \sigma_y}_{J,\beta}^{\Z^2} -  \avg{\sigma_x }_{J,\beta}^{\Z^2} \avg{ \sigma_y}_{J,\beta}^{\Z^2} \leq Ce^{-c|x-y|}, \text{ for all $x, y$ and $\beta < c$};
\end{equation}
see also \cite{MR833220,MR3508158} for very precise results on the extremal behaviour in this regime.

\subsection{Main results}

Our main result is that the scaling limit of the Discrete Gaussian model $\avg{\cdot}^{\Z^2}_{J,\beta}$
defined above is a multiple of the
Gaussian free field on $\R^2$ when $\beta$ is large.
To state this precisely, given $f \in C_c^\infty(\R^2)$ with $\int_{\R^2} f(x) \, dx = 0$, 
let $f_\epsilon : \Z^2 \rightarrow \R$ satisfy $\sum_{x\in \Z^2} f_\epsilon(x)=0$ and, with $d=2$,
\begin{equation} \label{eq:feps_def}
  \begin{split}
   &\max_{0 \leq k \leq 2}\max_{x\in\Z^d}|(\epsilon^{-1}\nabla)^k f_\epsilon(x)| \leq C_f \epsilon^{1+d/2},
  \qquad
  \supp f_\epsilon \subset [-C_f\epsilon^{-1},C_f\epsilon^{-1}]^d,
  \\
  &\qquad\qquad\qquad\qquad\quad
   \max_{x\in\Z^d} \big|\epsilon^{-1-d/2} f_\epsilon(x)  -  f(\epsilon x)\big|  \to 0,
   \end{split}
\end{equation}
where $C_f$ is a constant
and $\nabla$ is the vector of discrete gradients on $\Z^2$, see Section~\ref{sec:notation-Z2}.
For example,
if $f = \nabla_i g$ for some $g \in C_c^\infty(\R^2)$ and $i \in \{1,2\}$ then one can take
$f_\epsilon(x) = \epsilon^{d/2}(g(\epsilon x+\epsilon e_i)-g(\epsilon x))$.
Thus the following scaling limit in particular implies that of the gradient field $\nabla \sigma$.

We use the notation $(u,v)_{\Z^2} = \sum_{x\in\Z^2} u(x)v(x)$ for $u,v:\Z^2 \to \R$ square summable,
$(f,g)_{\R^2} = \int_{\R^2} f(x)g(x)\, dx$ for $f,g:\R^2\to\R$ square integrable, and $\Delta_{\R^2} = \ddp{^2}{x_1^2} + \ddp{^2}{x_2^2}$ is the Laplacian on $\R^2$.

\begin{theorem} \label{thm:highbeta-Z2}
  Let $J \subset \Z^2 \setminus \{0\}$ be any   finite-range step distribution
  that is invariant under lattice rotations and reflections and includes the nearest-neighbour edges.
  Then there exists $\beta_0(J) \in (0,\infty)$ such that for 
  the infinite-volume Discrete Gaussian Model $\avg{\cdot}^{\Z^2}_{J,\beta}$
  at temperature $\beta \geq \beta_0(J)$,
  there is $\beta_{\rm eff}(J,\beta) =\beta+O_J(e^{-c\beta})  \in (0,\infty)$
  with a universal constant $c>0$ (independent of $J$)
  such that for any $f \in C_c^\infty(\R^2)$ with $\int f \, dx =0$
  and $f_\epsilon$ as in  \eqref{eq:feps_def}, as $\epsilon\to 0$,
  \begin{equation} \label{e:highbeta-convergence-Z2-thm}
    \log \avgb{e^{(f_\epsilon,\sigma)_{\Z^2}}}_{J,\beta}^{\Z^2}
    \to
    \frac{\beta_{\rm eff}(J,\beta)}{2 v_J^2}
    (f,(-\Delta_{\R^2})^{-1}f)_{\R^2}.
  \end{equation}
\end{theorem}

Theorem~\ref{thm:highbeta-Z2} superficially resembles \cite[Theorem~\ref{thm:highbeta}]{dgauss1}, but we emphasise
that we are now considering the infinite-volume state; correspondingly the covariance on the right-hand side is now
$(-\Delta_{\R^2})^{-1}$ instead of $(-\Delta_{\T^2})^{-1}$.
The comparison  below \cite[Theorem~\ref{thm:highbeta}]{dgauss1} with previous results for the Discrete Gaussian model
however also applies to the infinite-volume version, i.e., to Theorem~\ref{thm:highbeta-Z2} of this paper.

Theorem~\ref{thm:highbeta-Z2} can be seen as an analogue
for the Discrete Gaussian model (with $\beta \geq \beta_0(J)$)
of the Naddaf--Spencer theorem \cite{MR1461951} which applies to strictly convex smooth 
gradient models.
In our first paper \cite{dgauss1} we discuss many further references concerning such
models and concerning discrete height functions, and we refer to \cite{dgauss1} for a more detailed discussion
and only list here the most relevant references.
For the Discrete Gaussian and XY models, we of course mention the fundamental work of Fr\"ohlich--Spencer \cite{MR634447,MR733469}
as well as the more recent articles \cite{1711.04720,2002.12284,1907.08868,2012.01400,MR4367953,2101.05139,2110.09498,2110.09465}.
For smooth gradient models, there is a very comprehensive picture including stochastic dynamics \cite{MR1463032,MR2228384,MR1872740}
and recent developments include
\cite{ %
MR2251117, 
MR4212193,MR3189075,MR4137943,
MR2855536, 
MR4003143,
1909.13325, 
2002.02946,
DR22,
MR4164451,MR3933043}.
For the smooth but nonconvex gradient models we refer to
\cite{MR2778801,MR2322690,MR2470934,MR2976565} and in particular
\cite{MR1048698} and \cite{1910.13564} which use the renormalisation group approach.
For other discrete height functions, recent works include
\cite{MR2523460,MR4121614,MR3606736,1909.03436,1911.00092,2012.13750,MR4315657}.
Our first paper (and therefore this paper as well) relies in important ways on ideas developed in
\cite{MR2523458,MR1101688,MR1777310,1311.2237,MR2917175}.

As a byproduct of the proof of Theorem~\ref{thm:highbeta-Z2} we also obtain the following mesoscopic
scaling limit for the Discrete Gaussian model on the torus.
(Effective error bounds also follow from the proof.)

\begin{theorem} \label{thm:highbeta-meso}
  Under the same assumptions as in Theorem~\ref{thm:highbeta-Z2},
  there exists $L = L(J)$ such that
  for the Discrete Gaussian model on the torus $\Lambda_N$  of side length $L^N$,
  for any $f \in C_c^\infty(\R^2)$ with $\int f \, dx =0$, $f_\epsilon$ as in \eqref{eq:feps_def},
  and any sequence $\epsilon_N>0$ such that $\epsilon_N\to 0$  as $N\to\infty$
  while $\epsilon_NL^{N} \to \infty$,
  \begin{equation} \label{e:highbeta-convergence-meso}
    \log \avgb{e^{(f_{\epsilon_N},\sigma)}}_{J,\beta}^{\Lambda_N}
    \to
    \frac{\beta_{\rm eff}(J,\beta)}{2 v_J^2}
    (f,(-\Delta_{\R^2})^{-1}f)_{\R^2}, \quad \text{ as } N \to \infty.
  \end{equation}
\end{theorem}

Note that the assumption $\epsilon_N L^N\to\infty$ is necessary. Indeed,
if $\epsilon_N \ll L^{-N}$ then the support of $f_\epsilon$ is not a subset of $\Lambda_N$.
Moreover, if $\epsilon_N=L^{-N}$ the limit would correspond to the macroscopic scaling limit
considered in \cite[Theorem~\ref{thm:highbeta}]{dgauss1} which
is different from the right-hand side above (given in terms of $(-\Delta_{\T^2})^{-1}$
rather than $(-\Delta_{\R^2})^{-1}$).

For some of the related open questions, we refer to our discussion in \cite[Section~\ref{sec:main_results}]{dgauss1},
but mention in addition that a characterisation of the gradient Gibbs measures with finite range $J$
as in  \cite[Theorem~9.1.1]{MR2251117} for the nearest-neighbour case would be interesting.

\subsection{Outline of the paper}

This paper relies heavily on our first article on the Discrete Gaussian model \cite{dgauss1},
and in particular we use the set-up and notation from Section~\ref{sec:first_step} and Sections~\ref{sec:scales}--\ref{sec:inequalities} of that paper.
Even though we included some reminders below, we will often refer to \cite{dgauss1} to avoid repetitiveness.

The proofs of Theorems~\ref{thm:highbeta-Z2} and \ref{thm:highbeta-meso} proceed by decomposing the external field from
the moment-generating function into contributions from all scales, with each contribution smooth at the respective scale.
This is set up in Section~\ref{sec:scale_dep_ext_field}. 
Then, the main technical contribution of the present paper compared to \cite{dgauss1} is an extension of the renormalisation group map,
originally defined in \cite[Section~\ref{sec:rg_generic_step}]{dgauss1}, to allow for a scale-dependent external field.
This is carried out in Section~\ref{sec:rg_generic_step_external_field}, after technical preparation
in the preceding sections.

Different methods to extend a renormalisation group flow by observables for
pointwise correlation functions in similar setups to ours were considered in
\cite{MR3345374,MR3332942,1311.2237,MR1156405}.
These approaches do not allow to derive the infinite-volume scaling limit as in our main result,
and we expect that the approach we develop here could have applications to other models.

\subsection{Notation}
\label{sec:notation-Z2}

We use the notation $|a| \leq O(|b|)$ or $a=O(b)$ to denote $|a|\leq C|b|$ for an absolute constant $C>0$
and $a\sim b$ to denote that $\lim a/b =1$ (where the limit is clear from the context). We stress that  all constants appearing below are uniform in $\beta$ 
unless explicitly stated.

Throughout the paper, the dimension will be $d=2$, but we sometimes write $d$ to emphasise the source of the constant $2$.
Let $e_1, \dots, e_d$ be the basis of unit vectors with nonnegative components  spanning $\Z^d$ or the local coordinates of $\Lambda$,
and set $\hat{e} = \{\pm e_1, \cdots, \pm e_d \}$.
For a function $f: \Z^d \to \C$ or $f: \Lambda_N \to \C$, we write
$\nabla^{\mu} f(x) = f(x + \mu) - f(x)$ for $\mu \in \hat{e}$.
For any multi-index $\alpha \in \{ \pm 1, \cdots, \pm d\}^n$ with $n =|\alpha| \geq 1$, 
we write $\nabla^{\alpha} f = \nabla^{e_{\alpha_1}}\cdots  \nabla^{e_{\alpha_n}} f$.
The vector of $n$-th order discrete partial derivatives are denoted by 
\begin{equation} \label{eq:nabla-n}
\nabla^n f (x) = (\nabla^{\mu_1} \cdots \nabla^{\mu_n} f (x) : \mu_k \in \hat{e} \text{ for all } k),
\end{equation}
and we write $|\nabla ^nf(x)|$ for the maximum over all of its components.
$\Delta$ without subscript denotes the \emph{unnormalised} nearest-neighbour Laplacian,
\begin{equation}
  \Delta f (x) = \sum_{\mu \in \hat{e}} (f(x+ \mu) - f(x))
  = \sum_{\mu \in \hat{e}} \nabla^\mu f(x)
  = \frac12 \sum_{\mu \in \hat{e}} \nabla^\mu \nabla^{-\mu} f(x),
\end{equation}
whereas $\Delta_J$ denotes the \emph{normalised} Laplacian \eqref{eq:Delta_J_definition-dgauss2} with finite-range step distribution $J$.

\section{Scale-dependent external fields}
\label{sec:scale_dep_ext_field}

In this section, after briefly reviewing some aspects from the setup of our first paper \cite{dgauss1}, we proceed to describe how the proofs of the above theorems follow by amending the renormalisation group flow constructed in \cite{dgauss1} by suitable external fields $u=(u_j)$, which start to appear at a characteristic scale $j=j_f$ in the renormalisation.
We then proceed, assuming these fields $u$ to have a negligible overall effect, as expressed in Theorem~\ref{thm:Z_N_ratio_conclusion} below, to conclude the proofs of Theorems~\ref{thm:highbeta-Z2} and~\ref{thm:highbeta-meso}. The remaining sections will be geared towards the proof of Theorem~\ref{thm:Z_N_ratio_conclusion}, which appears in Section~\ref{sec:proof_Z_N_ratio}.

\subsection{Multiscale decomposition of the field}
We first briefly review a few key aspects from the setup of our previous paper \cite{dgauss1}, which will prevail here.
As in Section~\ref{sec:intro-dgauss2}, we denote by $\Lambda_N$ the discrete torus of side length $L^N$ and we will later impose that $L$ is sufficiently large, see the discussion at the end of Section~\ref{sec:parameters} for details; the infinite volume limit will then correspond to the limit $N\to\infty$.
As explained in \cite[Section~\ref{sec:first_step}]{dgauss1}, it is convenient to work
with the mass-regularised Discrete Gaussian model
$\avg{\cdot}_{\beta,m^2}$ and take $m^2\downarrow 0$ in the end. This is the probability measure $\avg{\cdot}_{\beta,m^2} \equiv \avg{\cdot}_{\beta,m^2}^{\Lambda_N}$ obtained by replacing $-\Delta_J$ by $-\Delta_J+m^2$ in \eqref{eq:DG_model_1_external_field} and $\Omega^{\Lambda_N}$ by $\Z_\beta^{\Lambda_N}$ where $\Z_\beta = 2\pi\beta^{-1/2}\Z$, i.e., dropping the constraint $\sigma_0=0$. By \cite[Lemma~\ref{lemma:m2to0}]{dgauss1}, then
\begin{equation} \label{eq:mass-limit}
\langle F(\sigma) \rangle_{\beta} = \lim_{m^2 \downarrow 0} \langle F(\sigma) \rangle_{\beta, m^2},
\end{equation}
for any $F$ as appearing above \eqref{eq:DG_model_1_external_field} (and in particular for the choice $F(\sigma)=e^{(f,\sigma)}$ for any $f: \Lambda_N\to \R$).

The renormalisation group analysis will involve a decomposition of the covariance
\begin{equation} \label{eq:C(s)}
 	C(s,m^2)\stackrel{\text{def.}}{=}
 	(C(m^2)^{-1}-s\Delta)^{-1}, \qquad \text{with} \quad C(m^2) = (-\Delta_J+m^2)^{-1} - \gamma\id,
\end{equation}
where the inverses are interpreted on $\R^{\Lambda_N}$ and $\Delta$ is the (unnormalised) nearest-neighbour Laplacian on $\Lambda_N$, and $\gamma$ and $s$ are parameters with
$\gamma \in (0,\frac13)$ and $|s|$ tacitly assumed sufficiently small so that $C(m^2)^{-1}-s\Delta$ is positive definite.
As in \cite[\eqref{eq:frd_ulbds}]{dgauss1}, and without loss of generality, we work from here on under the standing assumptions that $|s| \leq \epsilon_s \theta_J$ (by which \eqref{eq:C(s)} is well-defined) and that, for an arbitrary constant $C>0$, we have $\theta_J \geq C^{-1}$ and $ v_J \geq C^{-1} \rho_J$, where $\theta_J$ and $\rho_J$ refer to the range and spectral characteristics of $J$, defined in \cite[\eqref{eq:rJ_range}, \eqref{eq:theta_def}]{dgauss1}, and $\varepsilon_s$ is the numerical constant appearing in \cite[Proposition~\ref{prop:decomp_compatible}]{dgauss1}. The last two conditions hold for any fixed $J$ as in the theorems. (The use of the constant $C$ will yield uniform estimates over families of $J$
  as above, see \cite[Remark~\ref{rk:highbeta-rho}]{dgauss1}. We do not state these in our main theorems above, but still introduce $C$ to follow the same setup as in \cite{dgauss1}).

Under these assumptions, it follows that for suitable choice of $\gamma \in (0,\frac13)$, which we henceforth regard as fixed,
one can decompose $C(s,m^2)$ from \eqref{eq:C(s)} as in \cite[Section~\ref{sec:scales}]{dgauss1} (see in particular \eqref{eq:frd_of_C^Lambda_N} therein) to obtain, for all $m^2 >0$ (and $|s| \leq \epsilon_s \theta_J$),
\begin{equation} \label{e:Cdecomp-Z2'}
  C(s,m^2)
  = 
    \Gamma_1(s,m^2)+\dots+\Gamma_{N-1}(s,m^2)+\Gamma_N^{\Lambda_N}(s,m^2)+t_N(s,m^2) Q_N.
\end{equation}
The right-hand side is a sum over positive (semi-)definite (covariance) matrices indexed by $\Lambda_N$. The matrix $Q_N$ has all entries  equal to $1/|\Lambda_N|=L^{-dN}$ and $t_N(s,m^2)$ is a scalar satisfying \cite[\eqref{eq:t_N_bound}]{dgauss1}, in particular, diverging like $m^{-2}$ as $m^2 \downarrow 0$. The covariances $\Gamma_{j+1}$ and $\Gamma_{N}^{\Lambda_N}$ in \eqref{e:Cdecomp-Z2'} refer to those defined in \cite[(4.2), (4.3)]{dgauss1}. They correspond to a decomposition over scales $L^j$ of the covariance $C(s,m^2)$. By construction, the matrices $\Gamma_j$ have range $\frac14 L^{j}$ and their key analytical features are summarised in \cite[Lemma~\ref{cor:Gammaj}]{dgauss1}.
We will frequently use the following notation. For $f: \Lambda_N \to \R$, we define (with a slight abuse of notation) $ \Gamma_{j} (f) =\Gamma_j \ast f$ where $(\Gamma_j \ast f) (x) =  \sum_y \Gamma_{j}(x-y) f(y)$ with $\Gamma_{j}(x)= \Gamma_{j}(0,x)$, cf.~\cite[below \eqref{eq:Dt_range}]{dgauss1}.

This completes the introduction of our setup. We observe that in fact, the parameter $s$ in \eqref{eq:C(s)}, which implements the renormalisation of the temperature of the model, can be fixed from the start in the present paper as $s=s_0^c(J,\beta)$ with the latter as defined in \cite[Proposition~\ref{prop:stable_manifold}]{dgauss1}; we will return to this later. 

In what follows, we write $\E_{\Gamma}$ denotes the expectation of a Gaussian field $\zeta$ with covariance $\Gamma$. We will frequently write $\Eplus$ 
for $\E_{\Gamma_{j+1}}$ when $j=1,\dots, N-2$ and $\Eplus$ for $\E_{\Gamma_{N}^{\Lambda_N}}$ when $j=N-1$, whenever the scale $j$ is clear from the context. Since $\Gamma_N^{\Lambda_N}$ satisfies exactly the same upper bounds as $\Gamma_j$ with $j=N$, we will usually not distinguish between the cases $j+1<N$ and $j+1=N$.
Generally, $j$ without further specification is allowed to take values $j=1,\dots, N-1$.

\subsection{Strategy}
\label{sec:strategy}

Contrary to the macroscopic torus scaling limit in \cite{dgauss1}, in which all the scales $j  < N$ appearing in \eqref{e:Cdecomp-Z2'} were treated equally, we will have to distinguish in what follows a characteristic scale $j_f$ at which a given test function $f$ starts to induce a `perturbation,' cf.~\eqref{eq:extfield_def} below, which manifests itself as a shift (or translation) of the corresponding Gaussian field (at the same scale).
This is because the infinite volume limit $N\rightarrow \infty$ in Theorem~\ref{thm:highbeta-Z2} is decoupled from the characteristic scale $j_f$, 
whereas \cite{dgauss1} simply takes $j_f = N$. 
The induced perturbation influences the renormalisation group flow in all the larger scales $j \geq j_f$.
The technical difficulties arising in this paper are due to these changes. 
Fortunately, it will turn out that the infinite chain of perturbations will only impact the analysis on a bounded region by the compact support condition on the external field (see Lemma~\ref{lem:extfield_bd}, for example).

% {\grey
% Let $M>0$ be a large parameter and 
% }
% \commentjp{We have removed $M >0$, so we should remove $M$-dependence from everywhere.}

Let $f : \Z^2 \to \R$ be a finitely
supported test function with $\sum_x f(x)=0$.
Let $j_f$ be the smallest integer $(\geq 1)$ such that the support of $f$ and $\Delta f$ is contained in $[0,\frac14 L^{j_f})^2$ up to a spatial translation. 
If $f : \Lambda_N \to \R$ then $j_f$ is defined similarly by identifying $\Lambda_N$ with $([0,L^N)\cap \Z)^2 \subset \Z^2$, whence $j_f \leq N$.
We call $j_f$ the smoothness scale of $f$ and will frequently assume that
\begin{equation} \label{eq:f_bd_ass}
  \|f\|_{\ell^{\infty} (\Z^2)}  \leq c L^{-2j_f},
%    \|f\|_{\ell^{\infty} (\Z^2)} := \max_{n=0,1,2} L^{nj_f}\|\nabla^n f\|_{L^\infty(\Z^2)} \leq c L^{-2j_f},
\end{equation}
%where $\nabla^n$ is the vector of $n$-th discrete gradients of $f$,
%see \eqref{eq:nabla-n}, 
%and 
where $c$ will be an $L$-dependent small constant fixed below Lemma~\ref{lem:extfield_bd}.
The interpretation of $j_f$ as a smoothness scale becomes clear when we focus on lattice functions scaled like $f_{\epsilon}$ given by \eqref{eq:feps_def}.
Indeed,  each $\epsilon^{-2} f_{\epsilon} (\epsilon^{-1} x)$ is an approximation of a smooth function, thus $j_{f_\epsilon}$ is the scale where $f_{\epsilon}$ becomes smooth: $L^{-j_{f_{\epsilon}}} \simeq \epsilon^{-1}$.

The macroscopic scaling limit considered in \cite{dgauss1} corresponds to $j_f=N$, but now
we are interested in $j_f \ll N$.
The analysis of the macroscopic scaling limit
proceeded through a translation of the field
by $\gamma f + C(s,m^2) (f+s\gamma \Delta f )$ at scale $N$, with $ C(s,m^2)$ as given by \eqref{eq:C(s)}. 
The term $\gamma f$
and the difference between $f$ and $f+s\gamma \Delta f $ will be insignificant
and result from the preliminary renormalisation group step in \cite[Section~\ref{sec:first_step}.3]{dgauss1}, which integrates out the i.i.d.~field with variance $\gamma$, cf.~\eqref{eq:C(s)}, thus transforming the original discrete field into a smooth periodic potential (integrated with respect to a Gaussian measure).
In view of \eqref{e:Cdecomp-Z2'}, we now rewrite $C(s,m^2)$ as
\begin{equation} \label{e:Cdecomp-Z2}
  C(s,m^2)
  = 
    \Gamma_{\leq j_{f}}(s,m^2)+\sum_{j=j_f+1}^{N-1} \Gamma_{j}(s,m^2)+\Gamma_N^{\Lambda_N}(s,m^2)+t_N(s,m^2) Q_N,
\end{equation}
where, with hopefully obvious notation, $ \Gamma_{\leq j} = \sum_{1\leq k \leq j} \Gamma_k$. Our starting point in this paper for the proofs of Theorems~\ref{thm:highbeta-Z2} and \ref{thm:highbeta-meso}
is also a translation, but at the smoothness scale $j_f$ rather than the macroscopic scale $N$, and by $\gamma f + \Gamma_{\leq j_f} (f+s\gamma \Delta f )$, see Lemma~\ref{lemma:reformulation-Z2} below.
An observation (made precise by Lemma~\ref{lem:extfield_bd} below) is that $\gamma f + \Gamma_{\leq j_f} (f+s\gamma \Delta f )$ is smooth at scale $j_f$ because $f$ is,
while on the other hand,  $\Gamma_k (f+s\gamma \Delta f )$ is smooth for $k>j_f$ because of the smoothing properties of
the covariance $\Gamma_k$. We will show that this allows to implement translations iteratively
for all scales $k\geq j_f$,
with small errors accumulating from each scale $k$ starting from $k=j_f$ and that as $j_f\to\infty$ the sum of these errors
is governed  by the contribution from the scale $j_f$ and tends to $0$ as $j_f\to\infty$.

\subsection{Scale-dependent external fields}

To formulate the above strategy more precisely,
first  recall (as mentioned above) that the parameter $s$ is fixed as $s=s_0^c(J,\beta)$ from the start of this paper.
Further let $s_0=s=s_0^c(J,\beta)$, and define (as in \cite[\eqref{eq:Z_0_definition}]{dgauss1})
\begin{equation}
  Z_0(\varphi) = e^{\frac{s_0}{2}(\varphi,-\Delta\varphi)+\sum_{x\in \Lambda_N} \tilde U(\varphi_x)}, \quad \varphi \in \R^{\Lambda_N},
  \label{eq:Z_0_definition-bis}
\end{equation}
with the function $\tilde{U}$ given by \cite[\eqref{e:tildeF}]{dgauss1}, which is a $2\pi \beta^{-1/2}$ periodic function of a single real variable. The next lemma is a slight reformulation of \cite[Lemma~\ref{lemma:reformulation}]{dgauss1}.
For its statement let $\tilde C(s,m^2)$ be given as in \cite[\eqref{e:tildeC}]{dgauss1}, i.e.,
\begin{equation}\label{eq:def_Ctilde}
  \tilde C(s, m^2) = \gamma(1 + s\gamma\Delta) + (1 + s\gamma \Delta) C(s, m^2) (1 + s\gamma \Delta),
\end{equation}
and recall the covariance decomposition \eqref{e:Cdecomp-Z2}.  
\begin{lemma}
    \label{lemma:reformulation-Z2}
    For all $\beta>0$,
    $\gamma\in (0,\frac13)$, $m^2 \in (0,1]$, 
    $|s|=|s_0|$ small, 
    one has for any $f\in \R^{\Lambda_N}$ such that $\sum_x f(x) =0$,
  \begin{equation} \label{e:reformulation-Z2}
    \big\langle e^{(f, \sigma)} \big\rangle^{\Lambda_N}_{\beta, m^2}
    \propto e^{\frac12 (f, \tilde C(s, m^2) f)} \E_{C(s,m^2)} \big[ Z_0(\varphi+\textstyle\sum_{j=j_f}^N u_j) \big],
  \end{equation}
  where the expectation acts on $\varphi$ and
    \begin{align} \label{eq:extfield_def}
    u_j  & =
    \begin{cases}
      0 & (j<j_f)\\
      \gamma f+ \Gamma_{\leq j_f} (f+s\gamma\Delta f) & (j=j_f)\\
      \Gamma_{j}(f+s\gamma\Delta f) & (N>j>j_f)\\
      \Gamma_N^{\Lambda_N}(f+s\gamma\Delta f) & (j=N).
    \end{cases}
  \end{align}
\end{lemma}

\begin{proof}
  By \cite[Lemma~\ref{lemma:reformulation}]{dgauss1},
  \begin{equation} \label{e:reformulation-Z2'}
    \sum_{\sigma \in \Z_{\beta}^{\Lambda_N}} e^{-\frac12(\sigma,(-\Delta_J + m^2)\sigma)} e^{(f,\sigma)}
    \propto e^{\frac12 (f, \tilde C(s, m^2) f)} \E_{C(s,m^2)} \big[ Z_0(\varphi+Af) \big],
  \end{equation}
  with
  \begin{equation}
    A= (1+s\gamma \Delta)^{-1} \tilde C(s,m^2)
    =  \gamma + C(s,m^2) (1 + s\gamma\Delta).
  \end{equation}
  The statement follows by applying the decomposition \eqref{e:Cdecomp-Z2} of $C(s,m^2)$ which gives
\begin{align}
Af = \sum_{j\leq N} u_j + t_NQ_N (f+s\gamma\Delta f) = \sum_{j \leq N} u_j,
\end{align}  
 where the last equality follows because $\sum_x f(x) = 0$, and hence $Q_N f = Q_N \Delta f = 0$.
\end{proof}

The renormalisation group flow constructed in \cite{dgauss1}, which we now sometimes refer to as the bulk renormalisation group flow,
is in terms of the recursion (cf.~\cite[\eqref{eq:general_RG_step1}]{dgauss1})
\begin{equation} \label{eq:dgauss1_RG_flow}
  Z_{j+1}(\varphi') = \E_{\Gamma_{j+1}}Z_j(\varphi'+\zeta), \quad \varphi' \in \R^{\Lambda_N},
\end{equation}
where here and below, $\E_{\Gamma_{j+1}}$ is the Gaussian expectation with covariance $\Gamma_{j+1}$
which always acts on the field $\zeta$.
To incorporate the scale-dependent external fields $u=(u_j)$ we now
define $Z_0(u,\varphi)=Z_0(\varphi)$ and
\begin{equation} \label{eq:Zplusu_def}
  Z_{j+1}(u,\varphi')
  = \E_{\Gamma_{j+1}} Z_j(u, \varphi'+\zeta+u_{j}), \quad \varphi' \in \R^{\Lambda_N},
\end{equation}
with $\Gamma_N^{\Lambda_N}$ instead of $\Gamma_{j+1}$ when $j+1=N$. 
Finally set
\begin{equation} \label{eq:Z^tilde_definition}
\tilde Z_N(u,\varphi') = \E_{t_N Q_N} Z_N(u,\varphi'+\zeta + u_N ), \quad \varphi' \in \R^{\Lambda_N}.
\end{equation}
Together, \eqref{eq:Zplusu_def}, \eqref{eq:Z^tilde_definition} and \eqref{e:Cdecomp-Z2'} imply in particular that the expectation appearing on the right-hand side of \eqref{e:reformulation-Z2} can be recast as (with $\E_{C(s,m^2)}$ acting on $\varphi$)
\begin{equation}
\label{eq:expectation-rewrite}
\E_{C(s,m^2)} \big[ Z_0(\varphi+ \textstyle\sum_{j=j_f}^N u_j) \big] = \tilde Z_N(u,0).
\end{equation}
Our analysis of the $Z_j(u,\varphi')$ relies on the property that the external fields $u_{j}$ are smooth on scale $j$ for all $j$, as demonstrated by the next lemma.
Here  assume that $j_f$ in \eqref{eq:extfield_def} is the smoothness scale of $f$, i.e.,
the smallest integer such that $\supp f$ is contained in a block of side length $\frac14 L^{j_f}$.
By definition, a block of size $L$ is any set of the form $x + ([0,L) \cap \Z)^2$ for some $x \in L\Z^2$.
Let $\norm{u_j}_{C^2_j} = \norm{u_j}_{C^2_j(\Z^2)} = \max_{n=0,1,2} \norm{\nabla^n_j u_j}_{\ell^{\infty}(\Z^2)}$, cf.~\cite[\eqref{eq:normC2}]{dgauss1}. 
In the sequel we often tacitly view a function $f$ with domain $\Lambda_N$ (such as $u_j$) as defined on $\Z^2$ by identifying $\Lambda_N$ with $[0,L^N)^2$ and extending $f$ to have value $0$ outside this set.

%\commentjp{I think this lemma is false, and is related a more serious issue. }

\begin{lemma} \label{lem:extfield_bd}
  There exists an $L$-independent constant $C >0$ 
   such that the following holds:
  for all $f:\Z^d\to\R$ satisfying $\sum f=0$
%  , $\norm{f}_{C^2_{j_f}} \leq c$ 
  and such that 
   $f$ and $\Delta f$ have
  support in a block of side length $\frac14 L^{j_f}$, the functions
  $u_j$ defined by \eqref{eq:extfield_def}
  have support in blocks of side lengths $\frac34 L^{j}$ for $j\leq N-1$ and 
  \begin{equation} \label{eq:extfield_bd}
    \|u_j\|_{C_j^2} \leq C L^{2j_f + 2} \norm{f}_{\ell^{\infty} (\Z^2)} ,  \qquad j \leq N.
  \end{equation}
\end{lemma}
%\begin{lemma} \label{lem:extfield_bd}
%  There exists $C>0$ such that the following holds:
%  for all $f:\Z^d\to\R$ satisfying $\sum f=0$ and such that 
%   $f$ and $\Delta f$ have
%  support in a block of side length $\frac14 L^{j_f}$, the functions
%  $u_j$ defined by \eqref{eq:extfield_def}
%  have support in blocks of side lengths $\frac34 L^{j}$ for $j\leq N-1$ and 
%  \begin{equation} \label{eq:extfield_bd}
%    \|u_j\|_{C_j^2} \leq C L^{2j_f} \norm{f}_{C_{j_f}^2}, \qquad j \leq N.
%  \end{equation}
%\end{lemma}

In particular,  if \eqref{eq:f_bd_ass} holds with $c \leq (C L^{2})^{-1}$,  then $\sup_{j}\|u_j\|_{C_j^2} \leq 1$. 
From here on, we fix (any) such value of $c$; this choice is implicit when referring to \eqref{eq:f_bd_ass} in the sequel.

\begin{proof}
Let $g=f+s\gamma\Delta f$ and note that by assumption $g$ has support in a block of side length $\frac14 L^{j_f}$.
Also, $\|g\|_{\ell^{\infty}} \leq (1+2|s|\gamma) \|f\|_{\ell^{\infty}} \leq C\|f\|_{\ell^{\infty}}$ 
since $\|\Delta f\|_{\ell^{\infty}} \leq 8 \|f\|_{\ell^{\infty}}$ for any $j$.
%, and analogously $\|g\|_{L^\infty} \leq C\|f\|_{L^\infty}$.
We may identify $\Gamma_j$ with its convolution kernel, i.e., $\Gamma_j g = \Gamma_j*g$.
Then
$\Gamma_j$ is supported in a block of side length $\frac14 L^{j}$
and satisfies 
$\|\nabla^\alpha_j \Gamma_j\|_{\ell^\infty} \leq C L^2$ for $|\alpha| \leq 2$
where $\nabla^\alpha_j = L^{j|\alpha|} \nabla^\alpha$, see \cite[Corollary~\ref{cor:Gammaj}]{dgauss1}, thus
 \begin{equation}
     \|\nabla_j^\alpha \Gamma_j g \|_{\ell^\infty}
     \leq L^{2j_f} \|\nabla_j^\alpha \Gamma_j\|_{\ell^\infty} \|g\|_{\ell^\infty}
     \leq C L^{2j_f + 2} \|f\|_{\ell^\infty}.
 \end{equation}
Thus the desired statement holds if $j < N$. 

The same estimates hold when $j=N$, i.e., with $\Gamma_j$ is replaced by $\Gamma_N^{\Lambda_N}$
which satisfies analogous bounds, see \cite[Corollary \ref{cor:Gammaj}]{dgauss1}.
This completes the proof of the bound \eqref{eq:extfield_bd}.

The statement about
the support of the $u_j$ follows immediately from the assumption
that the support of $f$ and $g$ 
have diameter $\frac14 L^{j_f} \leq \frac14 L^j$ for all $j \geq j_f$
and that $\Gamma_j$ has range $\frac14 L^j$.
\end{proof}

\subsection{Conclusion of the argument}

In Section~\ref{sec:proof_Z_N_ratio} we will show the following theorem
from which the proof of Theorem~\ref{thm:highbeta-Z2} can be completed
similarly as the torus result in \cite[Section~\ref{sec:integration_of_zero_mode}]{dgauss1}.
The theorem is stated under somewhat more general condition on 
the sequence $(u_j)_j = (u_j \in \R^{\Lambda} )_j$ of given external fields that are uniformly bounded and supported on a single block in the sense that:
\begin{quote}
\begin{itemize}
\item[\customlabel{assump:u}{$\assumpu$}]
There exists $j_u$
  such that $u_j =0$ for $j <  j_u$, 
  $\norm{u_j}_{C_j^2} \leq 1$ for each $j\leq N$,
  and $u_j$ is supported on the unique $B_0 \in \cB_j$ such that $0\in B_0$ and $d(\partial B_0, \supp (u_j)) > 4$.
\end{itemize}
\end{quote}
For the same reason that $j_{f}$ was called a smoothness scale of $f$,
we call $j_u$ the smoothness scale (of $u = (u_j)_j$).
Note that,  by translation invariance of the Discrete Gaussian model on the torus $\Lambda_N$,
we may assume that $f$ is centred with respect to the block decomposition;
that is, $\text{supp}(f)$ and $\text{supp}(\Delta f)$ are contained in the box $m+[0,\frac14L_{j_f})^2$, where $m$ is one of the lattice points closest to the center of some block $B \in \cB_j$ for all $j_f \leq j \leq N$.
In particular, then, by Lemma~\ref{lem:extfield_bd}, for all scales $j \leq N$,
there is a block $B\in \cB_j$ such that whenever $L \geq C$, $N_{ 5}(\supp(u_j)) \subset B$
where $N_k(X)$ denotes the set of points with $\ell^1$-distance at most $k$ from the set $X$. Thus the condition on the support of $u_j$ is not stronger than the condition on the support of $f$.

\begin{theorem} \label{thm:Z_N_ratio_conclusion}
  Let $J$ be a finite-range step distribution as in the statements of Theorems~\ref{thm:highbeta-Z2} and~\ref{thm:highbeta-meso}.
  There are $\beta_0(J) \in (0,\infty)$, a (large) integer $L=L(J)$ (which can be chosen dyadic),
  and a constant $\alpha>0$ such that if 
  $u=(u_j)$ satisfies \ref{assump:u}
  there is $C >0$ such that  for $\beta \geq \beta_0(J)$ and $N > j_u$,
  \begin{equation}
   \absa{ \frac{\tilde Z_N(u,0)}{\tilde Z_N(0,0)} -1 } \leq C L^{-\alpha j_u} .
  \end{equation}
\end{theorem}

%\begin{theorem} \label{thm:Z_N_ratio_conclusion}
%  Let $J$ be a finite-range step distribution as in the statements of Theorems~\ref{thm:highbeta-Z2} and~\ref{thm:highbeta-meso}.
%  There are $\beta_0(J) \in (0,\infty)$, a (large) integer $L=L(J)$ (which can be chosen dyadic),
%  and a constant $\alpha>0$ such that if $M>0$ and $f$ satisfies \eqref{eq:f_bd_ass} with this $M$, $\sum f=0$,
%  and $u=(u_j)$ is as in \eqref{eq:extfield_def},
%  there is $C(M)>0$ such that  for $\beta \geq \beta_0(J)$ and $N > j_f$,
%  \begin{equation}
%   \absa{ \frac{\tilde Z_N(u,0)}{\tilde Z_N(0,0)} -1 } \leq C( M ) L^{-\alpha j_f} .
%  \end{equation}
%\end{theorem}

Assuming Theorem~\ref{thm:Z_N_ratio_conclusion} to hold, and in view of Lemma~\ref{lemma:reformulation-Z2}, the proofs of
Theorems~\ref{thm:highbeta-Z2} and~\ref{thm:highbeta-meso} are readily completed by  means of the following elementary lemma, as explained below.
This lemma is the infinite-volume analogue of \cite[Lemma~\ref{lemma:integral_of_zero_mode}]{dgauss1};
we postpone its proof to the end of this section and first give the details for the proof of
Theorems~\ref{thm:highbeta-Z2} and~\ref{thm:highbeta-meso}.
In what follows, for $N > j_{f_{\varepsilon}}$, we tacitly identify $f_{\varepsilon}$ with the corresponding function having domain on the torus $\Lambda_N$ by identifying $\text{supp}(f_{\varepsilon})$ with a suitable subset of the torus $\Lambda_N$. We write $ \tilde{C}^{\Lambda_N} \equiv  \tilde {C}$ for the covariance matrix defined in \eqref{eq:def_Ctilde} to stress the dependence on the underlying torus $\Lambda_N$.

\begin{lemma} \label{lemma:discrete_approx_of_Lap^-1}
  Let $f\in C_c^\infty(\R^2)$ with $\int f \, dx =0$ and $f_\epsilon$ be as in \eqref{eq:feps_def}.
  Then
\begin{equation}
  \lim_{\epsilon\to 0}\lim_{N\to\infty} \lim_{m^2\downarrow 0}(f_\epsilon,\tilde{C}^{\Lambda_N}
  (s,m^2)f_\epsilon)
  = \frac{1}{v_J^2+s} (f,(-\Delta_{\R^2})^{-1}f)_{\R^2},
\end{equation}
and the statement also holds if the two leftmost limits  are replaced by $N\to\infty$
with $\epsilon=\epsilon_N \to 0$ while $\epsilon_N L^N\to\infty$.
\end{lemma}

\begin{proof}[Proof of Theorems~\ref{thm:highbeta-Z2} and \ref{thm:highbeta-meso}]
Our proof proceeds as the following. 
We will first prove our main limit results with $f_{\epsilon}$ replaced by $\tau f_{\epsilon}$ for $\tau >0$ sufficiently small (depending on $C_f$ of \eqref{eq:feps_def} and $c$ of \eqref{eq:f_bd_ass}). 
The convergence can then be extended to all $\tau \in \C$ by a standard argument which we include for completeness.
%Then we will demonstrate using a standard argument that these are actually enough to prove the desired Gaussian scaling limits. 

  Given $f \in C_c^\infty(\R^2)$ with $\int f\, dx=0$
  and $f_\epsilon$ as in \eqref{eq:feps_def}, set $j_f =\lceil \log_L (8C_f \epsilon^{-1})\rceil$.
  Then using the first two conditions in \eqref{eq:feps_def}, 
  it readily follows that 
%  there is a constant $\tau > 0$ depending on 
%  $C_f,L$ such that 
  $f_{\epsilon}$ satisfies \eqref{eq:f_bd_ass} for all $\epsilon \in (0,1)$ (including that $f_{\epsilon}$ is supported on a block of side length $L^{j_f}/4$).
  Now define $u_j= u_j[\epsilon]$ according to \eqref{eq:extfield_def} with $ f_{\epsilon}$ in place of $f$.
  Then by Lemma~\ref{lem:extfield_bd},  $(\tau u_j)_j$ satisfies \ref{assump:u} with $j_u = j_f$ whenever $\tau > 0$ is small enough depending on $C_f$ and $L$.
  Now by Lemma~\ref{lemma:reformulation-Z2} and \eqref{eq:expectation-rewrite},
\begin{equation}
  \langle e^{\tau (f_{\epsilon}, \sigma)} \rangle^{\Lambda_N}_{\beta, m^2} = e^{\frac{\tau^2}{2} (f_{\epsilon}, \tilde{C} (s, m^2) f_{\epsilon})} \frac{\tilde{Z}_N (\tau u[\epsilon], 0)}{\tilde{Z}_N (0,0)}.
\end{equation}
Since \ref{assump:u} holds for $\tau u[\epsilon]$,
the assumption of Theorem~\ref{thm:Z_N_ratio_conclusion} is satisfied
uniformly in $\epsilon$. Therefore
\begin{align}
\frac{\tilde{Z}_N (\tau u [\epsilon], 0)}{\tilde{Z}_N (0, 0)}
  = 1+ O_f(e^{-\alpha \log (8 C_f \epsilon^{-1}) } )
  = 1+ O_f(\epsilon^{\alpha})
\end{align}
uniformly in $m^2, \epsilon$ and $j_f < N$.
In the context of Theorem 1.1, the last condition $j_f<N$ is immediate as soon as $N \geq C(\varepsilon)$ since $\varepsilon > 0$ is fixed while $N \to \infty$;
in the context of Theorem 1.2, it follows from our assumption $\varepsilon_NL^N \to \infty$.
Finally, by Lemma~\ref{lemma:discrete_approx_of_Lap^-1},
if either first $N\to\infty$ and then $\epsilon \to 0$, or if
$\epsilon = \epsilon_N \rightarrow 0$ such that $\epsilon_N L^N \rightarrow \infty$, we have
\begin{equation}
  \lim_{m^2 \downarrow 0} \log \langle e^{\tau (f_{\epsilon}, \sigma)} \rangle^{\Lambda_N}_{\beta, m^2} \rightarrow \tau^2 \frac{\betaeff (J, \beta)}{2v_J^2} (f, (-\Delta_{\R^2})^{-1} f )_{\R^2} .
\end{equation}
By \eqref{eq:mass-limit}, using that $\sum f_\epsilon =0$,
the left-hand side equals
$\log\avg{e^{\tau (f_{\epsilon},\sigma)}}_{J,\beta}^{\Lambda_N}$.
Thus with $f_{\epsilon}$ replaced by $\tau f_{\epsilon}$ with sufficiently small $\tau >0$,  
the proof of Theorem~\ref{thm:highbeta-Z2} follows on account of Proposition~\ref{prop:infvollimitDG},
and Theorem~\ref{thm:highbeta-meso} follows directly from the above, 
i.e., 
\begin{align} \label{e:highbeta-convergence-Z2-thm-bis}
    \log \avgb{e^{\tau (f_\epsilon,\sigma)_{\Z^2}}}_{J,\beta}^{\Z^2}
    & \to
    \tau^2 \frac{\beta_{\rm eff}(J,\beta)}{2 v_J^2}
    (f,(-\Delta_{\R^2})^{-1}f)_{\R^2} , 
    	\qquad \text{ as } \epsilon \to 0,
    \\
    \log \avgb{e^{\tau (f_{\epsilon_N},\sigma)}}_{J,\beta}^{\Lambda_N}
    & \to
    \tau^2 \frac{\beta_{\rm eff}(J,\beta)}{2 v_J^2}
    (f,(-\Delta_{\R^2})^{-1}f)_{\R^2}, \qquad \text{ as } N \to \infty.
\end{align}

Now, we show that the domain of $\tau$ can be extended to $\C$ using the Gaussian domination inequality.
Indeed, by \eqref{e:highbeta-convergence-Z2-thm-bis}, we see
\begin{align}
\begin{split}
	\avgb{  (f_\epsilon,\sigma)_{\Z^2}^{2n}  }_{J,\beta}^{\Z^2}
    & \to
    \frac{(2n) !}{2^n n !} \frac{\beta_{\rm eff}(J,\beta)}{2 v_J^2}
    (f,(-\Delta_{\R^2})^{-1}f)_{\R^2}^n \\
    \avgb{  (f_\epsilon,\sigma)_{\Z^2}^{2n +1}  }_{J,\beta}^{\Z^2}
    	&= 0
\end{split}
	\label{eq:poly_moment_bounds}
\end{align}
for each $n\in \N$.
Also,  for any $T \geq 0$,
by the Taylor's theorem (for the second equality),
there exists $\theta \in [0,1]$ such that
\begin{align}
  \sum_{n > k} \frac{T^n}{n !} |(f_\epsilon, \sigma)|^n = e^{T | (f_\epsilon, \sigma) |} - \sum_{n=0}^k \frac{T^n}{n!} | (f_\epsilon, \sigma) |^n 
  & = \frac{e^{\theta T | (f_\epsilon , \sigma) |} }{(k+1) !} 
  \leq 
  \frac{e^{T (f_\epsilon , \sigma)}  + e^{-T (f_\epsilon , \sigma)}}{(k+1) !}
  .
\end{align}
But by \cite{MR496191} (see also \cite[Proposition~1.2]{1711.04720}), we have the Gaussian domination
\begin{equation}
  \label{eq:ginibre_1-bis}
	\avg{ e^{(g,\sigma)}}_{J,\beta}^{\Lambda_N} \leq e^{\frac{\beta}{2} (g,(-\Delta_J)^{-1}_{\Lambda} g)}
\end{equation}
for any $g : \Lambda_N \to\R$ with $\sum g=0$,
so we obtain
\begin{equation}
  \Big| \Big\langle \sum_{n > k} \frac{\tau^n}{n !} (f_\epsilon, \sigma)^n \Big\rangle_{J, \beta}^{\Lambda_N} \Big| \leq \frac{2}{(k+1) !} e^{\frac{\beta}{2} T^2 (f_\epsilon, (-\Delta)^{-1} f_\epsilon)}
\end{equation}
upon letting $T = | \operatorname{Re} (\tau) |$.
In other words, 
$
\langle \sum_{n = 0}^k \frac{\tau^n }{n !} (f_\epsilon, \sigma)^n \rangle_{J, \beta}^{\Lambda_N}
\rightarrow \langle e^{\tau (f_\epsilon, \sigma) } \rangle_{J, \beta}^{\Lambda_N}
$	as $k\rightarrow \infty$,
uniformly in $\epsilon$ and $N$,  proving
\begin{equation}
  \lim_{\epsilon \rightarrow 0, N \rightarrow \infty} \lim_{k\rightarrow \infty} \Big\langle \sum_{n = 0}^k \frac{\tau^n}{n !} (f_{\epsilon}, \sigma)^n \Big\rangle_{J, \beta}^{\Lambda_N} = 
  \lim_{k\rightarrow \infty} 	\lim_{\epsilon \rightarrow 0, N \rightarrow \infty} \Big\langle \sum_{n = 0}^k \frac{\tau^n}{n !} (f_{\epsilon}, \sigma)^n \Big\rangle_{J, \beta}^{\Lambda_N}
  .
\end{equation}
But by \eqref{eq:poly_moment_bounds}, the latter is $\tau^2 \frac{\beta_{\rm eff}(J,\beta)}{2 v_J^2} (f,(-\Delta_{\R^2})^{-1}f)_{\R^2}$, completing the proof of Theorem~\ref{thm:highbeta-Z2}. 
The extension for Theorem~\ref{thm:highbeta-meso} is done analogously.
\end{proof}

\begin{proof}[Proof of Lemma~\ref{lemma:discrete_approx_of_Lap^-1}]
  In what follows, given $f_{\varepsilon}: \Z^2 \to \R$, we denote by $\hat{f}_{\varepsilon}$ its Fourier transform, defined as in \cite[\eqref{eq:discr-FT}]{dgauss1}.
By definition of $\tilde C(s,m^2)$ and since $\hat{f}_\epsilon (0) = 0$, one has
\begin{align} 
\lim_{N\to\infty }\lim_{m^2\rightarrow 0} (f_\epsilon, \tilde{C}(s,m^2) f_\epsilon)
%  &= \frac{1}{4\pi^2} \int_{[-\pi,\pi]^2} \frac{\lambda_{J}(p)^{-1} (1- s\gamma \lambda(p))}{1+ s \lambda (p) ( \lambda_{J} (p)^{-1} - \gamma ) } |\hat{f}_\epsilon (p)|^2  \, dp
%  \nnb
  &= \frac{\epsilon^2}{4\pi^2}\int_{[-\pi/\epsilon,\pi/\epsilon]^2} \frac{\lambda_{J}(\epsilon p)^{-1} (1- s \gamma \lambda (\epsilon p))}{1+ s \lambda (\epsilon p)  (\lambda_{J} (\epsilon p)^{-1}  - \gamma ) } |\hat{f}_\epsilon (\epsilon p)|^2  \, dp, \label{eq:lim-no-epsilon}
\end{align}
%\begin{align} 
%\lim_{N\to\infty }\lim_{m^2\rightarrow 0} (f_\epsilon, \tilde{C}(s,m^2) f_\epsilon)
%  &= \frac{1}{4\pi^2}\int_{[-\pi,\pi]^2} \frac{\lambda_{J}(p)^{-1} - \gamma}{1+ s \lambda (p)  \lambda_{J} (p)^{-1} } |\hat{f}_\epsilon (p)|^2  \, dp
%  \nnb
%  &= \frac{1}{4\pi^2}\int_{[-\pi/\epsilon,\pi/\epsilon]^2} \frac{\epsilon^2 \lambda_{J}(\epsilon p)^{-1} - \epsilon^2 \gamma}{1+ s \lambda (\epsilon p)  \lambda_{J} (\epsilon p)^{-1} } |\hat{f}_\epsilon (\epsilon p)|^2  \, dp, \label{eq:lim-no-epsilon}
%\end{align}
where $\lambda(p)$ is the Fourier multiplier of the (unnormalised) discrete Laplacian $-\Delta$
and $\lambda_J(p)$ that of the (normalised) range-$J$ Laplacian $-\Delta_J$,
see \cite[Section~\ref{sec:decomp-fourier}]{dgauss1}.
By \cite[Lemma~\ref{lem:lambda_error}]{dgauss1},
\begin{equation}
  \lim_{\epsilon \to 0} \epsilon^{-2} \lambda (\epsilon p) = |p|^2,
  \qquad \lim_{\epsilon\to 0}  \epsilon^{-2} \lambda_{J} (\epsilon p) = v_{J} ^2|p|^2 ,
\end{equation}
and the
fraction in the integrand in \eqref{eq:lim-no-epsilon} is bounded by $C|p|^{-2}$
uniformly in $\varepsilon$ and $p \in [-\pi/\epsilon,\pi/\epsilon]^2$.
Moreover, as we now argue, \eqref{eq:feps_def} implies that
$\hat{f}_\epsilon ( \epsilon p) \rightarrow \hat{f} (p)$ as $\epsilon \to 0$ for each $p \in \R^2$
and that
$|\hat f_\epsilon(\epsilon p)| \leq C|p|(1+|p|)^{-3}$.
To see this in detail, we start from
\begin{equation}
\hat{f}_{\epsilon} (\epsilon p) = \sum_{y \in \epsilon \Z^2} f_{\epsilon} (y/\epsilon) e^{-iy\cdot p}.
\end{equation}
For $|\hat f(p)-\hat f_\epsilon(\epsilon p)| \to 0$ pointwise,
use $f \in C_c^\infty(\R^2)$ and the last condition in \eqref{eq:feps_def} to see that, with $[\cdot]$ denoting the integer part,
\begin{equation}
  |\hat f(p)-\hat f_\epsilon(\epsilon p)|
  \leq
  \int_{\R^2} |f(y)(e^{-iy\cdot p}-e^{-i\epsilon[y/\epsilon] \cdot p})| \, dy
   + \int_{\R^2} |f(y)-\epsilon^{-2}f_\epsilon([y/\epsilon])| \, dy
  \to 0.
\end{equation}
To see the bound on $\hat{f}_{\epsilon} (\epsilon p)$,
use summation by parts to write
\begin{equation}
  \lambda (p) |\hat{f}_{\epsilon}(p)|  =   |\widehat{\Delta f}_{\epsilon}(p)| = | \sum_{x\in \Z^2} e^{-ip \cdot x} \Delta f_{\epsilon}(x) | \leq \norm{ \Delta f_{\epsilon}}_{\ell^1(\Z^2)}.
\end{equation}
By \eqref{eq:feps_def},
\begin{equation}
  \norm{ \Delta f_{\epsilon}}_{\ell^1 (\Z^2)}
  \leq
  C_f^2 (\epsilon^{-1}+1)^2\norm{ \Delta f_{\epsilon}}_{\ell^\infty (\Z^2)}  
  \leq 2 C_f^2 \norm{ (\epsilon^{-1} \nabla)^2 f_{\epsilon} }_{\ell^{\infty} (\Z^2)} \leq 2 C_f^3 \epsilon^{2}
  ,
\end{equation}
and by \cite[Lemma~\ref{lem:lambda_error}]{dgauss1}, we have that $\frac{1}{\epsilon^2 |p|^2} \lambda(\epsilon p) \geq \frac{4}{\pi^2}$. 
Thus it follows that $|\hat{f}_{\epsilon} (\epsilon p )|\leq  C|p|^{-2}$.
On the other hand, since $\sum f_\epsilon =0$ and $\|f_\epsilon\|_{\ell^\infty} \leq C_f\epsilon^{2}$,  also
\begin{equation}
  |\hat{f}_{\epsilon} (\epsilon p)|
  =  \Big|\sum_{y \in \epsilon \Z^2} f_{\epsilon} (y/\epsilon) (e^{-iy \cdot p}-1)\Big|
  \leq \|f_\epsilon\|_{\ell^\infty} \sum_{y \in \epsilon \Z^2: |y|\leq C_f} |y \cdot p|
  \leq C C_f^4 |p|,
\end{equation}
and therefore $|\hat{f}_{\epsilon} (\epsilon p )|\leq  C|p| (1+|p|)^{-3}$ when combined with $|\hat f_\epsilon(\epsilon p)| \leq C|p|^{-2}$.

Finally, using the convergence in Fourier space and that
the integrand is dominated by 
$C|p|^{-2} \times ( |p|(1+|p|)^{-3} )^{ 2}\leq C(1+|p|)^{-6}$
which is integrable over $\R^2$, the Dominated convergence theorem implies
\begin{equation}
  \lim_{\epsilon\to 0}\lim_{N\rightarrow \infty} \lim_{m^2 \rightarrow 0} (f_\epsilon, \tilde{C}(s, m^2) f_\epsilon)
  = \frac{1}{4\pi^2} \int_{\R^2} \frac{1}{v_{J}^2 + s} |p|^{-2} |\hat{f}(p)|^2 \, dp
  = \frac{1}{v_{J}^2 +s} (f, (-\Delta_{\R^2})^{-1} f)
\end{equation}
as claimed.
\end{proof}

\section{Norms and contraction estimates}

We now prepare the ground for the proof of Theorem~\ref{thm:Z_N_ratio_conclusion}, which will essentially follow by suitably extending the RG flow developed in \cite{dgauss1}. This extension is designed to accomodate the external field $u$. In the present section, we discuss the necessary amendments to the norms introduced in \cite[Section~\ref{sec:norms}]{dgauss1} required to carry this out, as well as the resulting contraction estimates, cf.~\cite[Section~\ref{sec:inequalities}]{dgauss1}.
\subsection{Norms and regulators without external field}
\label{sec:recap}

We recall some essential elements of \cite{dgauss1}. Given $\Lambda_N$, the discrete two-dimensional torus of side lengths $L^N$ and a distinguished point $0 \in \Lambda_N$, 
let $\pi_N : \Z^2 \rightarrow \Lambda_N$ be the canonical projection with $\pi_N (0) = 0$. Then for each $j=0,\cdots, N$, $\cB_j$ ($j$-blocks) will be the sets of the form $\pi_N ( ([0, L^j) \cap\Z)^2  + n L^j )$ for $n\in \Z^2$, $\cP_j$ ($j$-scale polymers) are any subsets (not necessarily connected) of $\Lambda_N$ that can be obtained as the union of $j$-blocks.
For various notions related to $\cP_j$, see \cite[Section~\ref{sec:scales}]{dgauss1}.
Functions $F(X, \varphi)$ smooth in $\varphi$ that only depend on $\varphi|_{X^*}$ for each $X\in \cP_j$ are called polymer activities at scale $j$, see \cite[Section~\ref{sec:norms}]{dgauss1}; here $X^*$ refers to the small-set neighborhood of $X$, see \cite[Section~\ref{sec:polymersdef0}]{dgauss1}.

In \eqref{eq:dgauss1_RG_flow}, $Z_j$ will always be parametrised as
\begin{align}
Z_j (\varphi) &= e^{-E_j|\Lambda_N|}\sum_{X\in \cP_j (\Lambda_N)} e^{U_j (\Lambda_N \backslash X, \varphi)} K_j (X, \varphi) \label{eq:Z_j_form} \\
  U_j (X, \varphi) &= \frac{1}{2} s_j |\nabla \varphi|^2_X + \sum_{q\geq 1} L^{-2j} z_j^{(q)} \sum_{x\in X}  \cos(\sqrt{\beta} q \varphi(x))
                     = \frac{1}{2} s_j |\nabla \varphi|^2_X + W_j(X,\varphi)
                     , \label{eq:U_j_form}
\end{align}
with $U_j(\emptyset)=0$ and initial conditions $s_0 \in\R$ given, $E_0 =0$, $K_0 (X, \varphi) = 1_{X =\emptyset}$ and $z_0=  (z_0^{(q)})_{q\geq 0}$ given, where the latter refer to the Fourier coefficients of the periodic potential $\tilde{U}$ in \eqref{eq:Z_0_definition-bis}, see \cite[\eqref{eq:fourier_tildeV}]{dgauss1}.
The coordinates $U_j$ and $K_j$ are polymer activities, and in \cite[Sections~\ref{sec:norms} and \ref{sec:rg_generic_step}]{dgauss1}, they are controlled using the norms $\norm{\cdot}_{\Omega_j^U}$ and $\norm{\cdot}_{\Omega_j^K}$.
The latter norm needs an extension in the current work, so it will be reviewed in some detail here.
It is defined in terms of
positive parameters $r$, $A$, $L$, $\kappa_L$, $c_2$, $c_4$, $c_w$, $h$,
which will essentially be fixed as in \cite{dgauss1} in Section~\ref{sec:parameters} below.
The definition of the norms involves the regulator $G_j$, which is a weight defined for $X \in \cP_j$ and $\varphi\in\R^{\Lambda_N}$ by
\begin{equation}\label{eq:G_j}
G_j (X, \varphi) = \exp \bigg( \kappa_L \bigg( \norm{\nabla_j \varphi}^2_{L_j^2 (X)} + c_2 \norm{\nabla_j \varphi}^2_{L_j^2 (\partial X)} + \sum_{B\in \cB_j (X)} \norm{\nabla_j^2 \varphi}^2_{L^{\infty}(B^*)}  \bigg)  \bigg),
\end{equation}
where $\cB_j (X)$ is the set of $j$-blocks constituting $X$,
$\partial X$ denotes the inner $\ell^1$-vertex boundary of $X$,
and with the relevant $L^p$-norms as introduced in \cite[Definition~\ref{def:L-p-norms}]{dgauss1}.
The semi-norms and norms on polymer activities are then given by (cf.~\cite[Definition~\ref{def:NORM}]{dgauss1})
\begin{align}
& \norm{D^n F (X, \varphi)}_{n, T_j (X, \varphi)} = \sup \Big\{ D^n F(X, \varphi) (f_1, \cdots, f_n) : \norm{f_k}_{C_j^2 (X^*)} \leq 1 \,\, \forall k \Big\} \\
& \norm{F(X, \cdot)}_{h, T_j (X)} = \sup_{\varphi \in \R^{\Lambda_N}} G_j(X, \varphi)^{-1} \sum_{n=0}^{\infty} \frac{h^n}{n!} \norm{D^n F (X, \varphi)}_{n, T_j (X, \varphi)}  \\
& \norm{F}_{\Omega_j^K} 
= \norm{F}_{h, T_j} = \sup_{X\in \cP_j^c } A^{|X|_j} \norm{F(X, \cdot)}_{h, T_j (X)} . \label{eq:norm-Omega-Z2}
\end{align}

We will also need the following somewhat more technical properties of the norms and regulators.
For $X\in \cP_j$ and $\varphi\in\R^{\Lambda_N}$,
recall the definition $w_j (X, \varphi)^2 = \sum_{B\in \cB_j (X)} \max_{n=1,2} \norm{\nabla_j^n \varphi}^2_{L^{\infty} (B^*)}$ and
then that of the strong regulators
\begin{equation}\label{eq:strong-reg}
\exp \Big( c_w \kappa_L w_j (X, \varphi)^2 \Big), \quad g_{j} (X, \varphi) = \exp\Big( c_4 \kappa_L \sum_{a=0,1,2} W_{j} (X, \nabla^a_j \varphi)^2 \Big),
\end{equation}
where $W_j (X, \nabla_j^a\varphi)^2 = \sum_{B\in \mathcal{B}_j (X)} \norm{\nabla_j^a\varphi}_{L^{\infty}(B^*)}^2$.
For sharp integrability estimates, we subdecomposed in \cite[Section~\ref{sec:subscale}]{dgauss1}
each scale $j$ into  $M$ fractional scales $j+s$ with $s=0, \dots, 1-1/M$ when $L = \ell^M$ with $\ell$ an integer.
Each covariance  $\Gamma_{j+1}$  from the finite-range decomposition \eqref{e:Cdecomp-Z2}
has the corresponding subdecomposition
\begin{equation}
  \Gamma_{j+1}= \Gamma_{j,j+1/M}+\cdots +\Gamma_{j+(M-1)/M,j+1}.
  \label{eq:Gamma_subdecomposition}
\end{equation}
The regulators $G_{j+s}$ and the strong regulators $g_{j+s}$ are also defined on these fine scales, see \cite[\eqref{eq:W_j}]{dgauss1} and analogously for $G_{j+s}$.
The crucial property of $G_{j+s}$ and $g_{j+s}$ is stated in the next lemma, which is an extension of \cite[Lemma~\ref{lemma:G_change_of_scale}]{dgauss1} and proved in Appendix~\ref{app:prop_of_reg_with_ext_field}.
The fields $\xi_o$, $\xi_B$ appearing in the next lemma will correspond in practice to shifts induced by the external fields.

\begin{lemma} \label{lemma:G_change_of_scale_external_field}
For $X\in \cP_{j+s}$ and $\varphi$, $\xi_o$, $\xi_B \in \R^{\Lambda_N}$ for each $B\in \cB_{j+s} (X)$, define
\begin{align}
& \log G_{j+s} (X, \varphi, \xi_o, (\xi)_{B\in \cB_{j+s} (X)})  \\
& = \kappa_L \norm{\nabla_{j+s} (\varphi + \xi_o)}_{L_{j+s}^2 (X)}^2 + \kappa_L c_2 \norm{\nabla_{j+s} (\varphi + \xi_o)}_{L_{j+s}^2 (\partial X)}^2 + \kappa_L \sum_{B\in \cB_{j+s} (X)} \norm{\nabla_{j+s}^2 (\varphi + \xi_B)  }^2_{L^{\infty} (B^*)} . \nonumber
\end{align}
Assume $0\leq j < N$, $L=\ell^M$.
For any choice of $c_2$ small enough compared to $1$, there exist $c_4=c_4(c_2)$ and an integer $\ell_0=\ell_0(c_1,c_2)$ (both large), such that for all $\ell \geq \ell_0$, $M \geq 1$, $s \in \{0,\frac1M,\dots,1-\frac1M\}$ and $\kappa_L > 0$,
for $X\in \cP_{j+s}^c$,
\begin{equation}\label{eq:G_change_of_scale_variant_form}
G_{j+s} (X, \varphi , \xi_o, (\xi_B)_{B\in \cB_{j+s} (X)}) \leq \max_{\mathfrak{a} \in \{o\} \cup \cB_{j +s} (X)} g_{j+s}( X_{s+ M^{-1}}, \xi_{\mathfrak{a}} ) G_{j+s+M^{-1}} ( X_{s+ M^{-1}}, \varphi).
\end{equation}
and $X_{s+M^{-1}}$ is the smallest $(j+s+M^{-1})$-polymer containing $X$ (see \cite[Section~\ref{sec:subscale}]{dgauss1}).
\end{lemma}

\subsection{Norms and regulators with external field}

%To incorporate the effect of the scale-dependent external fields, we need an extension of the norms and regulators that take 
%the external field into account.
%
%In the context of Section~\ref{sec:scale_dep_ext_field}, we note that
%by translation invariance of the Discrete Gaussian model on the torus $\Lambda_N$,
%we may assume that $f$ is centred with respect to the block decomposition;
%that is, $\text{supp}(f)$ and $\text{supp}(\Delta f)$ are contained in the box $m+[0,\frac14L_{j_f})^2$, where $m$ is (one of) the lattice points closest to the center of some block $B \in \cB_j$ for all $j_f \leq j \leq N$.
%%
%In particular, then, by Lemma~\ref{lem:extfield_bd}, for all scales $j \leq N$,
%there is a block $B\in \cB_j$ such that whenever $L \geq C$, $N_{ 5}(\supp(u_j)) \subset B$
%where $N_k(X)$ denotes the set of points with $\ell^1$-distance at most $k$ from the set $X$.

%{\grey
%\commentjp{this is moved to above Theorem~\ref{thm:Z_N_ratio_conclusion}}
%Somewhat more generally,
%we assume that $(u_j)_j = (u_j \in \R^{\Lambda} )_j$ is a sequence of given external fields that are uniformly bounded and supported on a single block in the sense that
%\begin{itemize}
%\item[$(\textnormal{A}_u)$] There exist $j_u$, 
%  $M_u >0$ such that $u_j =0$ for $j\leq j_u$, $\kappa_L  \norm{u_j}_{C_j^2} \leq M_u$ for each $j\leq N$,
%  and $u_j$ is supported on the unique $B_0 \in \cB_j$ such that $0\in B_0$ and $d(\partial B_0, \supp (u_j)) > 4$,
%\end{itemize}
%with $\kappa_L$ same as that of \eqref{eq:G_j}.
%}
To incorporate the effect of the scale-dependent external fields, we need an extension of the norms and regulators that take 
the external field into account.
The following definition introduces modified regulators that effectively control the polymer activities perturbed by the external fields $(u_j)$.

\begin{definition} \label{def:Psi_regulator}
Given $(u_j)_j$ satisfying \ref{assump:u}, define the $\Psi$-regulators (cf.~\eqref{eq:G_j})
\begin{multline}
G_j^{\Psi} ( X, \varphi ; \, u_j) \\= \sup_{t \in [0,1]} \exp\Big( \kappa_L \norm{\nabla_j (\varphi + tu_j) }^2_{L^2_j (X)} + c_2 \kappa_L \norm{\nabla_j (\varphi + tu_j)}^2_{L^2_j (\partial X)} + \kappa_L W_j^{\Psi} (X, \varphi ; \, u_j)^2  \Big) 
\end{multline}
where
\begin{align}
W_j^{\Psi} (X, \nabla_j^2 \varphi ; \, u_j)^2  = \sum_{B\in \cB_j (X)} \sup_{t_B \in [0,1]} \norm{ \nabla_j^2 (\varphi + t_B u_j) }_{L^{\infty} (B^*)}^2 .
\end{align}
The dependence on $u_j$ will often be hidden.
\end{definition}
\begin{remark}
The main motivation for $G_j^{\Psi}$ is to have $\sup_{t \in [0,1]} G_j (X, \varphi + tu_j) \leq G_j^{\Psi} (X, \varphi)$ and hence
\begin{align}
\norm{K(X, \varphi + tu_j)}_{h, T_j (X, \varphi )} \leq \norm{K(X)}_{h, T_j(X)} G_j^{\Psi} (X, \varphi), \quad t \in[0,1],  \label{eq:Psi-regulator_key_property1}
	.
\end{align}
%where, with slight abuse of notation, $\norm{K(X, \varphi + tu_j)}_{h, T_j (X, \varphi )}$ means $\norm{K(X, \cdot+tu_j)}_{h, T_j (X, \varphi )}$.
Note that we could not use $\sup_{t \in [0,1]} G_j (X, \varphi + tu_j)$ for $G_j^{\Psi}$ because this definition does not factorise into connected components, i.e.,
\begin{align}\label{eq:nofact}
\sup_{t \in [0,1]} G_j (X \cup Y, \varphi + tu_j) \neq \sup_{t_1, t_2 \in [0,1]} G_j (X , \varphi + t_1 u_j) G_j (Y , \varphi + t_2 u_j) 
\end{align}
if $X \not\sim Y = \emptyset$ but $X^* \cap Y^* \neq \emptyset$.
 This is why we introduced the $W_j^{\Psi}$.

Also note that since $\norm{\nabla_j u_j}^2_{L_j^2 (X)}$, $\norm{\nabla_j u_j}^2_{L_j^2 (\partial X)}$, $W_j (X, \nabla^2_j u_j)$ are each bounded by some multiple of $\norm{u_j}^2_{C_j^2}$,
in particular, 
there exists finite $C >0$, 
%$C(M_u) >0$, 
independent of $X$,  such that,
under \ref{assump:u},
\begin{align}
G_j^{\Psi} (X, 0 ; u_j) \leq C . \label{eq:G_j^Psi_zero_point_bound}
\end{align}
%\begin{align}
%G_j^{\Psi} (X, 0 ; u_j) \leq C(M_u) . \label{eq:G_j^Psi_zero_point_bound}
%\end{align}
\end{remark}

The following are the key properties of $G_j^{\Psi}$ (cf.\ the properties of $G_j$ in
\cite[Section~\ref{sec:norms}]{dgauss1}).

\begin{proposition} 
\label{prop:key_prop_G_j^Psi} 
Let $(u_j)_j$ satisfy \ref{assump:u}.
Then there exists $C_{\Psi} >0$
%$C_{\Psi} \equiv C_{\Psi} (M_u) >0$
such that for $L$
{as in the assumption of Lemma~\ref{lemma:G_change_of_scale_external_field}}
and sufficiently small $c_2$, $c_w >0$, $(G_j^{\Psi})_{j\geq 0} \equiv (G_j^{\Psi} (\cdot ; u_j))_{j\geq 0}$ satisfies for each $(X, \varphi) \in \cP_j \times \R^{\Lambda_N}$, $j \geq 0$, 
\begin{itemize}
\item[\textnormal{(1)}] $G_j^{\Psi} (X, \varphi) \geq G_j (X, \varphi)$, 
\item[\textnormal{(2)}] $G_j^{\Psi} (X, \varphi) = \prod_{Y\in \textnormal{Comp}_j(X)} G_j^{\Psi} (Y, \varphi)$,
\item[\textnormal{(3)}] $e^{c_w \kappa_L w_j(X, \varphi + tu_j)^2} G_j^{\Psi} (Y, \varphi) \leq C_{\Psi} G_j^{\Psi} (X \cup Y, \varphi)$ if $X\cap Y = \emptyset$ and $t \in [0,1]$,
\item[\textnormal{(4)}] $\Eplus [G_j^{\Psi} (X, \varphi' + \zeta)] \leq C_{\Psi} 2^{|X|_j} G_{j+1} (X, \varphi')$ for all $\varphi' \in \R^{\Lambda_N}$.
\end{itemize}
\end{proposition}
\begin{proof}
By definition of $G_j^{\Psi}$, properties (1) and (2) are clear.
For (3), first observe from the definition of $w_j (X, \varphi)$ (see above \eqref{eq:strong-reg}) that for each $t' \in [0,1]$ and some geometric constant $C > 0$, 
\begin{align}
w_j (X, \varphi + t u_j)^2 & \leq 2 \sum_{B\in \cB_j (X)}  \Big(  \norm{\nabla_j (\varphi + t' u_j)}_{L^{\infty} (B^*)}^2 + (t- t')^2 \norm{\nabla_j u_j}_{L^{\infty} (B^*)}^2 \\
& \qquad \qquad \qquad + \norm{\nabla_j^2 (\varphi + t' u_j ) }_{L^{\infty} (B^*)}^2  + (t-t')^2 \norm{\nabla_j^2 u_j}_{L^{\infty} (B^*)} \Big)  \nnb
& \leq 2  \sum_{B\in \cB_j (X)}  \big( \norm{\nabla_j (\varphi + t' u_j)}_{L^{\infty} (B^*)}^2 + \norm{\nabla_j^2 (\varphi + t' u_j )}_{L^{\infty} (B^*)}^2 \big) + C  \norm{u_j}_{C_j^2}^2. \label{eq:key_prop_G_j^Psi1} 
\end{align}
We then note that for any $B\in \cB_j (X)$, $x_0 \in B$ and $x\in B^*$, there is another constant $C> 0$ such that
\begin{equation}
| \nabla_j^{ \mu} \varphi (x) | \leq |\nabla_j^{ \mu}  \varphi (x_0)| + C \norm{\nabla_j^2 \varphi}_{L^{\infty} (B^*)},
\end{equation}
for all $\mu \in \hat{e}$ (for example, cf.~\cite[\eqref{eq:gradL_infty}]{dgauss1}, applied to $f=\nabla_j^{\mu} \varphi$ and recall that $x_0$ and $x$ belong to some small set $X$, whence $|X|_j \leq C$)
and hence 
$\norm{\nabla_j \varphi}_{L^{\infty} (B^*)}^2 \leq 2 { \max_{\mu \in \hat{e}}} | \nabla^{ \mu} \varphi(x_0) |^2 + 2C^2 \norm{\nabla_j^2 \varphi}_{L^{\infty} (B^*)}^2$. 
Summing over all $x_0 \in B$, this implies
\begin{equation}
\norm{\nabla_j \varphi}^2_{L^{\infty} (B^*)} \leq 2 \norm{\nabla_j \varphi}^2_{L_j^2 (B)} +2 C^2 \norm{\nabla_j^2 \varphi}_{L^{\infty} (B^*)}^2.
\end{equation}
Plugging this into \eqref{eq:key_prop_G_j^Psi1}, we get
%\begin{equation}
%c_w w_j (X, \varphi + tu_j)^2  \leq C c_w  \norm{u_j}_{C_j^2}^2  + \frac{1}{2}  \norm{\nabla_j (\varphi + t' u_j)}^2_{L_j^2 (X)} + \frac{1}{2} W_j^{\Psi} (X\cup Y, \nabla_j^2 \varphi)^2 \label{eq:key_prop_G_j^Psi2} 
%\end{equation}
\begin{equation}
c_w w_j (X, \varphi + tu_j)^2  \leq C c_w  \norm{u_j}_{C_j^2}^2  + \frac{1}{2}  \norm{\nabla_j (\varphi + t' u_j)}^2_{L_j^2 (X)} + \frac{1}{2} W_j^{\Psi} (X, \nabla_j^2 \varphi)^2 \label{eq:key_prop_G_j^Psi2} 
\end{equation}
for $c_w$ sufficiently small. The discrepancy between the left- and right-hand sides of item (3) of the statement of the proposition
due to the boundary term of $G_j$ can be treated by the discrete Sobolev trace theorem \cite[Corollary~\ref{cor:discrete_trace_theorem}]{dgauss1}, which shows that there is $C>0$ such that
%\begin{equation}
%c_2 \norm{\nabla_j (\varphi + t' u_j)}_{L_j^2 (\partial Y \backslash \partial (X \cup Y) )}^2 \leq 
%C c_2 \Big( \norm{\nabla_j (\varphi + t' u_j)}_{L_j^2 (X)}^2 + \norm{\nabla_j^2 (\varphi + t' u_j)}_{L^{\infty} (Y) }^2  \Big),
%\end{equation}
\begin{equation}
c_2 \norm{\nabla_j (\varphi + t' u_j)}_{L_j^2 (\partial Y \backslash \partial (X \cup Y) )}^2 \leq 
C c_2 \Big( \norm{\nabla_j (\varphi + t' u_j)}_{L_j^2 (X)}^2 + \norm{\nabla_j^2 (\varphi + t' u_j)}_{L^{\infty} (X) }^2  \Big),
\end{equation}
so if $c_2$ is sufficiently small so that $C c_2 \leq \frac{1}{2}$, then this together with \eqref{eq:key_prop_G_j^Psi2} gives
\begin{align}
& c_w w_j (X, \varphi + tu_j)^2 + \norm{\nabla_j (\varphi + t' u_j)}^2_{L_j^2 (Y)} + c_2 \norm{\nabla_j (\varphi + t' u_j)}^2_{L_j^2 (\partial Y)} + W_j^{\Psi}(Y, \nabla_j^2 \varphi)^2 \\
& \leq C c_w \norm{u_j}^2_{C_j^2} + \norm{\nabla_j (\varphi + t' u_j)}^2_{L_j^2 (X \cup Y)} + c_2 \norm{\nabla_j (\varphi + t' u_j)}^2_{L_j^2 (\partial (X\cup Y))} + W_j^{\Psi}(X\cup Y, \nabla_j^2 \varphi)^2. \nonumber
\end{align}
After taking supremum over $t'$,
it follows that (3) holds for any $C_{\Psi} \geq \exp( C c_w \kappa_L \norm{u_j}_{C_j^2}^2)$, and $C_{\Psi}$ can be chosen independent of $j$ because of \ref{assump:u}.

For (4), we may assume that $j\leq N-2$, since $\Gamma_N^{\Lambda_N}$ satisfies the same estimates as $\Gamma_N$.
We use the regulator decomposition: by Lemma~\ref{lemma:G_change_of_scale_external_field},
\begin{align}
G_j^{\Psi} (X, \varphi' + \zeta  ; \, u_j) \leq \prod_{k=1}^M \sup_{t\in [0,1]} g_{j+ \frac{k-1}{M}} (X_{k/M}, \xi_k + 1_{k=1}t u_j ) G_{j+1} (\bar{X}, \varphi') \label{eq:G_j^Psi_quad_bound}
\end{align}
whenever $\zeta = \sum_{k} \xi_k$ and $X_{k/M}$ is the smallest polymer in $\cP_{j+k/M}$ containing $X\in \cP_j$ and $\overline{X}=X_1$.
Using the covariance subdecomposition \eqref{eq:Gamma_subdecomposition}, we may decompose
$\zeta \sim \cN(0, \Gamma_{j+1})$ as the sum of independent $\xi_k \sim \cN (0, \Gamma_{j+k/M, j+(k+1)/M} )$. Then each $\E^{\xi_k} [g_{j+ (k-1)/M} (X_{k/M}, \xi_k + 1_{k=1}t u_j ) ]$ are bounded using \cite[Lemma~\ref{lemma:g_j+s_bound_by_quadratic}]{dgauss1}. For $k=1$, we have from the definition of $g_{j}$ that
\begin{equation}
g_{j} (X_{M^{-1}}, \xi_1 + t u_j) \leq g_{j} (X_{M^{-1}}, \xi_1)^2 g_{j} (X_{M^{-1}}, u_j)^2 \leq g_{j} (X_{M^{-1}}, \xi_1)^2 e^{c \kappa_L \norm{u_j}_{C_j^2}}
\end{equation}
for some $c>0$. Also for any $k\in \{1,\dots, M\}$, \cite[Lemma~\ref{lemma:g_j+s_bound_by_quadratic}]{dgauss1} gives
\begin{equation}
\E^{\xi_k}[g_{j+ (k-1)/M} (X_{k/M}, \xi_k )] \leq \E^{\xi_k}[g_{j+ (k-1)/M} (X_{k/M}, \xi_k  )^2] \leq 2^{|X|_j/M}
\end{equation}
with the choice of $L$ and $\ell$ as in Lemma~\ref{lemma:G_change_of_scale_external_field} (cf. \cite[Appendix~\ref{app:lem_G_change_of_scale}]{dgauss1}). Therefore
\begin{equation}
\Eplus [ G_j^{\Psi} (X, \varphi' + \zeta) ] \leq e^{c \kappa_L \norm{u_j}_{C_j^2} } 2^{|X|_j} G_{j+1} (\bar{X}, \varphi')
\end{equation}
which implies the claim with the same choice of $C_\Psi$ as in (3).
\end{proof}

Next we define a norm corresponding to the $\Psi$-regulators.
This norm is defined in the same way as the $\|\cdot\|_{h,T_j}$-norm in \eqref{eq:norm-Omega-Z2} except that there is, apart from the use of  $G_{j}^{\Psi}$ instead of $G_{j}$, also a change of the
parameter $A$ (large-set regulator) from $A$ to $A/2$. This is to compensate a combinatorial factor coming from reblocking in the next section, which will not significantly affect the resulting estimates.

\begin{definition}
Define, for $\Psi_j : \cP_j \times \R^{\Lambda_N} \rightarrow \R$ such that $\Psi_j (X) = \prod_{Y\in \textnormal{Comp}_j (X)} \Psi_j (Y)$,
\begin{align}
& \norm{\Psi_j (X)}_{h, T_j^{\Psi} (X) } = \sup_{\varphi} G_{j}^{\Psi} (X, \varphi)^{-1} \norm{\Psi_j (X, \varphi)}_{h, T_j (X, \varphi)}  \label{eq:Omega_j^Psi_norm_def_X} \\
& \norm{\Psi_j}_{h, T_j^{\Psi}} = \sup_{ X \in \cP_j^c} (A/2)^{|X|_j} \norm{\Psi_j (X)}_{h, T_j^{\Psi} (X) } . \label{eq:Omega_j^Psi_norm_def} 
\end{align}
\end{definition}

\subsection{Contraction estimates}

This short section can be regarded as an extension of \cite[Section~\ref{sec:inequalities}]{dgauss1}, but some results are now generalized to apply to the norm $\norm{\cdot}_{h, T_j^{\Psi}(X)}$.
In the following we write
\begin{align}
G^*_j (X, \varphi) = \begin{array}{ll}
\begin{cases}
G_j (X, \varphi) & \text{if} \;\; * = 0 \\
G_j^{\Psi} (X, \varphi) &  \text{if} \;\; * = \Psi.
\end{cases}
\end{array}
\end{align}
Note that by applying Proposition~\ref{prop:key_prop_G_j^Psi} to both $u_j$ and $u_j'\equiv0$ (which also satisfies \ref{assump:u}), one obtains that 
\begin{equation}\label{eq:analogue-Prop5.9}
\Eplus [G_j^* (X, \varphi' + \zeta)] \leq 
C_{\Psi} 2^{|X|_j} G_{j+1} (X, \varphi'), \quad \text{ for both $* \in \{0, \Psi \}$.} 
\end{equation}
We also use the notation $\norm{\cdot}_{h, T_j^*(X)}$ for either $\norm{\cdot}_{h, T_j (X)}$ or $\norm{\cdot}_{h, T_j^{\Psi} (X)}$ and $\norm{\cdot}_{h, T_j^*}$ for either $\norm{\cdot}_{h, T_j}$ or $\norm{\cdot}_{h, T_j^{\Psi}}$ when $*= 0$ or $\Psi$, respectively. 

Below, we refer to $2\pi /\sqrt{\beta}$-periodic polymer activities to be the functions $F(X, \varphi)$ such that $t \mapsto F(X, \varphi +t)$ is $2\pi /\sqrt{\beta}$-periodic, see \cite[Definition~\ref{def:chargedecomp}]{dgauss1}. Then its charge-$q$ part is defined by the Fourier expansion
\begin{align}
F(X, \varphi + t) = \sum_{q\in \Z} e^{i \sqrt{\beta} qt} \hat{F}_q (X, \varphi), \quad t\in \R,
\end{align}
and $F$ is called neutral if $F = \hat{F}_0$.
Recall that the norm $\norm{\cdot}_{h, T_j^* (X)}$ in \eqref{eq:Omega_j^Psi_norm_def_X} depends implicitly on a choice of $u=(u_j)_j$ and the notion of small sets $\mathcal{S}_j$ at scale $j$ from \cite[Section~\ref{sec:polymersdef0}]{dgauss1}.

\begin{proposition} \label{prop:contraction_estimates_external_field}
Let $X \in \mathcal{S}_j$, and let $F$ be a $2\pi/\sqrt{\beta}$-periodic polymer activity such that $\norm{F}_{h, T_j^* (X)} < \infty$ where $* \in \{0, \Psi \}$. Let $(u_j)_j$ satisfy \ref{assump:u}. Then for some 
$C > 0$
%$C=C (M_u) \in (0,\infty)$ 
and $L \geq L_0$, the following hold.
\begin{itemize}
\item If $F$ has charge $q$ with $|q|\geq 1$, then for all $\varphi' \in \R^{\Lambda_N}$,
\begin{align}
\norm{\Eplus F (X, \varphi' + \zeta) }_{h, T_{j+1} ({X}, \varphi')} \leq C
 e^{\sqrt{\beta} |q| h} e^{-(|q|-1/2) r \Gamma_{j+1} (0)} \norm{F(X)}_{h, T_j^* (X)} G_{j+1} (\bar{X}, \varphi') . \label{eq:contraction_of_charge_q_term_external_field}
\end{align}
\item If $F$ is neutral, then for all $\varphi' \in \R^{\Lambda_N}$,
\begin{align}
\norm{\Eplus [ F(X,\varphi' + \zeta) - F(X, \zeta)  ]  }_{h, T_{j+1} (X, \varphi')} \leq C
L^{-1} (\log L)^{1/2} \norm{F(X)}_{h, T_j^* (X)} G_{j+1} (\bar{X}, \varphi') . \label{eq:gaussian_contraction_external_field}
\end{align}
\end{itemize}
\end{proposition}
\begin{proof}
The first item, \eqref{eq:contraction_of_charge_q_term_external_field} for $* = 0$ is just \cite[Lemma~\ref{lemma:contraction_of_charge_q_term}]{dgauss1}.

%\sout{ For $* = \Psi $, it suffices to argue that the conclusion of Lemma~\ref{lemma:charged_term_exponential_contraction}, i.e., the identity} 
%\cite[\eqref{eq:exponential_contraction}]{dgauss1} 
%\sout{continues to hold under the modified assumption that $\norm{F}_{h, T_j^{\Psi} (X)} < \infty$. Indeed with} \cite[\eqref{eq:exponential_contraction}]{dgauss1} \sout{at hand, }

For $* = \Psi $, it suffices to argue that the conclusion of \cite[Lemma~\ref{lemma:charged_term_exponential_contraction}]{dgauss1} continues to hold under the modified assumption that $\norm{F}_{h, T_j^{\Psi} (X)} < \infty$. Indeed with this at hand, 
the proof of \eqref{eq:contraction_of_charge_q_term_external_field} proceeds exactly as that of \cite[Lemma~\ref{lemma:contraction_of_charge_q_term}]{dgauss1}, except that one invokes \eqref{eq:analogue-Prop5.9} above rather than \cite[Proposition~\ref{prop:E_G_j}]{dgauss1} towards the end of that proof. As to why the identity \cite[\eqref{eq:exponential_contraction}]{dgauss1} still holds, one simply observes upon inspecting its proof that an analogue of the argument in \cite[\eqref{eq:F-bound-id}--\eqref{eq:g-o1}]{dgauss1} involving $\norm{F(X)}_{h, T_j^* (X)}$ still applies when combining \eqref{eq:G_j^Psi_quad_bound} (which generalises \cite[Lemma~\ref{lemma:G_change_of_scale}]{dgauss1}) with \cite[\eqref{eq:g_j+s_bound}]{dgauss1}.

To see the second point, we proceed similarly as in \cite[Lemma~\ref{lemma:gaussian_contraction}]{dgauss1}: writing $(\Rem_0 \Eplus F) (X, \varphi') = \Eplus[ F(X,\varphi' + \zeta) - F(X, \zeta) ]$, Taylor's theorem and neutrality of $F$ give
\begin{align}\label{eq:Rem-0}
(\Rem_0 \Eplus F) (X, \varphi') = \int_0^1 dt \, (1-t) D \Rem_0 \Eplus F(X, \zeta + t\varphi')(\delta \varphi')  ,
\end{align}
where $\delta \varphi'(x)= \varphi'(x)-\varphi'(x_0)$ for a fixed point $x_0 \in X$. 
But since $D \Rem_0 \Eplus F (X,\varphi')= \Eplus D F(X, \cdot + \varphi')$, the left-hand side of \eqref{eq:Rem-0} is bounded in absolute value
by
\begin{multline}\label{eq:bd-rem_0}
  h^{-1}
  \int_0^1 dt\, (1-t) \norm{ \Eplus D F(X, \cdot + t\varphi')}_{h, T_{j+1}^* (X, t\varphi')} \norm{\delta \varphi'}_{C^2_{j+1} (X^*)}\\
  \leq h^{-1} C_d L^{-1} \int_0^1 dt\, (1-t) \norm{F (X)}_{h, T_j^* (X)} \Eplus [ G_{j}^{*}(X, t\varphi' + \zeta) ] \norm{\delta \varphi'}_{C^2_{j+1} (X^*)} ,
\end{multline}
applying \cite[\eqref{eq:neutral_contract1}]{dgauss1} in the second line.
Moreover, $\Eplus [ G_{j}^{*}( X, t\varphi' + \zeta) ] \leq C_{\Psi} 2^{|X|_j} G_{j+1} (\bar{X}, t\varphi')$ for both $* \in \{0, \Psi \}$, as follows readily from \eqref{eq:analogue-Prop5.9}, and by \cite[\eqref{eq:G_j_with_t_bounded_by_G_j}]{dgauss1} (applied with $n=2$),
\begin{align}
\Eplus [G_j^* (X, t\varphi' + \zeta)] \norm{\delta \varphi'}_{C_{j+1}^2 (X^*)} \leq C (\log L)^{1/2} G_{j+1} (\bar{X}, \varphi') .
\end{align}
On the other hand, for $n\geq 1$, $D^n \Rem_0 = D^n$ and thus by \cite[\eqref{eq:neutral_contract1}]{dgauss1}, we immediately get
\begin{align}\label{eq:bd-rem_n}
	| D^n (\Rem_0 \Eplus F) (X, \varphi) (f_1, \cdots, f_n) | \leq (C_g L)^{-n} \norm{D^n \Eplus F (X, \varphi' + \zeta)}_{h, T_j (X, \varphi')} \prod_{k=1}^n \norm{f_k}_{C_{j+1}^2 (X^*)}
\end{align}
for some constant $C_g >0$.
We obtain \eqref{eq:gaussian_contraction_external_field} from \eqref{eq:bd-rem_0}, \eqref{eq:bd-rem_n} by summing $\frac{h^n}{n!} \norm{D^n \Rem_0 \Eplus F(X, \varphi')}_{h, T_j (X, \varphi')}$ over $n \geq 0$.
\end{proof}

Finally, we recall  the definition of the reblocking operator from \cite[Definition~\ref{def:reblocking}]{dgauss1},
defined for a $j$-scale polymer activity $F$ by
\begin{equation}
\mathbb{S} F (X) = \sum_{Y \in \cP_j^c}^{\bar{Y}=X} F(Y), \quad X \in \cP_{j+1}^c
\end{equation}
and extended to disconnected $Z \in \cP_{j+1}$ by $\mathbb{S} F(Z) = \prod_{X\in \operatorname{Comp}_{j+1}(Z)} \mathbb{S} F (X)$.
The following lemma extends the reblocking estimate from \cite[Proposition~\ref{prop:largeset_contraction-v2}]{dgauss1}.
The only difference is that the bound on the right-hand side 
also holds for the weaker norm $\norm{\cdot}_{h, T_j^{\Psi}}$.

\begin{proposition} \label{prop:largeset_contraction_external_field}
  There exists a geometric constant $\eta >0$ and $\epsilon_{rb} := A^{-8}$ such that the following holds. 
  Let $F$ be a polymer activity supported on large sets and satisfy $\norm{F}_{h, T_j^{*}} \leq \epsilon_{rb}$. Then for any $L\geq 5$, $(A/2)^{\eta} \geq L (2eL^2)^{1+\eta}$, $X\in \cP_{j+1}$ and $* \in \{0, \Psi \}$,
\begin{equation}
\norm{\mathbb{S} \Eplus [ F(X, \cdot + \zeta)  ] }_{h, T_{j+1} (X)} \leq (L^{-1} A^{-1} )^{|X|_{j+1}} \norm{F}_{h, T_j^*}.
\end{equation}
\end{proposition}

\begin{proof}
The case $* = 0$ is exactly \cite[Proposition~\ref{prop:largeset_contraction-v2}]{dgauss1}.
The case $* = \Psi$ is obtained by following the same proof, but $A$ is replaced by $A/2$ in view of the definition of $\norm{\cdot}_{h, T_j^{\Psi}}$, see \eqref{eq:Omega_j^Psi_norm_def}.
\end{proof}

\subsection{Choice of parameters}
\label{sec:parameters}

Finally, we explain how the parameters in the norms above are chosen.
  
First of all, the parameters $\kappa_L$, $c_2$, $c_4$, $c_w$ are chosen as in \cite[Section~5]{dgauss1} (see the end of Section~\ref{sec:polymersdef} and Remark~\ref{R:choice-parameters} therein), except that we impose the extra conditions resulting from the assumptions of Lemma~\ref{lemma:G_change_of_scale_external_field} and Proposition~\ref{prop:key_prop_G_j^Psi}. These do not contradict the conditions  from \cite[Section~5]{dgauss1} as they only impose further smallness conditions
on  $c_w$, $c_2$, $c_4$.

Next, given a finite-range step distribution $J$, we fix an additional parameter $r\in (0,1]$ such that (with $C= \sqrt{2} c_h c_f^{-1} $, an absolute constant from \cite[Lemma~\ref{lem:W_norm}]{dgauss1},
cf.~also \cite[\eqref{eq:U_norm} and Lemma~\ref{lemma:choice_of_c_h}]{dgauss1} regarding the choices of $c_f$ and $c_h$, respectively)
\begin{equation}
  Cr \leq \rho_J^2 ,
\end{equation}
and we always impose the condition (with $C= 2 \max\{ c_f^{-2}, c_f^{-1}  \} $,  also an absolute constant from \cite[Lemma~\ref{lem:W_norm}]{dgauss1})
\begin{equation}
  \beta \geq C.
\end{equation}
The parameter $h$ is then chosen as in \cite[Definition~\ref{def:K_space}]{dgauss1}
as $h = \max\{ c_f^{1/2}, r c_h \rho_J^{-2} \sqrt{\beta}, \rho_J^{-1} \}$.

Finally, we will assume that $L \geq L_0$ and $A \geq A_0(L)$ with $L_0$ and $A_0(L)$ chosen to satisfy the assumptions
of \cite[Theorem~\ref{thm:local_part_of_K_j+1}]{dgauss1} as well as of those of Lemma~\ref{lemma:G_change_of_scale_external_field}, Proposition~\ref{prop:contraction_estimates_external_field}, and Proposition~\ref{prop:largeset_contraction_external_field} above.
Moreover, we will always tacitly assume from here on that $L$ is $\ell$-adic, i.e.,~of the form $L=\ell^M$ for some integer $M \geq 1$, where $\ell := \min \{ 2^n : 2^n \geq \ell_0\}$ is the smallest dyadic integer larger than $\ell_0$ (with $\ell_0$ as supplied by Lemma~\ref{lemma:G_change_of_scale_external_field}, now fixed since $c_2$ is). This ensures that i) Lemma~\ref{lemma:G_change_of_scale_external_field} is always in force and ii) eventually, \eqref{eq:TDlimit} can be used (since $L$ is automatically dyadic).
Later in Sections~\ref{sec:rg_generic_step_external_field} and~\ref{sec:proof_Z_N_ratio},
further lower bound conditions on $L$ and $A$ will be imposed,
which are consistent with our standing assumptions $L \geq L_0$ and $A \geq A_0 (L)$.

\section{Reblocking the external field}

We will use a renormalisation group analysis in Section~\ref{sec:rg_generic_step_external_field} to study the flow of the partition functions defined by \eqref{eq:Zplusu_def}.
Ideally, we would like to write the renormalisation group maps
in identical form as those of \cite[Section~\ref{sec:rg_generic_step}]{dgauss1}, but the introduction of the external field $u_j$ breaks the algebraic form of $U_j$ (see \eqref{eq:U_j_form})
and the symmetry of the system that we used to define the localisation operators $\Loc$ in \cite{dgauss1}.
Thus, we will first reduce the problem caused by the external field to a setting where the form of $U_j$ stays the same as in the original renormalisation group steps and then bound the perturbation created by this operation.
This is achieved by the following proposition and lemma.
For $X \in \cP_j$, recall that $ \cP_j (X)$ denotes the set of all $j$-polymers $Y$ such that $Y \subset X$.

\begin{definition} \label{def:Psi_j}
Given $u_j \in \R^{\Lambda_N}$ and scale-$j$ polymer activities $K_j$ and $U_j$, define for $X \in \cP_j^c$,
\begin{align}
\cF_{\Psi} [u_j, U_j, K_j \, ; \, j ] (X, \varphi) = 
-K_j (X, \varphi) + \sum_{Y\in \cP_j (X)} (e^{U_j(\cdot, \varphi + u_j)  } - e^{U_j (\cdot, \varphi)} )^{X \backslash Y} K_j (Y, \varphi + u_j)
\end{align}
 where
\begin{equation}
(e^{U_j(\cdot, \varphi + u_j)  } - e^{U_j (\cdot, \varphi)} )^{Z} \stackrel{\textnormal{def.}}{=} \prod_{B \in \cB_j(Z) } (e^{U_j(B, \varphi + u_j)  } - e^{U_j (B, \varphi)}), \quad \text{ for $Z\in \cP_j$,}
\end{equation}
and $\cF_{\Psi} [u_j, U_j, K_j \, ; \, j ] (Z, \varphi) = \prod_{X\in \textnormal{Comp}_j (Z)} \cF_{\Psi} [u_j, U_j, K_j \, ; \, j ] (X, \varphi)$ for general $Z\in \cP_j$.
\end{definition}

The dependence of $\cF_{\Psi}$ on the scale $j$ will often be omitted when it is clear from the context.
The following is a purely algebraic statement.  Note in particular that the assumptions on $U_j$, $K_j$ appearing below will be satisfied by the choices in \eqref{eq:Z_j_form}, \eqref{eq:U_j_form}.

\begin{proposition} \label{prop:reblocking_Z_with_Psi}
  Assume that for some scale-$j$ polymer activities $K_j$ and $U_j$,
  \begin{equation}
    Z_j (\varphi) = e^{-E_j |\Lambda_N|} \sum_{X \in \cP_j} e^{U_j (\Lambda\setminus X, \varphi)} K_j (X, \varphi),
\end{equation}
and that $U_j$ is additive over blocks, i.e.,~$U_j(X \cup Y)=U_j(X)+ U_j(Y)$ for all $X\cap Y =\emptyset$, $X,Y \in \cP_j$.
Let $\Psi_j = \cF_{\Psi} [u_j, U_j, K_j ; j]$. 
Then
\begin{equation}
Z_j (\varphi + u_j) = e^{-E_j |\Lambda_N|} \sum_{X\in \cP_j} e^{U_j (\Lambda \backslash X, \varphi)} \prod_{Z\in \textnormal{Comp}_j (X)} ( K_j + \Psi_j ) (Z, \varphi) .  \label{eq:reblocking_Z_with_Psi}
\end{equation}
If $K_j, U_j$ are $2\pi/\sqrt{\beta}$-periodic, then so is $\Psi_j$.
If $u_j$ satisfies \ref{assump:u}, then $\Psi_j (X) =0$ whenever $B_0^* \cap X = \emptyset$. 

%\sout{If $u_j$ satisfies \ref{assump:u}, then $\Psi_j (X) =0$ whenever $0\not\in X$. }
\end{proposition}
\begin{proof}
  This is a result of a simple reblocking argument.  Using the assumption
  \begin{equation}
   Z_j (\varphi + u_j) = e^{-E_j |\Lambda_N|} \sum_{X \in \cP_j} e^{U_j (X, \varphi + u_j)} K_j (\Lambda \backslash X, \varphi + u_j),
  \end{equation}
  by making the substitution
\begin{equation}
  e^{U_j (X, \varphi + u_j)} = \prod_{B\in \cB_j (X)}  e^{U_j (B, \varphi + u_j)}
  = \sum_{Y\in \cP_j (X)} ( e^{U_j (\cdot, \varphi + u_j)} - e^{U_j (\cdot, \varphi)} )^{Y} e^{U_j (X\backslash Y, \varphi)}
\end{equation}
we immediately obtain that
\begin{align}
Z_j (\varphi + u_j) = e^{-E_j |\Lambda_N|} \sum_{Y\in \cP_j} e^{U_j (Y, \varphi)} \sum_{X' \in \cP_j (\Lambda \backslash Y)} ( e^{U_j (\cdot, \varphi + u_j)} - e^{U_j (\cdot, \varphi)} )^{X'} K_j (\Lambda \backslash (Y \cup X'), \varphi + u_j).
\end{align}
Then we arrive at \eqref{eq:reblocking_Z_with_Psi} after factoring the above expression into connected components of $\Lambda \backslash Y$.

The asserted periodicity of $\Psi_j$ is plainly inherited from $K_j, U_j$ 
 and the last remark is a consequence of the fact that $K_j (X, \varphi + u_j) = K_j (X, \varphi)$ for $0 \not\in X^*$ and $u_j$ satisfying \ref{assump:u}.
\end{proof}

For the next estimates, recall the definition of the space $\Omega_j^U$ from \cite[Definition~\ref{def:U_space}]{dgauss1}
and of $\Omega_j^K$ from \cite[Definition~\ref{def:K_space}]{dgauss1}.
In particular, the parameters these spaces and their norms depend on are always assumed to satisfy the conditions specified in Section~\ref{sec:parameters}.

\begin{lemma} \label{lemma:G_j^Psi_is_Psi_regulator}
Suppose $(u_j)_{j}$ satisfies \ref{assump:u}. Given $U_j$ in form \eqref{eq:U_j_form} and $K_j$ a $2\pi/\sqrt{\beta}$-periodic polymer activity, let $\Psi_j = \cF_{\Psi} [u_j, U_j, K_j]$.
Then there exist $C >0$
%$C_{\Psi} \equiv C_{\Psi} (M_u), C> 0$ 
and $\epsilon_{\Psi} >0$ such that, whenever $\norm{\omega_j}_{\Omega_j} := \max\{ \norm{U_j}_{\Omega_j^U} , \norm{K_j}_{\Omega_j^K} \} \leq \epsilon_{\Psi}$, 
\begin{itemize}
\item[\textnormal{(1)}] $\norm{\Psi_j}_{h, T_j^{\Psi}} \leq C \norm{\omega_j}_{\Omega_j}$; 
%\item[\textnormal{(1)}] $\norm{\Psi_j}_{h, T_j^{\Psi}} \leq C \max\{ \norm{U_j}_{\Omega_j^U} , \norm{K_j}_{\Omega_j^K} \}$;
\item[\textnormal{(2)}] for $X\in \cS_j$, $\norm{\Eplus [ \Psi_j (X, \varphi'+ \zeta) - \hat{\Psi}_{j, 0} (X, \zeta)  ] }_{h, T_{j+1} (X, \varphi')} \leq A^{-|X|_j} \alphaLoc^{\Psi} \norm{\Psi_j}_{h, T_j^{\Psi}} G_{j+1} (\bar{X}, \varphi')$
\end{itemize}
where $\alphaLoc^{\Psi} = CL^{-1}(\log L)^{1/2}+ C\min\ha{1,\sum_{q\geq 1} e^{\sqrt{\beta}qh} e^{-(q-1/2)r\beta\Gamma_{j+1}(0)}}$ and $\hat{\Psi}_{j, 0}$ is the charge-0 term of $\Psi_j$.
\end{lemma}
\begin{proof}
To prove (1), we first notice that by \cite[Lemma~\ref{lem:W_norm}]{dgauss1} (whose assumptions are satisfied by the assumptions of this lemma)
and \eqref{eq:Psi-regulator_key_property1}, for $t \in \{0,1\}$,
%\begin{align}
%& \norm{U_j (B, \varphi + tu_j)}_{h, T_j (B, \varphi)} \leq C \norm{U_j}_{\Omega_j^U} w_j (B, \varphi + t u_j)^2 , \quad B \in \cB_j \\
%& \norm{K_j (X, \varphi + tu_j)}_{h, T_j (X, \varphi)} \leq C \norm{K_j}_{\Omega_j^K} G_j^{\Psi} (X, \varphi), \quad X \in \cP_j.
%\end{align}
\begin{align}
& \norm{U_j (B, \varphi + tu_j)}_{h, T_j (B, \varphi)} \leq C A^{-1} \norm{U_j}_{\Omega_j^U} w_j (B, \varphi + t u_j)^2 , \quad B \in \cB_j \\
& \norm{K_j (X, \varphi + tu_j)}_{h, T_j (X, \varphi)} \leq  A^{-|X|_j} \norm{K_j}_{\Omega_j^K} G_j^{\Psi} (X, \varphi), \quad X \in \cP_j.
\end{align}
%Also, using the submultiplicativity of the $\| \cdot \|_{h,T_j(B,\varphi)}$-(semi-)norm 
%and 
Also, using
$\norm{e^{F} - 1}_{h, T_j (B, \varphi)} \leq \norm{F}_{h, T_j (B, \varphi)} e^{\norm{F}_{h, T_j (B, \varphi)}}$,
\begin{align}
	& \norm{e^{U_j (B, \varphi + u_j)} - e^{U_j (B, \varphi)} }_{h, T_j (B, \varphi)} 
%	\leq  
%	\norm{e^{U_j (B, \varphi)} }_{h, T_j (B, \varphi)} 	
%	\norm{e^{U_j (B, \varphi + u_j) - U_j (B, \varphi)} - 1 }_{h, T_j (B, \varphi)} 
	\nnb
	& \qquad
	\leq C A^{-1} \norm{U_j}_{\Omega_j^U} \max_{t\in \{0,1\}} w_j (B,  \varphi + tu_j) e^{C A^{-1} \norm{U_j}_{\Omega_j^U} \max_{t\in \{0,1\}} w_j (B,  \varphi + tu_j)}
\end{align}
Using the submultiplicativity of the $\| \cdot \|_{h,T_j(B,\varphi)}$-norm to bound the powers of $e^{U_j (B, \varphi + u_j)} - e^{U_j (B, \varphi)}$
and Proposition~\ref{prop:key_prop_G_j^Psi} (3), it follows that
%\begin{align}
%\frac{ \norm{ \Psi_j (X, \varphi) }_{h, T_j (X, \varphi)}  }{G_j^{\Psi} (X, \varphi) } \leq C \norm{\omega_j}_{\Omega_j} \sum_{Y\in \cP_j (X)} A^{-|X|_j} = C \norm{\omega_j}_{\Omega_j} (A/2)^{-|X|_j} \label{eq:Psi_j^2_bound}
%\end{align}
%\commentjp{
%If we only had $\norm{U_j (B, \varphi + tu_j)}_{h, T_j (B, \varphi)} \leq C \norm{U_j}_{\Omega_j^U} w_j (B, \varphi + t u_j)^2$, this is not okay.
%For example, in the case $|X|_j =1$, then $\Psi_j (X) = e^{U_j (X, \varphi + u_j)} - e^{U_j (X, \varphi)} + K_j (X, \varphi + u_j) - K_j (X, \varphi)$. So $C$ on the RHS of \eqref{eq:Psi_j^2_bound} should have $A$ in it.}
%Hence it follows from Proposition~\ref{prop:key_prop_G_j^Psi} (3) that
\begin{align}
\frac{ \norm{ \Psi_j (X, \varphi) }_{h, T_j (X, \varphi)}  }{G_j^{\Psi} (X, \varphi) } \leq  \sum_{Y\in \cP_j (X)} (C \norm{\omega_j}_{\Omega_j} )^{|X \backslash Y|_j + |\operatorname{Comp}_j (Y)|} A^{-|X|_j} \leq C' \norm{\omega_j}_{\Omega_j} (A/2)^{-|X|_j} 
\label{eq:Psi_j^2_bound}
\end{align}
whenever $\norm{\omega}_{\Omega_j}$ is sufficiently small.  
%\sout{(polynomially in $L$ and $A$). }
%
This proves (1).
To show (2), take $X\in \cS_j$ and recall that $\Psi_j$ is $2\pi/\sqrt{\beta}$-periodic to decompose
\begin{align}\label{eq:Psi_j-decomp}
\Psi_j (X, \varphi) = \sum_{q\in \Z} \hat{\Psi}_{j,q} (X, \varphi)
\end{align}
where $\hat{\Psi}_{j,q}$ is the charge-$q$ term of $\Psi_j$. 
Then apply \eqref{eq:contraction_of_charge_q_term_external_field} to bound $\Eplus [ \hat{\Psi}_{j,q} (X, \varphi + \zeta) ]$ for $|q|\geq 1$ and \eqref{eq:gaussian_contraction_external_field} to bound $\Eplus [ \hat{\Psi}_{j,0} (X, \varphi + \zeta) - \hat{\Psi}_{j,0} (X, \varphi) ]$.
\end{proof}

\section{The renormalisation group map with external field}
\label{sec:rg_generic_step_external_field}

To prove the infinite-volume scaling limit we need an extended version of the renormalisation group maps that admits
an external field at every scale. In this section we extend the (bulk) renormalisation group map from \cite[Section~\ref{sec:rg_generic_step}]{dgauss1} to allow for such an external field.
The starting point is the generalisation of the parametrisation of the partition function 
from \cite[\eqref{eq:general_RG_step2}]{dgauss1}
to take into account a local perturbation.
In accordance with \eqref{eq:reblocking_Z_with_Psi}, partition functions will now be parametrised as 
\begin{equation} 
  Z_j (\varphi , \Psi_j ; (\Psi_k)_{k<j} | \Lambda_N ) =
  e^{-E_{j} |\Lambda_N| + e_j }
  \sum_{X\in \mathcal{P}_j (\Lambda_N)} e^{ U_j (\Lambda \backslash X, \varphi)} \prod_{Y \in \textnormal{Comp}_j (X)} \big( K_j (Y, \varphi ; (\Psi_k)_{k < j}) + \Psi_j (Y, \varphi) \big) , \label{eq:general_RG_step2_external_field}
\end{equation}
%where $\Lambda =\Lambda_N$, 
and where $e_j$ is a scalar coupling constant (like $E_j$), 
but originating from a bounded number of blocks near the origin.
Then the renormalisation group flow corresponding to 
\begin{equation}
  Z_{j+1}(\varphi' , 0 ; (\Psi_k)_{k\leq j} | \Lambda_N ) = \Eplus Z_j(\varphi' +\zeta , \Psi_j ; (\Psi_k)_{k<j} | \Lambda_N ),
  \qquad (j<N-1),
  \label{eq:general_RG_step1_external_field}
\end{equation}
will be considered.
Here recall that $\Eplus = \E_{\Gamma_{j+1}}$ for $j\leq N-1$ and
$\Eplus=\E_{\Gamma_N^{\Lambda_N}}$ for the last step $j=N-1$.

\subsection{Renormalisation group flow without external field}

When $\Psi_k \equiv 0$ for each $k<j$, then we will just denote $K_j (\cdot ; (\Psi_k)_{k<j})$ by $K_j (\cdot ; 0)$; this corresponds to the setting of \cite{dgauss1}.
Here we briefly recall the main estimates for the renormalisation group map in this setting from \cite[Sections 7 and 8]{dgauss1}.
This maps acts on the coupling constants $E_j \in \R$, 
$U_j$ of the form \eqref{eq:U_j_form}, and $K_j (\cdot ; 0)$ from  \cite[Sections~7 and~8]{dgauss1}.
In particular, $U_j$ can be identified with its coupling constants $s_j$ and $z_j  = (z_j^{(q)})_{q\geq 1}$.

Also, we use the abbreviations $\omega_j = (U_j, K_j)$ and $\norm{\omega_j}_{\Omega_j} = \max\{ \norm{U_j}_{\Omega_j^U}, \norm{K_j}_{\Omega_j^K} \}$,
where norms are still as in \cite[Definitions~\ref{def:U_space}--\ref{def:K_space}]{dgauss1} with the parameters 
they depend on always assumed to satisfy the conditions of Section~\ref{sec:parameters}.

The following theorem puts together \cite[Theorems~\ref{thm:H_j_E_j_estimate} and \ref{thm:local_part_of_K_j+1}]{dgauss1}
for $j+1 \leq N$
with its analogue \cite[Proposition~\ref{prop:final_RG_v2}]{dgauss1} for the last step $j+1=N$.

\begin{theorem} \label{thm:RG_without_external_field}
  Fix a finite-range step distribution $J$ as in Theorem~\ref{thm:highbeta-Z2}.
  There exist 
  $\epsilon_{nl} (\beta, A, L)$ such that the following holds for 
  $\norm{\omega_j}_{\Omega_j} \leq \epsilon_{nl}$.
%  $\norm{\omega_j}_{\Omega_j} := \norm{(U_j, K_j)}_{\Omega_j} \leq \epsilon_{nl}$.
For all $N$ and $0\leq j \leq N-1$, there is a map 
\begin{equation} \label{e:Phi-exist}
  \Phi^{\Lambda_N}_{j+1} = (\cE_{j+1}, \mathfrak{s}_{j+1}, \mathfrak{z}_{j+1}, \cK_{j+1}) : (E_{j}, s_j, z_j, K_j (\cdot ; 0)) \mapsto (E_{j+1}, s_{j+1}, z_{j+1}, K_{j+1} (\cdot ; 0)) , 
\end{equation}
such that  \eqref{eq:general_RG_step2_external_field}, \eqref{eq:general_RG_step1_external_field} hold with $e_j \equiv 0$ and $\Psi_k \equiv 0$, and $Z_0$ given by \eqref{eq:Z_0_definition-bis}.
The maps $\cE_{j+1} - E_j$, $\mathfrak{s}_{j+1}$, $\mathfrak{z}_{j+1}$, and $\cK_{j+1}$ are functions of $\omega_j$ satisfying
\begin{align}
& |\mathfrak{s}_{j+1} (\omega_j) - s_j | \leq C A^{-1} \norm{\omega_j}_{\Omega_j} , \\
& | (\cE_{j+1} - E_j ) (\omega_j) + s_j \nabla^{(e_1, -e_1)} \Gamma_{j+1} (0) | \leq C A^{-1} L^{-2j} \norm{\omega_j}_{\Omega_j} ,\\
& \mathfrak{z}_{j+1} (\omega_j) = L^2 e^{-\frac12 \beta q^2 \Gamma_{j+1}(0)} z_j,
\end{align}
for some $C>0$ and there exists $\epsilon_{nl} > 0$ such that whenever $\norm{\omega_j}_{\Omega_j} \leq \epsilon_{nl}$, $\cK_{j+1}$ is
continuously (Fr\'echet-)differentiable and admits a decomposition $\cK_{j+1} = \cL_{j+1} + \cM_{j+1}$ satisfying the estimates
\begin{align}
& \norm{\cL_{j+1} (\omega_j)}_{ \Omega_{j+1}^K} \leq C_1 L^2 \alpha_{\Loc} \norm{\omega_j}_{\Omega_j} \label{eq:bound_for_L_j_K_j_wo_external_field} \\
& \norm{\cM_{j+1} (\omega_j)}_{ \Omega_{j+1}^K} \leq C_2 (\beta, A, L)  \norm{\omega_j}_{\Omega_j}^2  \\
& \norm{D \cM_{j+1} (\omega_j)}_{ \Omega_{j+1}^K} \leq C_2 (\beta, A, L) \norm{\omega_j}_{\Omega_j}
\end{align}
for some $C_1, C_2( \beta,  A, L) >0$, where $\cL_{j+1}$ is linear in $\omega_j$ and
\begin{equation} 
\alphaLoc = CL^{-3}(\log L)^{3/2}+ C\min \Big\{1,\sum_{q\geq 1} e^{\sqrt{\beta}qh} e^{-(q-1/2)r\beta\Gamma_{j+1}(0)} \Big\}.
\end{equation}
\end{theorem}

The next theorem concerns the
existence of initial conditions independent of $N$ such that the renormalisation group
flow exists for all $N$, i.e., that for all $N \geq 1$ and all $j\leq N-1$,
\begin{equation}
  (E_{j+1}, s_{j+1},z_{j+1},K_{j+1}) = \Phi^{\Lambda_N}_{j+1}(E_j, s_j,z_j,K_j)
  \label{eq:abstract_RG_map-Z2}
\end{equation}
such that $\norm{(U_j, K_j)}_{\Omega_j} < \epsilon_{nl}$ for each $j \leq N$.
With $j \leq N-1$ instead of $j\leq N$, the theorem is exactly
\cite[Proposition~\ref{prop:stable_manifold}]{dgauss1},
and the bounds \eqref{eq:finalbounds} for $j=N$ follows from the bounds with $j=N-1$ by a single application of Theorem~\ref{thm:RG_without_external_field}.

\begin{theorem}
\label{thm:tuning_s}
  For any finite-range step distribution $J$ as in Theorem~\ref{thm:highbeta-Z2},
  there exist $\beta_0(J)\in(0,\infty)$,
  $s_0^c(J,\beta) = O(e^{-\frac{1}{4}\gamma \beta})$, and $\alpha=\alpha(J,\beta)>0$,
  such that for $\beta\geq \beta_0(J)$ the solution to  the flow equation \eqref{eq:abstract_RG_map-Z2} with parameter $s=s_0^c(J,\beta)$
  and initial conditions $s_0=s_0^c(J,\beta)$, $z_0 = \tilde z(\beta)$ as in \cite[Lemma~\ref{lemma:Fourier_repn_of_V}]{dgauss1}, and $K_0 (X)= 1_{X=\emptyset}$,  satisfies
  for all $j \leq N$ and $N> 1$,
  \begin{equation}\label{eq:finalbounds}
    \|U_j\|_{\Omega_j^U} \leq O(e^{-\frac14 \gamma \beta}L^{-\alpha j}), \qquad \|K_j\|_{\Omega_j^K} \leq O(e^{-\frac14 \gamma \beta}L^{-\alpha j}).
  \end{equation}
  In fact, one can take any $\alpha>0$ such that $C L^2 \alphaLoc \leq L^{-\alpha}$ for sufficiently large $C$.
\end{theorem}

\subsection{Extended coordinates}

{We next define the extended renormalisation group coordinates that incorporate a perturbation $\Psi$.
First,} recall the definition of polymer activities
from \cite[Definition~\ref{def:polymeractivity}]{dgauss1} and the definition of (bulk) renormalisation
group coordinates $K_j$ from \cite[Definition~\ref{def:K_space}]{dgauss1}.
{The extended version of the $K$-coordinate is then defined as follows.}

\begin{definition} 
  The coordinate $\KK_j = ( K_j(\cdot ; 0),  K_j(\cdot ; (\Psi_k)_{k<j} ) ) $ is a 
  pair of $2\pi/\sqrt{\beta}$-periodic polymer activities such that $K_j (X, \varphi ; 0)$ is even and
  invariant under the lattice symmetries.
  For pairs of such polymer activities, define
  \begin{equation}
    \| \KK_j \|_{\Omega_j^{\KK}} = \max \{ \norm{K_j (\cdot ; 0)}_{h, T_j}, \norm{K_j (\cdot ; (\Psi_k)_{k < j} )}_{h, T_j} \} .
\end{equation}
Let $\Omega_j^{\KK}$ be the Banach space (cf.~\cite[Appendix~\ref{app:completeness}]{dgauss1}) of such pairs where the maximum is finite.
\end{definition}

We also need a new definition of the product space of $(U_j,\KK_j, \Psi_j)$ as follows.

\begin{definition} [Extended coordinates] \label{def:extended_Omega_j}
Define the normed space 
of polymer activity perturbations based at the origin by
%\begin{align}
%\Omega_j^{\Psi} = \{ \Psi_j \text{ is $2\pi/\sqrt{\beta}$-periodic} \, : \norm{\Psi_j}_{h, T_j^{\Psi}} < \infty, \; \Psi_j (X) = 0 \textnormal{ if } 0\not\in X \}
%\end{align}
\begin{align}
\Omega_j^{\Psi} = \{ \Psi_j \text{ is $2\pi/\sqrt{\beta}$-periodic} \, : \norm{\Psi_j}_{h, T_j^{\Psi}} < \infty, \; \Psi_j (X) = 0 \textnormal{ if } 0\not\in X^* \}
\end{align}
equipped with the norm $\norm{\cdot}_{\Omega_j^{\Psi}} = \norm{\cdot}_{h, T_j^{\Psi}}$.
Also let $\bar{\Omega}_j = \Omega_j^U \times \Omega_j^{\vec{K}} \times \Omega_j^{\Psi}$, i.e.,
\begin{equation} \label{eq:final-norm}
  \bar{\Omega}_j = \{ \omega_j = (U_j, \KK_j, \Psi_j ) : \norm{\omega_j}_{\bar{\Omega}_j} < + \infty \},
  \quad
  \norm{\omega_j}_{\bar{\Omega}_j} = \max\{ \norm{U_j}_{\Omega_j^U},  \norm{\KK_j}_{\Omega_j^{\vec{K}}}, \norm{\Psi_j}_{\Omega_j^{\Psi}} \} .
\end{equation}
Given $\epsilon_{\Psi}, C_{\Psi} >0$
also define $\cY_j \equiv \cY_j (\epsilon_{\Psi}, C_{\Psi}) \subset \bar{\Omega}_j$ be the closed subset defined by the conditions
\begin{itemize}
\item[(1)] $K_j (X, \varphi; 0) = K_j (X, \varphi ; (\Psi_k)_{k<j})$ if $0\not\in X^*$; 
\item[(2)] $\norm{\Psi_j}_{\Omega_j^{\Psi}} \leq C_{\Psi} \, \max\{ \norm{U_j}_{\Omega_j^U},  \norm{\KK_j}_{\Omega_j^{\vec{K}}} \}$ 
and $\norm{\omega_j}_{\bar{\Omega}_j} \leq \epsilon_{\Psi}$ ;
\item[(3)] For $X\in \cS_j$, 
\begin{equation}
\norm{\Eplus [ {\Psi}_{j} (X, \cdot + \zeta) - \hat{\Psi}_{j,0} (X, \zeta) ] }_{h, T_{j+1} (\bar{X})} \leq C_{\Psi} A^{-|X|_j} \alphaLoc^{\Psi} \norm{ \omega_j }_{\bar{\Omega}_j}  \label{eq:Psi_j_condition}
\end{equation}
with $\alphaLoc^{\Psi}$ as defined below Lemma~\ref{lemma:G_j^Psi_is_Psi_regulator} (2).
\end{itemize}
\end{definition}

In particular, if we define $\Psi_j = \cF_{\Psi} [u_j, K_j(\cdot ; (\Psi_k)_{k<j}), U_j ; j ]$ for given $U_j$, $\KK_j$, 
then, if their assumptions are satisfied, Proposition~\ref{prop:reblocking_Z_with_Psi} and Lemma~\ref{lemma:G_j^Psi_is_Psi_regulator} imply
$(U_j, K_j (\cdot ; 0), K_j (\cdot ; (\Psi_k)_{k<j}), \Psi_j) \in \cY_j (\epsilon_{\Psi}, C_{\Psi})$ for some $\epsilon_{\Psi}, C_{\Psi} >0$ whenever $\norm{U_j}_{\Omega_j^U}, \norm{\KK_j}_{\Omega_j^{\KK}} \leq \epsilon_{\Psi}$.

\subsection{Definition of the extended renormalisation group map}

We will now introduce the extended renormalisation group map with the extra coordinates $K_j (\cdot ; (\Psi_k)_{k<j})$ and $\Psi_j$, which we denote by
\begin{equation}
  \bar{\Phi}_{j+1} : (E_j, e_j, s_j,  z_j, \KK_j, \Psi_j) \mapsto (E_{j+1}, e_{j+1}, s_{j+1},  z_{j+1}, \KK_{j+1}, 0)
\end{equation}
(cf. \cite[\eqref{eq:Phi_j+1_definition}]{dgauss1} for $\Phi_{j+1}$ which we now call the bulk part of the renormalisation group map);
{here the $e_j$ are scalar coupling constants taking the role for the perturbation due to the external field that the $E_j$ have for the bulk part of the renormalisation group map. In analogy with}
$\Phi_{j+1}$, we will also denote the components of the map $\bar \Phi_{j+1}$ by $(\cE_{j+1}, \mathfrak{e}_{j+1}, \cU_{j+1}, \cK^0_{j+1}, \cK^{\Psi}_{j+1})$ 
{and require that}
\begin{equation} \label{eq:energy-renorm}
(\mathfrak{e}_{j+1} - \cE_{j+1} |\Lambda|)(E_j, e_j, \cdot) = (\mathfrak{e}_{j+1} - \cE_{j+1} |\Lambda|) (0, 0, \cdot) + e_{j} - E_{j} |\Lambda|.
\end{equation}
{The last condition can be imposed because} the scalar prefactor $e^{-E_j|\Lambda| + e_j}$ appearing in $Z_j$ (see \eqref{eq:general_RG_step2_external_field}) is mapped to the corresponding quantity at scale $j+1$ and hence does not contribute to the dynamics, {see the discussion below \cite[\eqref{eq:Phi_j+1_definition}]{dgauss1}
for the bulk case.}
Moreover, when we write $\mathfrak{e}_{j+1}$, $\cE_{j+1}$ without $e_j, E_j$ as their arguments, they are just $\mathfrak{e}_{j+1} (0, 0,\cdot)$ and $\cE_{j+1} (0,0, \cdot)$ respectively.

We are thinking of $\bar{\Phi}_{j+1}$ as $\Phi_{j+1}$ with a perturbation, which entails that the $\cE_{j+1}$, $\cU_{j+1}$ and $\cK_{j+1}^0$ will be given as in \cite[Section~\ref{sec:rg_generic_step}]{dgauss1},
i.e., by Definitions~\ref{def:evolution_of_U} and~\ref{def:evolution_of_remainder} in that paper respectively.
The other coordinates $\mathfrak{e}_{j+1}$ and $\cK_{j+1}^{\Psi}$ are defined explicitly as follows.
{The definition of $\cK_{j+1}^\Psi$ is almost the same as that of $\cK_{j+1}$ except for the perturbed activity $\Psi_j$ and the one-point energy $\mathfrak{e}_{j+1}$ arising from it.
}
\begin{definition}
\label{def:evolution_of_e_j_K_j}
For $0\leq j\leq N-1$, let $\zeta$ be the centred Gaussian random variable with covariance $\Gamma_{j+1}$ if $j\leq N-2$ and $\Gamma_N^{\Lambda_N}$ if $j= N-1$. Then for each $Y \in \cP_j$,
define the map $(\KK_j, \Psi_{j}) \mapsto \mathfrak{e}_{j+1} (\KK_j, \Psi_j)$ by 
%\begin{equation}
%  \mathfrak{e}_{j+1} (\KK_j, \Psi_j) = \sum_{0\in X \in \cS_j} \Eplus [ 
%\hat{\Psi}_{j,0} (X, \zeta) + \hat{K}_{j, 0} (X, \zeta ; (\Psi_k)_{k<j}) - \hat{K}_{j, 0} (X, \zeta ; 0) ] 
%. \label{eq:e_j+1_definition}
%\end{equation}
\begin{equation}
	\mathfrak{e}'_{j+1} (Y,\KK_j, \Psi_j) = \sum_{B\in \cB_j (B_0^* \cap Y)} \sum_{Z \in \cS_j}^{B\subset Z} \frac{1}{|Z \cap B_0^*|_j} \Eplus [ 
\hat{\Psi}_{j,0} (Z, \zeta) + \hat{K}_{j, 0} (Z, \zeta ; (\Psi_k)_{k<j}) - \hat{K}_{j, 0} (Z, \zeta ; 0) ] 
	\label{eq:e_j+1_definition}
\end{equation}
where we recall that $B_0$ is the unique $j$-block such that $0\in B_0$ and let
\begin{align}
\mathfrak{e}_{j+1} (\KK_j, \Psi_j) = \mathfrak{e}'_{j+1} (\Lambda_N,\KK_j, \Psi_j)
.
\end{align}
The map $(U_j, \KK_j, \Psi_j )\mapsto K_{j+1}^{\Psi}$ is defined by
%\begin{align}\label{eq:expression_for_K_j+1_external_field}
%  \mathcal{K}_{j+1}^{\Psi} & (U_j, \KK_j,  \Psi_j, X )
%  = \sum_{X_0, X_1, Z, (B_{Z''})}^{*} e^{ \cE_{j+1} |T| - 1_{0\in T} \mathfrak{e}_{j+1} } e^{\cU_{j+1} (X \backslash T)} \\\nonumber
%  & \times \Eplus \Big[ (e^{U_j} - e^{ -\cE_{j+1} |B| + 1_{0\in B} \mathfrak{e}_{j+1} + \cU_{j+1}})^{X_0} (\bar{K}^{\Psi}_j - \mathcal{E}^{\Psi} K_j)^{[X_1]}  \Big] \prod_{Z'' \in \operatorname{Comp}_{j+1} (Z)} J^{\Psi}_j (B_{Z''}, Z''), 
%\end{align}
\begin{align}\label{eq:expression_for_K_j+1_external_field}
  \mathcal{K}_{j+1}^{\Psi} & (U_j, \KK_j,  \Psi_j, X )
  = \sum_{X_0, X_1, Z, (B_{Z''})}^{*} e^{ \cE_{j+1} |T| - \mathfrak{e}_{j+1} (T) } e^{\cU_{j+1} (X \backslash T)} \\\nonumber
  & \times \Eplus \Big[ (e^{U_j} - e^{ -\cE_{j+1} |B| +  \mathfrak{e}'_{j+1} (B) + \cU_{j+1}})^{X_0} (\bar{K}^{\Psi}_j - \mathcal{E}^{\Psi} K_j)^{[X_1]}  \Big] \prod_{Z'' \in \operatorname{Comp}_{j+1} (Z)} J^{\Psi}_j (B_{Z''}, Z''), 
\end{align}
where the polymer powers follow the convention \cite[\eqref{eq:polymer_power_convention1}, \eqref{eq:polymer_power_convention2}]{dgauss1},  the summation $*$ is running over disjoint $(j+1)$-polymers $X_0, X_1, Z$ such that
$X_1 \not\sim Z$,  $B_{Z''} \in \cB_{j+1} (Z'')$ for each $Z'' \in \operatorname{Comp}_{j+1} (Z)$,
$T = X_0 \cup X_1 \cup Z$ and $X = \cup_{Z''}B^*_{Z''} \cup X_0 \cup X_1$,  and
\begin{align}
\mathcal{E}^{\Psi} K_j ( X, \varphi' ) &= \sum_{B\in \mathcal{B}_{j+1} (X)} J^{\Psi}_j (B, X, \varphi') \label{eq:cEK^Psi_definition} \\
%Q^{\Psi}_j ( D, Y, \varphi') &= 1_{Y\in \cS_j} \Big( \Loc_{Y, D} \Eplus [  K_{j} ( Y, \varphi' + \zeta ; 0 ) ]  \\
%& \qquad \qquad + { 1_{0\in D} \Eplus [  \hat{\Psi}_{j,0} (Y, \zeta)  +  \hat{K}_{j, 0} (Y, \zeta ; (\Psi_k)_{k<j}) - \hat{K}_{j, 0} (Y, \zeta ; 0)  ] } \Big)
 Q^{\Psi}_j ( D, Y, \varphi') & = 1_{Y\in \cS_j} \Big( \Loc_{Y, D} \Eplus [  K_{j} ( Y, \varphi' + \zeta ; 0 ) ] 
 \\
& 
\qquad \qquad + \frac{ 1_{D \subset \cB_j (B_0^* \cap Y )} }{|Y \cap B_0^*|_j} \Eplus [  \hat{\Psi}_{j,0} (Y, \zeta)  +  \hat{K}_{j, 0} (Y, \zeta ; (\Psi_k)_{k<j}) - \hat{K}_{j, 0} (Y, \zeta ; 0)  ]  \Big) 
\nnb
J^{\Psi}_j (B, X, \varphi')  &= 1_{B\in \cB_{j+1} (X)} \sum_{D\in \cB_j (B)} \sum_{Y\in \cS_j}^{D\in \cB_j (Y)} Q_j^{\Psi} (D, Y, \varphi') ( 1_{\bar{Y} = X} - 1_{B=X} ) . \label{eq:J_j^Psi_definition} \\
\bar{K}_j^{\Psi} (X, \varphi' + \zeta) &= \sum_{Y\in \mathcal{P}_j}^{\bar{Y}=X} e^{U_j (X \backslash Y, \varphi' + \zeta)} \Big( { ( K_j (Y, \varphi' + \zeta ; (\Psi_k)_{k<j})  + \Psi_j  (Y, \varphi' + \zeta) ) } \Big)  \label{eq:K_bar^Psi_definition}
\end{align}

for $D\in \cB_j$, $B\in \cB_{j+1}$, $Y\in \cP_j$ and $X\in \cP_{j+1}$. 
\end{definition}

Note that each $(j+1)$-block $B_{Z''}$ appearing in the summation defining $\mathcal{K}_{j+1}^{\Psi}$ is such that $Z'' \in B_{Z''}^*$ since $J^{\Psi}_j (B_{Z''}, Z'', \varphi')$ vanishes whenever $Z'' \notin \cS_{j+1}$.

In the remainder of the argument, we will focus on the case $j \leq N-2$, and hence $\zeta \sim \cN (0, \Gamma_{j+1})$. The argument is identical for the case $j\leq N-1$ because $\Gamma_N^{\Lambda_N}$ satisfies the same estimates as $\Gamma_N$.

The next theorem is the extension of \cite[Theorem~\ref{thm:general_RG_step_consistent}]{dgauss1}
with essentially the same proof; see Appendix~\ref{app:reblocking} for the proof.
It shows that $Z_{j+1}$ defined by the map $\bar\Phi_{j+1}$ is indeed the desired partition function of scale $j+1$.

\begin{theorem} \label{thm:general_RG_step_consistent_external_field}
Let $Z_j (\varphi, \Psi_j ; (\Psi_k)_{k<j} | \Lambda)$ and $Z_{j+1} (\varphi', 0 ; (\Psi_k)_{k\leq j} | \Lambda)$ be defined by \eqref{eq:general_RG_step2_external_field} with coordinates $(E_j, e_j, U_j, \KK_j, \Psi_j)$ and $(E_{j+1}, e_{j+1}, U_{j+1} \KK_{j+1}, 0) = \bar{\Phi}_{j+1} (E_j, e_j, U_j, \KK_j, \Psi_j)$ respectively.
Then they satisfy
\eqref{eq:general_RG_step1_external_field} (and \eqref{eq:energy-renorm} holds).
\end{theorem}

\subsection{Estimates for the extended renormalisation group map}

Since we have already established estimates on
the bulk components $\cE_{j+1}$, $\cU_{j+1}$, and $\cK_j^0 \equiv \cK_j (\cdot ; 0)$
of the renormalisation group map in \cite[Theorems~\ref{thm:H_j_E_j_estimate} and~\ref{thm:local_part_of_K_j+1}]{dgauss1}, we only need additional estimates for $\mathfrak{e}_{j+1}$ and $\cK^{\Psi}_{j+1}$. Since we will not need a stable manifold theorem to tune parameters, a cruder control of these suffices.

\begin{theorem}
\label{thm:e_j_estimate}
Let $(u_j)_j$ satisfy \ref{assump:u} and 
the parameters be as in Section~\ref{sec:parameters}.
If $(U_j, \KK_j, \Psi_j) \in \cY_j (\epsilon, C_{\Psi})$, for some $\varepsilon > 0$
and 
%$C_{\Psi} = C_{\Psi}(M_u)$ 
$C_{\Psi}$ as given by Proposition~\ref{prop:key_prop_G_j^Psi},
%\begin{equation}
%  | \mathfrak{e}_{j+1} (\Psi_j, \KK_j)  | \leq C C_{\Psi} A^{-1} \norm{\omega_j}_{\bar{\Omega}_j} .  \label{eq:e_j+1_bound} 
%\end{equation}
\begin{equation}
  | \mathfrak{e}'_{j+1} (B,\Psi_j, \KK_j)  | \leq  C C_{\Psi} A^{-1} \norm{\omega_j}_{\bar{\Omega}_j},
  \quad B\in \cB_j
   .  \label{eq:e_j+1_bound} 
\end{equation}
\end{theorem}
\begin{proof}
Let $X \in \cS_j$ be such that $0 \in X^*$ and $B \in \cB_j (X)$.
By \eqref{eq:Psi_j-decomp} and the definition of $\norm{\cdot}_{\Omega^\Psi_j} (= \norm{\cdot}_{h, T_j^{\Psi}})$, see \eqref{eq:Omega_j^Psi_norm_def_X}-\eqref{eq:Omega_j^Psi_norm_def},
\begin{equation}
	\big| \Eplus [ \hat{\Psi}_{j, 0} (X, \zeta) ] \big| \leq (A/2)^{-|X|_j} \norm{\Psi_j}_{\Omega_j^{\Psi}} \Eplus [ G_j^{\Psi} (X, \zeta) ]
\end{equation}
and by the assumption $(U_j, \KK_j, \Psi_j) \in \cY_j$, 
we also have $\norm{\Psi_j}_{\Omega_j^{\Psi}} \leq C_{\Psi} \norm{\omega_j}_{\bar{\Omega}_j}$.
Similarly, 
\begin{align}
	\big| \Eplus [ \hat{K}_{j, 0} (X, \zeta ; (\Psi_k)_{k<j} ) - \hat{K}_{j, 0} (X, \zeta ; 0) ] \big| \leq 2 A^{-|X|_j} \norm{\KK_j}_{\Omega_j^{\vec{K}}} \Eplus [G_j (X, \zeta)]
\end{align}
and note that $\norm{\KK_j}_{\Omega_j^{\vec{K}}} \leq \norm{\omega_j}_{\bar{\Omega}_j} $ by definition, see \eqref{eq:final-norm}.
By Proposition~\ref{prop:key_prop_G_j^Psi} and since $|X|_j \leq 4$ for $X \in \mathcal{S}_j$, we have that
\begin{align}
\Eplus [G_j (X, \zeta)] \leq \Eplus [G_j^{\Psi} (X, \zeta)] \leq C_{\Psi} 2^{|X|_j} \leq 16 C_{\Psi} .
\end{align}
Hence, by definition of $\mathfrak{e}'_{j+1}$ in \eqref{eq:e_j+1_definition},
we obtain \eqref{eq:e_j+1_bound}.
\end{proof}

\begin{theorem}[Estimate for remainder coordinate] 
\label{thm:local_part_of_K_j+1_external_field}
Let $0\leq j \leq N-1$ and the parameters be as in Section~\ref{sec:parameters}. 
Further assume \ref{assump:u} to hold and let $C_{\Psi}$ be given by Proposition~\ref{prop:key_prop_G_j^Psi}.
%$C_{\Psi} = C_{\Psi} (M_u)$.
Then the map $\cK^{\Psi}_{j+1}(U_j, \KK_j, \Psi_j)$ admits a decomposition
\begin{equation}\label{eq:decomp}
\cK^{\Psi}_{j+1}(U_j, \KK_j, \Psi_j)
= \cL^{\Psi}_{j+1} (\KK_j, \Psi_j) + \cM^{\Psi}_{j+1} (U_j, \KK_j, \Psi_j)
\end{equation}
such that the following estimates hold: the map $\cL^{\Psi}_{j+1}$ is linear in $(\KK_j, \Psi_j)$ and
there exist
$L'_0$, $A'_0(L)$,
$\tilde{\epsilon}_{nl} \equiv \tilde{\epsilon}_{nl} (\beta, A, L,  C_{\Psi} ) >0$ (only polynomially small in its arguments), 
$C_1 >0$ independent of $A$ and $L$ and 
$C_2=C_2 (\beta, A, L, C_{\Psi}) >0$ (only polynomially large in its arguments)
such that for 
$L \geq L'_0$, $A \geq A'_0 (L)$,
$\omega_j =(U_j, \KK_j, \Psi_j) \in \cY_j (\tilde{\epsilon}_{nl}, C_{\Psi})$,
\begin{equation}
\norm{\mathcal{L}_{j+1}^{\Psi} (\KK_{j}, \Psi_j )}_{ \Omega_{j+1}^\Psi} \leq C_1 C_{\Psi}  \big( L^2\alphaLoc  \norm{K_j (\cdot ; 0)}_{\Omega_j^K} + \alphaLoc^{\Psi} \norm{\omega_j}_{\bar{\Omega}_j}  \big) \label{eq:bound_for_L_j_K_j_external_field}
, \quad 
\end{equation}
with $\alphaLoc^{\Psi}$ from Lemma~\ref{lemma:G_j^Psi_is_Psi_regulator},
and $\cM^{\Psi}_{j+1} (U_j,\KK_j, \Psi_j)$ is continuously
Fr\'echet-differentiable with
\begin{align}
& \norm{{\cM}^{\Psi}_{j+1} (\omega_j) }_{ \Omega_{j+1}^\Psi} \leq C_2 (\beta,  A, L, C_{\Psi}) \norm{\omega_j}_{ \bar{\Omega}_j}^2 \label{eq:bound_for_N_j_K_j_external_field} \\
& \norm{D {\cM}^{\Psi}_{j+1}  (\omega_j) }_{ \Omega_{j+1}^\Psi} \leq C_2 (\beta,   A, L, C_{\Psi}) \norm{\omega_j}_{ \bar{\Omega}_j} \label{eq:bound_for_derivative_of_Nj_external_field} .
\end{align}
\end{theorem}

\subsection{Proof of Theorem~\ref{thm:local_part_of_K_j+1_external_field}: bound of linear part}
\label{subsec:the_expression_Lj_external_field}
We first introduce $\mathcal{L}_{j+1}^{\Psi} $.
Proceeding as in \cite[Section~\ref{subsec:the_expression_Lj}]{dgauss1},
we may write the terms linear in $U_j$, $\KK_j$ from \eqref{eq:expression_for_K_j+1_external_field} by keeping only the terms in \eqref{eq:expression_for_K_j+1_external_field} with
\begin{align}
\# (X_0, X_1, Z) := |X_0|_{j+1} + |\operatorname{Comp}_{j+1} (X_1)| + |\operatorname{Comp}_{j+1} (Z)| \leq 1
\end{align}
and replacing exponentials by their linear approximations. This linearisation process is identical to that of \cite[Section~\ref{subsec:the_expression_Lj}]{dgauss1}. For $X\in \cP_{j+1}^c$, this gives
%\begin{equation}
%\begin{split} \label{eq:L_j_K_j_external_field}
%&\mathcal{L}_{j+1}^{\Psi} ( X, \varphi') \\
%&= \sum_{Y : \bar{Y} = X} 1_{Y\in \cP_j^c} \Big( \Eplus [ {  K_j (Y, \zeta + \varphi' ; (\Psi_k)_{k\leq j-1})  + \Psi_j  (Y, \zeta + \varphi') } ]- 1_{Y\in \cS_j} \sum_{D\in \cB_j (Y)} Q_j^{\Psi} (D, Y, \varphi') \Big) \\
%&\quad + \sum_{D \in \mathcal{B}_{j}}^{\bar{D}=X} \Big( \Eplus [ U_j (D, \zeta + \varphi')] + \cE_{j+1} |D| - 1_{0\in D} \mathfrak{e}_{j+1} - \cU_{j+1} (D, \varphi') + \sum_{Y\in \cS_j}^{D\in \cB_j (Y)}  Q_j^{\Psi} (D, Y, \varphi') \Big) \\
%&=: \mathcal{L}^{(1)}_{j+1} (\KK_j) ( X, \varphi')  
%+ \mathcal{L}^{(2)}_{j+1} ( \Psi_j) ( X, \varphi') 
%+ \mathcal{L}^{(3)}_{j+1} (\KK_j) ( X, \varphi') 
%\end{split} 
%\end{equation}
\begin{equation}
\begin{split} \label{eq:L_j_K_j_external_field}
&\mathcal{L}_{j+1}^{\Psi} ( X, \varphi') \\
&= \sum_{Y : \bar{Y} = X} 1_{Y\in \cP_j^c} \Big( \Eplus [ {  K_j (Y, \zeta + \varphi' ; (\Psi_k)_{k\leq j-1})  + \Psi_j  (Y, \zeta + \varphi') } ]- 1_{Y\in \cS_j} \sum_{D\in \cB_j (Y)} Q_j^{\Psi} (D, Y, \varphi') \Big) \\
&\quad + \sum_{D \in \mathcal{B}_{j}}^{\bar{D}=X} \Big( \Eplus [ U_j (D, \zeta + \varphi')] + \cE_{j+1} |D| - \mathfrak{e}_{j+1}' (D) - \cU_{j+1} (D, \varphi') + \sum_{Y\in \cS_j}^{D\in \cB_j (Y)}  Q_j^{\Psi} (D, Y, \varphi') \Big) \\
&=: \mathcal{L}^{(1)}_{j+1} (\KK_j) ( X, \varphi')  
+ \mathcal{L}^{(2)}_{j+1} ( \Psi_j) ( X, \varphi') 
+ \mathcal{L}^{(3)}_{j+1} (\KK_j) ( X, \varphi') 
\end{split} 
\end{equation}
where, using the choice of $\cU_{j+1}$ and $\mathfrak{e}'_{j+1}$, see \cite[\eqref{eq:evolution_of_U}]{dgauss1} and \eqref{eq:e_j+1_definition}, respectively, we set
\begin{align}
& \mathcal{L}^{(1)}_{j+1} (\KK_j) ( X, \varphi') =  \sum_{Y : \bar{Y} = X} 1_{Y\in \cP_j^c}  \Eplus K_j ( Y, \zeta + \varphi' ; 0) - 1_{Y\in \cS_j} \Eplus [ \operatorname{Loc}_{Y} \hat{K}_{j, q=0} (Y,  \varphi' + \zeta  ; 0 )  ],\\
& \mathcal{L}^{(2)}_{j+1} (\Psi_j) ( X, \varphi') =  \sum_{Y : \bar{Y}=X} 1_{Y\in \cP_j^c}  \Eplus[ {\Psi_j  (Y, \varphi' + \zeta)  - 1_{Y\in \cS_j} \hat{\Psi}_{j,0} (Y, \zeta) } ], \\
& \mathcal{L}^{(3)}_{j+1} (\KK_j ) ( X, \varphi') =  \sum_{Y : \bar{Y}=X} 1_{Y\in \cP_j^c}  \E [ D_j (Y,  \varphi' + \zeta) - 1_{Y\in \cS_j} \hat{D}_{j,0} (Y, \zeta)  ]  
\end{align}
and  
\begin{align}
D_{j}  (Y, \varphi) := K_j (Y, \varphi ; (\Psi_k)_{k<j}) - K_j (Y, \varphi ; 0) .
\end{align}
In fact, the $\cL_{j+1}$ in Theorem~\ref{thm:RG_without_external_field} (see \cite[Section~7.4]{dgauss1}) is identical to $\cL_{j+1}^{(1)}$, 
i.e.,
\begin{equation}
\cL_{j+1} (K_j (\cdot ; 0)) = \cL^{(1)}_{j+1} (\vec{K}_j) \label{eq:L^1_j+1_is_same_as_cL_j+1} 
\end{equation}
and also $\cL^{\Psi}_{j+1}$ is a function of $(\KK_j, \Psi_j)$, not depending on $U_j$.

\begin{proof}[Proof of \eqref{eq:bound_for_L_j_K_j_external_field} of Theorem~\ref{thm:local_part_of_K_j+1_external_field}]
We will show that the bound \eqref{eq:bound_for_L_j_K_j_external_field} holds for any choice of $\tilde{\varepsilon}_{nl} \leq \varepsilon_{nl}$, where the latter refers to the (bulk) value supplied by Theorem~\ref{thm:RG_without_external_field}, see above \eqref{eq:bound_for_L_j_K_j_wo_external_field}. Thus, let $\omega_j =(U_j, \KK_j, \Psi_j) \in \cY_j (\tilde{\epsilon}_{nl}, C_{\Psi})$.
By \eqref{eq:bound_for_L_j_K_j_wo_external_field} and \eqref{eq:L^1_j+1_is_same_as_cL_j+1}, we already know that
\begin{equation}
\norm{\cL_{j+1}^{(1)} (\KK_j)}_{ \Omega^K_{j+1}} \leq C_1 L^2 \alphaLoc  \norm{K_j (\cdot ; 0)}_{\vec{\Omega}_j^K}.
\end{equation}
The estimate on $\cL_{j+1}^{(2)}$ follows from the decomposition
%\begin{equation}
%\cL_{j+1}^{(2)} (X, \varphi') = \sum_{0\in Y \in \cS_j}^{\bar{Y}=X}  \Eplus [ {\Psi}_{j} (Y, \varphi' + \zeta) - \hat{\Psi}_{j,0} (Y, \zeta) ]  + \mathbb{S} ( 1_{Y\not\in \cS_j} \Eplus [ \Psi_j (\cdot, \cdot+ \zeta) ] ) (X, \varphi') .
%\end{equation}
\begin{equation}
\cL_{j+1}^{(2)} (X, \varphi') =  \sum_{Y \in \cS_j,  \, 0 \in Y^*}^{\bar{Y}=X}   \Eplus [ {\Psi}_{j} (Y, \varphi' + \zeta) - \hat{\Psi}_{j,0} (Y, \zeta) ]  + \mathbb{S} ( 1_{Y\not\in \cS_j} \Eplus [ \Psi_j (\cdot, \cdot+ \zeta) ] ) (X, \varphi') .
\end{equation}
The summation is running over 
%\sout{$Y\ni 0$} 
$Y^* \ni 0$ now because of the assumption that $\Psi_j (Y, \varphi) = 0$ if 
%\sout{$0\not\in X$} 
$0\not\in Y^*$ (which is a part of the assumption $(\KK_j , \Psi_j) \in \cY_j (\tilde{\epsilon}_{nl}, C_{\Psi})$). Then the first term is bounded by
$C A^{-|X|_{j+1}} \alphaLoc^{\Psi} \norm{\omega_j}_{\bar{\Omega}_j}$
because of the assumption $( \KK_j, \Psi_j ) \in \cY_j (\epsilon_{\Psi}, C_{\Psi})$
and \eqref{eq:Psi_j_condition} implied by it (here we also used that $|\overline{Y}|_{j+1} \leq |Y|_{j}$).
The second term is bounded using Proposition~\ref{prop:largeset_contraction_external_field} with $* = \Psi$ with $L$ and $A = A(L)$ sufficiently large:
\begin{equation}
\norm{ \mathbb{S}\big[ 1_{Y\in \cP_j^c \backslash \cS_j} \Eplus[\Psi_j (\cdot, \cdot + \zeta)] \big]  }_{h, T^{\Psi}_{j+1} (X)} \leq (L^{-1} A^{-1})^{|X|_{j+1}} \norm{\Psi_j}_{\Omega_j^{\Psi}} \leq C \alphaLoc^{\Psi} A^{-|X|_{j+1}}  \norm{\Psi_j}_{\Omega_j^{\Psi}}
\end{equation}
Finally, we bound
%\begin{align}\label{eq:L-3}
%\cL_{j+1}^{(3)} (X, \varphi') = \sum_{0\in Y \in \cS_j}^{\bar{Y}=X}  \Eplus [ {D}_{j} (Y,  \varphi' + \zeta) - \hat{D}_{j,0} (Y, \zeta) ]   + \mathbb{S} ( 1_{Y\not\in  \cS_j} \Eplus [ D_j (\cdot, \cdot+ \zeta) ] ) (X, \varphi') .
%\end{align}
\begin{align}\label{eq:L-3}
	\cL_{j+1}^{(3)} (X, \varphi') = \sum_{Y \in \cS_j,  \,0 \in Y^*}^{\bar{Y}=X}  \Eplus [ {D}_{j} (Y,  \varphi' + \zeta) - \hat{D}_{j,0} (Y, \zeta) ]   + \mathbb{S} ( 1_{Y\not\in  \cS_j} \Eplus [ D_j (\cdot, \cdot+ \zeta) ] ) (X, \varphi') .
\end{align}
Again, the assumption $D_j (Y, \zeta + \varphi') = 0$ for 
%\sout{$Y \not\ni 0$} 
$Y^* \not\ni 0$ (which, as above, is a part of the assumption $(U_j, \KK_j, \Psi_j) \in \cY_j (\tilde{\epsilon}_{nl}, C_{\Psi})$) effectively restricts the sum in the first term to $Y^*  \ni 0$,
then Proposition~\ref{prop:contraction_estimates_external_field} with case $* = 0$ applies to give the bound $C A^{-|X|_{j+1}} \alphaLoc^{\Psi} \norm{\KK_j}_{\Omega_j^{\vec{K}}}$.
For the second term, Proposition~\ref{prop:largeset_contraction_external_field} with $*=0$ gives the bound same bound with the same choice of $L$ and $A$ as above.
\end{proof}

\subsection{Proof of Theorem~\ref{thm:local_part_of_K_j+1_external_field}: bound of non-linear part}
\label{subsec:Mj_bound_external_field}

Analogously as in \cite[Section~\ref{subsec:M_j+1_decomposition}]{dgauss1},
the non-linear part $\cM_{j+1}^{\Psi} := \cK^{\Psi}_{j+1}- \cL^{\Psi}_{j+1}$ (with $\cL^{\Psi}_{j+1}$ as defined by the first line of \eqref{eq:L_j_K_j_external_field}) can be decomposed into four parts,
\begin{align}
\mathcal{M}_{j+1}^{\Psi}  (U_j, \KK_j, X, \varphi') =  \sum_{k=1}^4 \MM_{j+1}^{\Psi, (k)} ( \mathfrak{K}_j^{\Psi} (\omega_j),  X, \varphi')
\end{align}
with
%\begin{align}
%\MM_{j+1}^{\Psi, (1)} ( \mathfrak{K}_j^{\Psi} (\omega_j),  X) 
%	&= \sum_{X_0, X_1, Z, (B_{Z''})}^{\ast} 1_{\# (X_0, X_1, Z)  \geq 2}  e^{\cE_{j+1} |X| - 1_{0 \in X} \mathfrak{e}_{j+1} } e^{  \bar{U}^{\Psi}_{j+1} (X \backslash T)} \nonumber \nnb
%	&\quad \times \Eplus \Big[ (e^{U_j} - e^{\bar{U}_{j+1}^{\Psi}  })^{X_0} (\bar{K}_j^{\Psi} - \mathcal{E}^{\Psi} K_j)^{[X_1]}  \Big] \prod_{Z'' \in \operatorname{Comp}_{j+1} (Z)} J^{\Psi}_j (B_{Z''}, Z'')   \label{eq:M^1_j+1_external_field}  \\
%\MM_{j+1}^{\Psi, (2)} ( \mathfrak{K}_j^{\Psi} (\omega_j),  X) 
%	&= \sum_{X_0, X_1, Z, (B_{Z''})}^{\ast} 1_{ \# (X_0, X_1, Z)  \leq 1 } (e^{\cE_{j+1} |X| - 1_{0 \in X} \mathfrak{e}_{j+1} } e^{\bar{U}^{\Psi}_{j+1} (X \backslash T)} - 1) \nonumber \nnb
%	&\quad \times \Eplus \Big[ (e^{U_j} - e^{\bar{U}^{\Psi}_{j+1}})^{X_0} (\bar{K}^{\Psi}_j - \mathcal{E}^{\Psi} K_j)^{[X_1]}  \Big] \prod_{Z'' \in \operatorname{Comp}_{j+1} (Z)} J^{\Psi}_j (B_{Z''}, Z'')  \label{eq:M^2_j+1_external_field}   \\
%\MM_{j+1}^{\Psi, (3)} ( \mathfrak{K}_j^{\Psi} (\omega_j),  X) 
%	&= \sum_{|X_0|_{j+1} = 1}^{X_0 = X} \Eplus \Big[ \Big( e^{U_j} - e^{ \bar{U}_{j+1}^{\Psi} } - U_j + \bar{U}_{j+1}^{\Psi} \Big)^{X_0} \Big]  \label{eq:M^3_j+1_external_field}  \\
%\MM_{j+1}^{\Psi, (4)} ( \mathfrak{K}_j^{\Psi} (\omega_j),  X) 
%	&= \Eplus \Big[ \sum_{Y\in \cP_j}^{\bar{Y} = X} e^{U_j (Y)} ( K_j (\cdot ; (\Psi_k)_{k<j}) + \Psi_j ) (X\backslash Y) - \mathbb{S} ( K_j +  \Psi_j ) (X) \Big] \label{eq:M^4_j+1_external_field} 
%\end{align}
\begin{align}
\MM_{j+1}^{\Psi, (1)} ( \mathfrak{K}_j^{\Psi} (\omega_j),  X) 
	&= \sum_{X_0, X_1, Z, (B_{Z''})}^{\ast} 1_{\# (X_0, X_1, Z)  \geq 2}  e^{\cE_{j+1} |X| - \mathfrak{e}'_{j+1} (X) } e^{  \bar{U}^{\Psi}_{j+1} (X \backslash T)} \nonumber \nnb
	&\quad \times \Eplus \Big[ (e^{U_j} - e^{\bar{U}_{j+1}^{\Psi}  })^{X_0} (\bar{K}_j^{\Psi} - \mathcal{E}^{\Psi} K_j)^{[X_1]}  \Big] \prod_{Z'' \in \operatorname{Comp}_{j+1} (Z)} J^{\Psi}_j (B_{Z''}, Z'')   \label{eq:M^1_j+1_external_field}  \\
\MM_{j+1}^{\Psi, (2)} ( \mathfrak{K}_j^{\Psi} (\omega_j),  X) 
	&= \sum_{X_0, X_1, Z, (B_{Z''})}^{\ast} 1_{ \# (X_0, X_1, Z)  \leq 1 } (e^{\cE_{j+1} |X| - \mathfrak{e}'_{j+1} (X) } e^{\bar{U}^{\Psi}_{j+1} (X \backslash T)} - 1) \nonumber \nnb
	&\quad \times \Eplus \Big[ (e^{U_j} - e^{\bar{U}^{\Psi}_{j+1}})^{X_0} (\bar{K}^{\Psi}_j - \mathcal{E}^{\Psi} K_j)^{[X_1]}  \Big] \prod_{Z'' \in \operatorname{Comp}_{j+1} (Z)} J^{\Psi}_j (B_{Z''}, Z'')  \label{eq:M^2_j+1_external_field} 
	\\
	\MM_{j+1}^{\Psi, (3)} ( \mathfrak{K}_j^{\Psi} (\omega_j),  X) 
	&= \sum_{|X_0|_{j+1} = 1}^{X_0 = X} \Eplus \Big[ \Big( e^{U_j} - e^{ \bar{U}_{j+1}^{\Psi} } - U_j + \bar{U}_{j+1}^{\Psi} \Big)^{X_0} \Big]  \label{eq:M^3_j+1_external_field}  \\
\MM_{j+1}^{\Psi, (4)} ( \mathfrak{K}_j^{\Psi} (\omega_j),  X) 
	&= \Eplus \Big[ \sum_{Y\in \cP_j}^{\bar{Y} = X} e^{U_j (Y)} ( K_j (\cdot ; (\Psi_k)_{k<j}) + \Psi_j ) (X\backslash Y) - \mathbb{S} ( K_j +  \Psi_j ) (X) \Big] \label{eq:M^4_j+1_external_field} 
\end{align}
where $\mathfrak{K}_j^{\Psi} \equiv \mathfrak{K}_j^{\Psi} (\omega_j)$ is short for the collection
%\begin{align}
%(\cE_{j+1} |X| - \mathfrak{e}_{j+1} 1_{0\in X}, \ U_j, \ \bar{U}^{\Psi}_{j+1}, \ K_j (\cdot ; (\Psi_k)_{k<j}) + \Psi_j , \ \bar{K}^{\Psi}_j,\ \cE^{\Psi} K_j, \ J^{\Psi}_j )(\omega_j),
%\end{align}
\begin{align}
(\cE_{j+1} |X| - \mathfrak{e}'_{j+1} (X) ,  \, U_j,  \, \bar{U}^{\Psi}_{j+1},  \, K_j (\cdot ; (\Psi_k)_{k<j}) + \Psi_j , \, \bar{K}^{\Psi}_j, \, \cE^{\Psi} K_j, \, J^{\Psi}_j )(\omega_j),
\end{align}
we consider $X \mapsto \cE_{j+1} |X| - \mathfrak{e}_{j+1} 1_{0 \in X}$ as a polymer activity,
%\begin{equation} \label{eq:newbar-U_j}
%\bar{U}_{j+1}^{\Psi} (X, \varphi') := - \cE_{j+1} |X| + 1_{0\in X} \mathfrak{e}_{j+1} +  U_{j+1} ( X, \varphi'),
%\end{equation}
\begin{equation} \label{eq:newbar-U_j}
\bar{U}_{j+1}^{\Psi} (X, \varphi') := - \cE_{j+1} |X| + \mathfrak{e}'_{j+1} (X) +  U_{j+1} ( X, \varphi'),
\end{equation}
and the rest of the notations are those of Definition~\ref{def:evolution_of_e_j_K_j}.
Also notice that $U_{j+1}$ is used in place of $\cU_{j+1}$ to simplify notations.
These look somewhat complicated, but in view of \cite[Lemma~\ref{lemma:bound_on_M^k}]{dgauss1}, it is actually sufficient to check some regularity properties of terms appearing in each $\MM_{j+1}^{\Psi, (k)}$ to show the differentiability of $\cM_{j+1}$ along with the desired estimates \eqref{eq:bound_for_N_j_K_j_external_field} and \eqref{eq:bound_for_derivative_of_Nj_external_field}. We now proceed to supply the necessary details. Our discussion follows closely the line of arguments yielding \cite[Lemmas~\ref{lemma:Ujbound-summary} and~\ref{lemma:bound_on_M^k}]{dgauss1}. We first gather the estimates that will lead to a suitable analogue of \cite[Lemma~\ref{lemma:Ujbound-summary}]{dgauss1}. This is the object of the next lemma.

\begin{lemma} \label{lemma:Ujbound_external_field} 
Under the assumptions of Theorem~\ref{thm:local_part_of_K_j+1_external_field}, for any $\delta >0$, there exists $\epsilon =\epsilon (\delta , \beta, L, C_\Psi) > 0$ such that for $\omega_j \in \cY_j (\epsilon , C_{\Psi})$,
$B\in \mathcal{B}_{j+1}$, $k \in \{ 0,1,2 \}$,
\begin{align}
& \norm{\mathfrak{U} (B, \varphi)}_{h, T_j (\varphi, B)} \leq C(\delta,  \beta, L, C_{\Psi}) ( 1+  \delta c_w \kappa_L w_j (B, \varphi)^2 ) \norm{\omega_j}_{\bar{\Omega}_j}  \label{eq:Ujbound_external_field1}, \\
& \norm{e^{\mathfrak{U} (B, \varphi )} - \sum_{m=0}^{k} \frac{1}{m!} (\mathfrak{U}(B, \varphi ))^m }_{h, T_{j} (\varphi, B)} \leq C(\delta, \beta, L, C_{\Psi})  e^{\delta c_w \kappa_L w_j (B, \varphi)^2} \norm{\omega_j}_{\bar{\Omega}_j} ^{k+1},  \label{eq:Ujbound_external_field}
\end{align}
where $\mathfrak{U}$ is either $U_j$ or $\bar{U}_{j+1}^{\Psi}$. The same inequalities hold with $\mathfrak{U} (B)$ and $C(\delta, \beta, L, C_{\Psi})$ replaced by 
%\sout{$\cE_{j+1} |B| - \mathfrak{e}_{j+1} 1_{0\in B}$} 
$\cE_{j+1} |B| - \mathfrak{e}'_{j+1} (B)$ and $C(\beta, L, C_{\Psi})$, respectively, and $\delta$ set to $0$.
\end{lemma}

\begin{proof}
  For $\mathfrak{U}=U_j$ or $\cE_{j+1} |B|$, the asserted bounds are then an immediate consequence of \cite[Lemma~\ref{lemma:Ujbound}]{dgauss1}.   
  For the remaining choices of  $\mathfrak{U}$, recall 
the definition and the bound on $\bar{U}_{j+1}$ provided by \cite[\eqref{eq:Ubar_def}]{dgauss1} and
\cite[\eqref{eq:Ujbound1}]{dgauss1} that
 for $B\in \cB_{j+1}$,
  \begin{align}
\norm{\bar{U}_{j+1} (B, \varphi )}_{h, T_{j} (B, \varphi)}
\leq  C(\delta,  \beta,L) 
\big( 1+  \delta c_w \kappa_L w_j (B, \varphi)^2 \big) \norm{\omega_j}_{\bar{\Omega}_j}
	.
  \end{align}
Also by Theorem~\ref{thm:e_j_estimate}, we have
%\begin{equation}
%|\mathfrak{e}_{j+1} (\omega_j)| \leq CC_{\Psi} A^{-1} \norm{\omega_j}_{\bar{\Omega}_j}
%	,
%\end{equation}
\begin{equation}
|\mathfrak{e}'_{j+1} (B, \omega_j)| \leq CC_{\Psi} A^{-1} \norm{\omega_j}_{\bar{\Omega}_j},
\end{equation}
and since 
%\sout{$\bar{U}_j^{\Psi} (X, \varphi) =  1_{0\in X} \mathfrak{e}_{j+1} + \bar{U}_{j+1} (X, \varphi')$}
$\bar{U}_j^{\Psi} (X, \varphi) =  \mathfrak{e}'_{j+1} (X) + \bar{U}_{j+1} (X, \varphi')$
 by \eqref{eq:newbar-U_j},  we have
 \begin{align}
\norm{\bar{U}_{j+1}^{\Psi} (B, \varphi )}_{h, T_{j} (B, \varphi)}
\leq  C(\delta,  \beta,L, C_{\Psi}) 
\big( 1+  \delta c_w \kappa_L w_j (B, \varphi)^2 \big) \norm{\omega_j}_{\Omega_j} , \label{eq:Ujbound_external_field3}
  \end{align}
showing \eqref{eq:Ujbound_external_field1}. 
For the second inequality,  assume
$\epsilon \leq 1/C(\delta,\beta, L, C_{\Psi})$ and
$\norm{\omega_j}_{\bar{\Omega}_j} \leq \epsilon$, then the submultiplicativity of norm and \eqref{eq:Ujbound_external_field3} shows
\begin{equation} \label{eq:Ujbound_external_field4}
  \|e^{\bar{U}_{j+1}^{\Psi} }\|_{h,T_j(B, \varphi)} \leq e^{\norm{\bar{U}_{j+1}^{\Psi}}_{h,T_j(B, \varphi)}}  \leq C(\delta,  \beta,L, C_{\Psi})   e^{\delta c_w\kappa_L w_j(B,\varphi)^2} .
\end{equation}
 Then \eqref{eq:Ujbound_external_field3} and \eqref{eq:Ujbound_external_field4} shows \eqref{eq:Ujbound_external_field}.

\end{proof}

We now state the analogue of \cite[Lemma~\ref{lemma:Ujbound-summary}]{dgauss1} in the present context.

\begin{lemma} \label{lemma:derivative_of_components_external_field}
Under assumptions of Theorem~\ref{thm:local_part_of_K_j+1_external_field}, there exist 
\sout{$c_w >0$} $\epsilon \equiv \epsilon (  \beta, L) >0$, $\eta>0$, $C \equiv C(c_w,   \beta, L, {C_{\Psi}})$ and $C_{A} \equiv C_A(c_w, L, A, {C_{\Psi}})$ such that 
\begin{align}
& \norm{D e^{\mathfrak{U} (B, \varphi)} }_{h, T_j (B, \varphi)} \leq C e^{ c_w \kappa_L w_j (B, \varphi)^2} \label{eq:derivatives1_external_field} \\
& \norm{D^2 e^{\mathfrak{U} (B, \varphi)} }_{h, T_j (B, \varphi)} \leq C e^{ c_w \kappa_L w_j (B, \varphi)^2} \label{eq:derivatives2_external_field} \\
& \norm{D J_j^{\Psi} (B, Z, \varphi')}_{h, T_j (B, \varphi')} \leq CA^{-1} e^{ c_w \kappa_L w_{j} (B, \varphi')^2} \label{eq:derivatives3_external_field} \\
& \norm{D  \bar{K}^{\Psi}_j (Z, \varphi)  }_{h, T_j (Z, \varphi)} \leq C_A A^{- (1+ \eta)|Z|_{j+1}} G_j^{\Psi} (Z,\varphi)  \label{eq:derivatives4_external_field} \\
& \norm{D \mathcal{E}^{\Psi} K_j (Z, \varphi')  }_{h, T_j (Z, \varphi')} \leq C_A A^{- (1+ \eta)|Z|_{j+1}} e^{ c_w \kappa_L w_j (Z, \varphi')^2}  \label{eq:derivatives5_external_field}
\end{align}
for $B \in \mathcal{B}_{j+1}$, $Z\in \mathcal{P}_{j+1}$ whenever 
$\omega_j \in \cY_j (\epsilon ( L) , C_{\Psi})$ and $\mathfrak{U}$ is either $U_j$ or $\bar{U}_{j+1}^{\Psi}$ or $\cE_{j+1} |B| - \mathfrak{e}'_{j+1} (B)$.
In the final case, $e^{c_w \kappa_L w_j (B, \varphi)^2}$ can be omitted. 
\end{lemma}
\begin{proof} 
The proof is mostly the same as that of \cite[Lemma~\ref{lemma:derivative_of_components_v2_1}--\ref{lemma:derivative_of_components_v2}]{dgauss1}. The bounds \eqref{eq:derivatives1_external_field} and \eqref{eq:derivatives2_external_field} are consequences Lemma~\ref{lemma:Ujbound_external_field},
cf.~the discussion around~\cite[\eqref{eq:easy-bd1}--\eqref{eq:easy-bd2}]{dgauss1}.
The bound \eqref{eq:derivatives5_external_field} follows directly from \eqref{eq:derivatives3_external_field} (cf.~\cite[Lemma~\ref{lemma:derivative_of_components_v2_1}]{dgauss1}), which in turn follows from{ a bound on} $\norm{D Q_j^{\Psi}}_{h, T_j (Y, \varphi')}$ (namely, \eqref{e:DQPsi-bd} below).
To obtain {this bound}, notice that for $D\in \cB_j$, $Y\in \cS_j$,
%\begin{align}
%Q_j^{\Psi} (D, Y, \varphi') = Q_j (D, Y, \varphi) + 1_{Y\in \cS_j} 1_{0\in D} \E^{\zeta}[\hat{\Psi}_{j, 0} (Y, \zeta) + \hat{D}_{j,0} (Y, \zeta) ]
%\end{align}
\begin{align}
Q_j^{\Psi} (D, Y, \varphi') = Q_j (D, Y, \varphi' ) + 1_{Y\in \cS_j} \frac{ 1_{D \subset Y \cap B_0^*} }{|Y \cap B_0^*|_j } \E^{\zeta}[\hat{\Psi}_{j, 0} (Y, \zeta) + \hat{D}_{j,0} (Y, \zeta) ]
\end{align}
where $D_j (Y, \zeta) = K_{j} (Y, \zeta ; (\Psi_k)_{k<j}) - K_j (Y, \zeta ; 0)$ and $Q_j$ is defined by \cite[\eqref{eq:Q_j_definition}]{dgauss1}.
But \cite[\eqref{eq:Q_j_bound}]{dgauss1} already bounds $Q_j (D, Y, \varphi)$, 
so we actually only have to bound $\Eplus[\hat{\Psi}_{j, 0} (Y, \zeta) + \hat{D}_{j,0}(Y, \zeta) ]$. But
\begin{align}
& \norm{ \Eplus [ \hat{\Psi}_{j,0} (Y, \zeta) ] }_{h, T_j (Y, \varphi')} \leq C (A/2)^{-|Y|_j} \norm{\Psi_j }_{\Omega_j^{\Psi}} \Eplus [G_j^{\Psi}(Y, \zeta)] \leq C_{\Psi} C (A/4)^{-|Y|_j} \norm{\Psi_j }_{\Omega_j^{\Psi}} , \\
& \norm{ \Eplus [ \hat{D}_{j,0} (Y, \zeta)] }_{h, T_j (Y, \varphi')} \leq C A^{-|Y|_j} \norm{\KK_j}_{\Omega_j^{\vec{K}}} \Eplus [G_j (Y, \zeta) ] \leq  C (A/2)^{-|Y|_j} \norm{\KK_j}_{\Omega_j^{\vec{K}}}
\end{align}
so it follows that $Q_j^{\Psi}$ is differentiable with
\begin{equation} \label{e:DQPsi-bd}
\norm{DQ_j^{\Psi} (D, Y, \varphi')}_{h, T_{j} (Y, \varphi')} \leq C A^{-|Y|_j} e^{c_w \kappa_L w_j (D, \varphi')} .
\end{equation}
For \eqref{eq:derivatives4_external_field}, notice that if we write $\bar{\cF}$ for the 
function 
%\sout{$\bar{\cF}(K_j) = \bar{K}_j$, }
\begin{align}
	\bar{\cF}(U_j,K_j) = \bar{K}_j := \sum_{Y\in \cP_j}^{\bar{Y} = X} e^{U_j (X \backslash Y)} K_j (Y)
	,
\end{align}
it follows that $\bar{K}_j^{\Psi} = \bar{\cF} (U_j,K_j (\cdot ; (\Psi_k)_{k<j}) + \Psi_j )$. So by inspecting the proof of \cite[Lemma~\ref{lemma:derivative_of_components_v2}]{dgauss1}, one sees that $D\bar{K}_j^{\Psi}$ satisfies exactly the same bound as $D \bar{K}_j$ (see \cite[\eqref{eq:K_bar_definition}, \eqref{eq:derivatives4-v2}]{dgauss1} for its definition and bound), only with $A$ replaced by $A/2$, i.e.,
\begin{equation}
G_j^{\Psi} (Z, \varphi)^{-1}  \norm{D  \bar{K}^{\Psi}_j (Z, \varphi)  }_{h, T_j (Z, \varphi)} \leq C_A  (A/2)^{- (1+ \eta)|Z|_{j+1}}. 
\end{equation}
But for $A$ large enough, this is less than $C_A A^{- (1+ \eta')|Z|_{j+1}}$ for some $\eta' \in (0, \eta)$
as needed.
\end{proof}

\begin{proof}[Proof of continuous differentiability of $\cM^{\Psi}_{j+1}$ and \eqref{eq:bound_for_N_j_K_j_external_field}, \eqref{eq:bound_for_derivative_of_Nj_external_field}]
For $j\leq N-2$, \cite[Lemma~\ref{lemma:bound_on_M^k}]{dgauss1} implies that the bounds on $\mathfrak{K}^{\Psi}_j (\omega_j) = (\cE_{j+1}, U_j, \bar{U}_{j+1}^{\Psi},  K_j (\cdot ; (\Psi_k)_{k<j}) + \Psi_j ,  \bar{K}_j^{\Psi}, \cE^{\Psi} K_j, J_j^{\Psi})$ provided by
Lemma~\ref{lemma:Ujbound_external_field} and Lemma~\ref{lemma:derivative_of_components_external_field} are sufficient to prove the differentiability and bounds on $\MM_{j+1}^{\Psi, (k)} (\mathfrak{K}_j (\omega_j))$, $k\in \{1,2,3,4\}$. 
In fact, \eqref{eq:derivatives4_external_field} now imposes bound in terms of $G_j^{\Psi}$ instead of $G_j$ but this does not affect the proof because \cite[Lemma~\ref{lemma:bound_on_M^k}]{dgauss1} uses the properties of $G_j$ that (1) $e^{c_w \kappa_L w_j (X)^2} G_j (Y) \leq G_j (X\cup Y)$ if $X\cap Y = \emptyset$,
(2) $G_j (X) = \prod_{\operatorname{Comp}_j (X)} G_j (X)$ and 
(3) $\Eplus [G_j (X, \varphi' + \zeta)] \leq 2^{|X|_j} G_{j+1} (X, \varphi')$. But the same properties are verified on account of Proposition~\ref{prop:key_prop_G_j^Psi}, while the constant $C_{\Psi}$ only contributes as a multiplicative factor in each estimate.

For $j=N-1$, all of the arguments of Sections-\ref{subsec:the_expression_Lj_external_field}--\ref{subsec:Mj_bound_external_field}
continue to apply as $\Gamma_{N}^{\Lambda_N}$ satisfies exactly the same bounds as required for $\Gamma_j$ when $j=N$.
\end{proof}

\section{Proof of Theorem~\ref{thm:Z_N_ratio_conclusion}}
\label{sec:proof_Z_N_ratio}

In Section~\ref{sec:rg_generic_step_external_field},
we defined the extended renormalisation map $\bar{\Phi}_{j+1}$ corresponding to the finite torus $\Lambda_N$. 
In this section, we analyse the limit (as $N \to \infty$) of the final renormalisation group coordinates $(E_N, e_N, U_N, \KK_N, \Psi_N)_{N\geq 0}$
obtained by the iteration of the renormalisation group map up to scale $N$, with initial conditions provided by Theorem~\ref{thm:tuning_s}.
This limit is not exactly as the same as the limit $j\rightarrow \infty$ of the local infinite volume limit; in the former limit the size of the torus $\Lambda_N$ is also varying as $N\rightarrow \infty$.
For this reason, we temporarily write the dependence on $\Lambda_N$ of the coordinates explicitly in the following theorem and the corollary, e.g., the coordinates will be denoted $(E^{\Lambda_N}_j, e_j^{\Lambda_N}, U^{\Lambda_N}_{j}, \KK_j^{\Lambda_N}, \Psi_j^{\Lambda_N})$ and the renormalisation group map will be denoted $\Phi_{j+1}^{\Lambda_N}$ and $\bar{\Phi}_{j+1}^{\Lambda_N}$ for the {bulk} and the extended flows, respectively.

\begin{theorem} \label{thm:dynamics_of_KK_j}
Let $J$ be any finite-range step distribution as in Theorem~\ref{thm:highbeta-Z2},
choose the parameters as in Section~\ref{sec:parameters},
assume that $\beta \geq \beta_0(J)$ as in Theorem~\ref{thm:tuning_s},
and let $(E^{\Lambda_N}_j, U^{\Lambda_N}_j, K^{\Lambda_N}_j)$ be the 
(bulk) renormalisation group map on $\Lambda_N$ as in Theorem~\ref{thm:tuning_s}, i.e.,
\begin{equation}
(E^{\Lambda_N}_{j+1}, U^{\Lambda_N}_{j+1}, K_{j+1}^{\Lambda_N} (\cdot ; 0)) = \Phi_{j+1}^{\Lambda_N} (E^{\Lambda_N}_j, U^{\Lambda_N}_j, K_j^{\Lambda_N} (\cdot ; 0)), \quad 0\leq j \leq N-1 .
\end{equation}
Assume that $(u_j)_{j\geq 0}$ satisfies \ref{assump:u}, and
define $(e_j)_{0 \leq j \leq N}$, 
$(\Psi_j^{\Lambda_N} )_{0\leq j \leq N}$, $(K_j^{\Lambda_N} (\cdot ; (\Psi^{\Lambda_N}_k)_{k<j} ) )_{0\leq j \leq N}$ inductively by
\begin{align}
& \Psi_j^{\Lambda_N} = \cF_{\Psi} [u_j, U^{\Lambda_N}_j , K^{\Lambda_N}_j(\cdot ; (\Psi_{k})_{k<j}) ; j ] \label{eq:Psi_j^Lambda_N_definition} \\
& K_{j+1}^{\Lambda_N} (\cdot ; (\Psi^{\Lambda_N}_k)_{k\leq j}) = \cK_{j+1}^{\Psi, \Lambda_N} ( U^{\Lambda_N}_j, \KK^{\Lambda_N}_j, \Psi^{\Lambda_N}_j ) \\
& e^{\Lambda_N}_{j+1} = e^{\Lambda_N}_j + \mathfrak{e}^{\Lambda_N}_{j+1} (\KK^{\Lambda_N}_j, \Psi^{\Lambda_N}_j)
\end{align}
with initial conditions $K^{\Lambda_N}_0 (X) = 1_{X = \emptyset}$ and $e^{\Lambda_N}_0 = 0$.
%Then there exists $C=C(M_u) \in (0,\infty)$ 
Then there exists $C > 0$ 
such that for all $N \geq 1$ and $0\leq j \leq N$,
if $L$ and $j_u$ are large enough, then
%depending on $M_u$, then
\begin{equation}
 \max\big\{ \norm{\KK_j^{\Lambda_N}}_{\Omega_j^{\KK}} ,\, \norm{\Psi^{\Lambda_N}_j}_{\Omega_j^{\Psi}} \big\}  \leq C L^{-\alpha  j},
\label{eq:vec_V,Psi_decay_bound}
\end{equation}
with decay factor $\alpha \equiv  \alpha (\beta, J ) >0$ as in Theorem~\ref{thm:tuning_s}.
\end{theorem}

\begin{proof}
%  Lemma~\ref{lemma:G_j^Psi_is_Psi_regulator} implies that $\Psi_j^{\Lambda_N}$ defined by \eqref{eq:Psi_j^Lambda_N_definition} falls into the admissible range of Theorem~\ref{thm:local_part_of_K_j+1_external_field}, 
%{\darkmagenta i.e., $(U^{\Lambda_N}_j, \KK_j^{\Lambda_N}, \Psi_j^{\Lambda_N}) \in \cY_j (\tilde{\epsilon}_{nl}, C_{\Psi} )$ }
%%i.e., $(U^{\Lambda_N}_j, \KK_j^{\Lambda_N}, \Psi_j^{\Lambda_N}) \in \cY_j (\tilde{\epsilon}_{nl}, C_{\Psi} (M_u))$ 
%whenever $\norm{U^{\Lambda_N}_j}_{\Omega_j^{U}}$ and $\norm{\KK_j^{\Lambda_N}}_{\Omega_j^{\KK}}$ are sufficiently small.
  
The asserted exponential
decay in $j$ (uniform in $N$) is almost immediate from Theorems~\ref{thm:RG_without_external_field}, \ref{thm:tuning_s}, and \ref{thm:local_part_of_K_j+1_external_field}, as we now explain. Throughout the remainder of the proof, we drop the superscripts $N$ and $\Lambda_N$.
All the following estimates hold uniformly in $N$. 
By Theorem~\ref{thm:tuning_s}, it has already been shown that 
$\omega_j \in \cY_j (\tilde{\epsilon}_{nl}, C_{\Psi})$ 
%$\omega_j \in \cY_j (\tilde{\epsilon}_{nl}, C_{\Psi}= C_{\Psi} (M_u))$ 
and $\norm{(U_j, K_j (\cdot ; 0)  )}_{\Omega_j} \leq C L^{-\alpha j}$ for all $j \leq N$.
We will now argue that there is $C' > 0$ such that, for all $j$, both
\begin{align}
& \label{eq:dynamics_of_K_j_induction_hp}
\norm{ K_j (\cdot ; (\Psi_k)_{k<j }) - K_j (\cdot ; 0) }_{ \Omega_j^{K}} \leq C' 
L^{-\alpha j}, \\
& \label{eq:psi_j-ind}
\norm{\Psi_j}_{\Omega_j^{\Psi}} \leq C_{\Psi} (C + C') L^{-\alpha j} 
\end{align}
hold, where $C$ refers to the constant in  the bound $\norm{(U_j, K_j (\cdot ; 0)  )}_{\Omega_j} \leq C L^{-\alpha j}$.
 The claim then immediately follows by combining these two estimates with \eqref{eq:finalbounds}.
We now show these two bounds by induction. 
For $j\leq j_u$ there is nothing to prove, as $\Psi_j \equiv 0$ and $K_j (\cdot ; (\Psi_k)_{k<j}) \equiv K_j (\cdot ; 0)$.
Now assume \eqref{eq:dynamics_of_K_j_induction_hp} and \eqref{eq:psi_j-ind} hold for some $j \in [j_u, N)$.   
If $j_u$ is sufficiently large, then these bounds and Lemma~\ref{lemma:G_j^Psi_is_Psi_regulator} imply that $\omega_j$ falls into the admissible range of Theorem~\ref{thm:local_part_of_K_j+1_external_field}, 
i.e., $(U^{\Lambda_N}_j, \KK_j^{\Lambda_N}, \Psi_j^{\Lambda_N}) \in \cY_j (\tilde{\epsilon}_{nl}, C_{\Psi})$.
Then \eqref{eq:bound_for_L_j_K_j_external_field} and linearity of $\mathcal{L}_{j+1}^{\Psi}$ 
give for $\omega_j = (U_j, \KK_j, \Psi_j) \in \cY_j (\epsilon, C_{\Psi})$ (with $\epsilon \leq \tilde{\epsilon}_{nl}$)
\begin{equation}
  \norm{ \cL_{j+1}^{\Psi} (\omega_j) (\cdot ; 0 ) - \cL_{j+1}^{\Psi} (\omega_j) (\cdot ; (\Psi_k)_{k<j }, 0 ) }_{ \Omega_{j +1}^{K}}
  \leq  C_1  C_{\Psi}\alphaLoc^{\Psi} \norm{ K^{\Lambda_N}_j (\cdot ; (\Psi_k)_{k<j }) - K^{\Lambda_N}_j (\cdot ; 0) }_{ \Omega_j^{K}}  \label{eq:cL_diff1}
\end{equation}
and \eqref{eq:bound_for_derivative_of_Nj_external_field} gives
\begin{equation}
\norm{ \cM_{j+1}^{\Psi} (\omega_j) (\cdot ; 0 ) - \cM_{j+1}^{\Psi} (\omega_j) (\cdot ; (\Psi_k)_{k<j }, 0 ) }_{ \Omega_{j + 1}^{K}} \leq C_2 \norm{ K^{\Lambda_N}_j (\cdot ; (\Psi_k)_{k<j }) - K^{\Lambda_N}_j (\cdot ; 0) }_{ \Omega_j^{K}} \epsilon .  \label{eq:cM_diff1}
\end{equation}
Here $((\Psi_k)_{k<j }, 0 )$ refers to $(\Psi'_k)_{k\leq j}$ with $\Psi'_k=\Psi_k$ for $k<j$ and $\Psi_j'=0$.
For $\epsilon$ sufficiently small in $\alphaLoc^{\Psi}$ and $(C_2 (\beta, A,L))^{-1}$,  \eqref{eq:dynamics_of_K_j_induction_hp},  \eqref{eq:cL_diff1} and \eqref{eq:cM_diff1} imply
\begin{equation}
\norm{ K^{\Lambda_N}_{j+1} (\cdot ; 0 ) - K^{\Lambda_N}_{j+1} (\cdot ; (\Psi_k)_{k<j }, 0 ) }_{ \Omega_{j + 1}^{K}}
\leq 2 C_1 C_{\Psi} C'  \alphaLoc^{\Psi}  L^{-\alpha j}
\end{equation}
Similar arguments gives
\begin{equation}
  \norm{ K^{\Lambda_N}_{j+1} (\cdot ; (\Psi_k)_{k \leq j }) - K^{\Lambda_N}_{j+1} (\cdot ; (\Psi_k)_{k<j }, 0 ) }_{ \Omega_{j + 1}^{K}}
  \leq 2 C_1  C_{\Psi} \alphaLoc^{\Psi} \norm{\Psi_j}_{{\Omega}_j^{\Psi}} 
\end{equation}
Together with \eqref{eq:psi_j-ind}, these  inequalities imply
\begin{align}
\norm{ K^{\Lambda_N}_{j+1} (\cdot ; (\Psi_k)_{k\leq j }) - K^{\Lambda_N}_{j+1} (\cdot ; 0) }_{ \Omega_{j+1}^{K}} \leq  C'' L^{\alpha}  \alphaLoc^{\Psi} L^{-\alpha (j+1)} .
\end{align}
To proceed, we need the fact that $L^{\alpha} \leq C (L^2 \alphaLoc)^{-1}$ for some $C >0$, see the last remark of Theorem~\ref{thm:tuning_s}.
Also since $\alphaLoc^{\Psi} = (\log L)^{-1} O (L^2 \alphaLoc)$, we now have $L^{\alpha} \alphaLoc^{\Psi} \leq C /\log L$ and therefore
\begin{align}
\norm{ K^{\Lambda_N}_{j+1} (\cdot ; (\Psi_k)_{k\leq j }) - K^{\Lambda_N}_{j+1} (\cdot ; 0) }_{ \Omega_{j+1}^{K}} \leq \frac{C'''}{\log L}  L^{-\alpha (j+1)}
\end{align}
which completes the induction step for \eqref{eq:dynamics_of_K_j_induction_hp} after choosing $C' \log L \geq C'''$. 
To obtain \eqref{eq:psi_j-ind} at scale $j+1$, one now uses that $\norm{(U_{j+1}, K_{j+1} (\cdot ; 0)  )}_{\Omega_j} \leq C L^{-\alpha (j+1)}$ by Theorem~\ref{thm:tuning_s}, and the fact that $\norm{K_{j+1} (\cdot ; (\Psi_k)_{k\leq j})}_{\Omega_{j+1}^K} \leq (C + C') L^{-\alpha {j+1}}$ which follows by combining with the newly proved \eqref{eq:dynamics_of_K_j_induction_hp} at scale $j+1$, along with the fact that
 $\norm{\Psi_{j+1}}_{\Omega_{j+1}^{\Psi}} \leq C_{\Psi} \norm{\KK_{j+1} (\cdot ; (\Psi_k)_{k\leq j}))}_{\Omega_{j+1}^{\KK}}$ by Lemma~\ref{lemma:G_j^Psi_is_Psi_regulator}.
 
%{\magenta We finally note that, by the bound \eqref{eq:vec_V,Psi_decay_bound},   each coordinate is in the domain of Theorem~\ref{thm:local_part_of_K_j+1_external_field} once we choose $j_u$ sufficiently large.}
\end{proof}

\begin{corollary} \label{cor:dynamics_of_e_j}
  Under the assumptions of Thereom~\ref{thm:dynamics_of_KK_j}, 
\begin{equation}
	|e_{N}^{\Lambda_N}| \leq O\Big( \sum_{j\geq j_u} \norm{\KK^{\Lambda_N}_{j}}_{ \Omega_j^{\KK}} \Big) \leq O( {L}^{-\alpha j_{u}} )
\end{equation}
%\begin{equation}
%|e_{N}^{\Lambda_N}| \leq C(M_u) \sum_{j\geq j_u} \norm{\KK^{\Lambda_N}_{j}}_{ \Omega_j^{\KK}} \leq C(M_u) C {L}^{-\alpha j_{u}}
%\end{equation}
%for $M_u$ and $j_u$ 
for $j_u$ from \ref{assump:u}, uniformly in $N$.
\end{corollary}
\begin{proof}
We start from the the explicit expression
$e_n^{\Lambda_N} =
\sum_{j\leq n-1}  \mathfrak{e}_{j+1} (\KK^{\Lambda_N}_j, \Psi_j^{\Lambda_N}) $
and use \eqref{eq:e_j+1_bound}.
To see that the sum actually only starts from $j=j_u$,
note that, by construction, $\Psi_k \equiv 0$ for $k< j_u$, and hence $K_j^{\Lambda_N} (\cdot ; (\Psi_k)_{k<j }) = K_j^{\Lambda_N} (\cdot ; 0)$ for $j\leq j_u$
{which implies that $\mathfrak{e}_{j+1}=0$ by its definition, \eqref{eq:e_j+1_definition}.}
Hence 
$|e_N^{\Lambda_N}| \leq C \sum_{j_u \leq j \leq N-1} \norm{\KK_j^{\Lambda_N}}_{ \Omega_j^{\KK}}$ 
%$|e_N^{\Lambda_N}| \leq C(M_u) \sum_{j_u \leq j \leq N-1} \norm{\KK_j^{\Lambda_N}}_{ \Omega_j^{\KK}}$ 
and the 
{sum is uniformly bounded in $N$} because $\norm{\KK_j^{\Lambda_N}}_{ \Omega_j^{\KK}} = O(L^{-\alpha j})$ uniformly in $N$.
\end{proof}

Theorem~\ref{thm:Z_N_ratio_conclusion}  is almost direct from the above two results.

\begin{proof}[Proof of Theorem~\ref{thm:Z_N_ratio_conclusion}]
We first note that Lemma~\ref{lem:extfield_bd} implies that $(u_j)_{j\geq 0}$ defined by \eqref{eq:extfield_def} satisfies \ref{assump:u} 
with some $j_u = j_f$, 
%with some $M_u = M$ and $j_u = j_f$, 
and so Theorem~\ref{thm:dynamics_of_KK_j} and Corollary~\ref{cor:dynamics_of_e_j} may be used. 
We then assume that $\beta \geq \beta_0(J)$ with $\beta_0(J)$ as supplied by Theorem~\ref{thm:tuning_s}, pick $L=L(J)$ large enough (and of the form specified in Section~\ref{sec:parameters}) such that the conclusions Theorem~\ref{thm:dynamics_of_KK_j} hold and set $A(J)= A_0'(L)$ for this choice of $L$.

For a constant field $\zeta$, we have $\nabla \zeta = 0$ and $G_N^{\Psi} (X, \zeta) = G_N^{\Psi} (X, 0)$ so,
with $W_N$ denoting the non-gradient term (involving the cosines) in \eqref{eq:U_j_form} with $j=N$,
\begin{align}
 e^{E_N |\Lambda_N| - e_N} Z_N  (u, \zeta + u_N ) 
& =  e^{\frac{1}{2} s_N |\nabla (\zeta + u_N)  |^2_{\Lambda_N} + W_N (\Lambda_N, \zeta + u_N) } + K_N (\Lambda_N, \zeta + u_N ; (\Psi_k)_{k<N} ) \nnb 
& =  e^{\frac{1}{2} s_N |\nabla \zeta  |^2_{\Lambda_N} + W_N (\Lambda_N, \zeta ) } + K_N (\Lambda_N, \zeta  ; (\Psi_k)_{k<N} ) + \Psi_N (\Lambda_N, \zeta) \nnb 
&= 1 + O \big( \norm{W_N}_{\Omega_N^U} + \norm{K_N (\cdot ; (\Psi_k)_{k<N}  )}_{\Omega_N^K}  G_N^{\Psi} (\Lambda_N, 0) \big)
\end{align}
whenever $\norm{W_N}_{\Omega_N^U} \leq 1$ and we have used $\Psi_N = \cF_{\Psi}[u_N, U_N, K_N (\cdot ; (\Psi_k)_{k<N}) ; N]$ and Proposition~\ref{prop:reblocking_Z_with_Psi} for the second equality. Also Lemma~\ref{lemma:G_j^Psi_is_Psi_regulator} bounds $\Psi_N$ in terms of $K_N (\cdot ; (\Psi_k)_{k<N})$ in the third equality.
Then by \eqref{eq:Z^tilde_definition} and \eqref{eq:G_j^Psi_zero_point_bound},
\begin{align}\label{eq:Q_N-int}
\tilde{Z}_N (u, 0) &= \E_{t_N Q_N} Z_N (u, \zeta + u_N ) \nnb
&= e^{-E_N |\Lambda_N| + e_N} \big( 1 + O( \norm{W_N}_{\Omega_N^U} + \norm{K_N (\cdot ; (\Psi_k)_{k<N}  )}_{\Omega_N^K} ) \big).
\end{align}
%\begin{align}\label{eq:Q_N-int}
%\tilde{Z}_N (u, 0) &= \E_{t_N Q_N} Z_N (u, \zeta + u_N ) \nnb
%&= e^{-E_N |\Lambda_N| + e_N} \big( 1 + C(M) O( \norm{W_N}_{\Omega_N^U} + \norm{K_N (\cdot ; (\Psi_k)_{k<N}  )}_{\Omega_N^K} ) \big).
%\end{align}
For $\norm{W_N}_{\Omega_N^U} + \norm{K_N (\cdot ; 0 )}_{\Omega_N^K}$ sufficiently small, it follows that
\begin{equation} 
\frac{\tilde{Z}_N (u, 0)}{\tilde{Z}_N (0, 0)} 
  =\exp( e_N )  \frac{ 1 +  O( \norm{W_N}_{\Omega_N^U} + \norm{K_N (\cdot ; (\Psi_k)_{k<N}  )}_{\Omega_N^K} ) }{1 +  O( \norm{W_N}_{\Omega_N^U} + \norm{K_N (\cdot ; 0 )}_{\Omega_N^K} ) }.
\end{equation}
%\begin{equation}
%\frac{\tilde{Z}_N (u, 0)}{\tilde{Z}_N (0, 0)} 
%  =\exp( e_N )  \frac{ 1 + C(M)  O( \norm{W_N}_{\Omega_N^U} + \norm{K_N (\cdot ; (\Psi_k)_{k<N}  )}_{\Omega_N^K} ) }{1 + C(M)   O( \norm{W_N}_{\Omega_N^U} + \norm{K_N (\cdot ; 0 )}_{\Omega_N^K} ) }.
%\end{equation}
But $\norm{W_N}_{\Omega_N^U} \leq C L^{-\alpha N}$ by Theorem~\ref{thm:tuning_s},
$\norm{\KK_N}_{\Omega_N^{\KK}} \leq C_1 L^{-\alpha N}$ 
%$\norm{\KK_N}_{\Omega_N^{\KK}} \leq C_1 (M) L^{-\alpha N}$ 
by Theorem~\ref{thm:dynamics_of_KK_j},
and 
$|e_N| \leq C_2  L^{-\alpha j_f}$ 
%$|e_N| \leq C_2 (M)  L^{-\alpha j_f}$ 
by Corollary~\ref{cor:dynamics_of_e_j}. This implies the desired conclusion.
\end{proof}

\appendix
\section{Existence of infinite-volume limit}
\label{app:corrineq}

We recall the Fr\"ohlich--Park--Ginibre inequalities:
Let $\Lambda$ be finite, let $C$ be a positive definite matrix, and let $\avg{\cdot}_C$ be the
expectation of the associated (generalised) Discrete Gaussian model:
\begin{equation}
\avg{F}_C \propto \sum_{\sigma \in \Z^\Lambda} e^{-\frac12(\sigma, C^{-1}\sigma)} F(\sigma).
\end{equation}
By taking limits, the definition of $\avg{\cdot}_C$ can also be extended to $C$ positive semidefinite.
The finite volume states $\avg{\cdot}^{\Lambda}_{J,\beta}$ given by \eqref{eq:DG_model_1_external_field}
then correspond to $C=\beta(-\Delta_J)^{-1}$ when $\sigma$ is identified up to  constants (as we do), see also \cite[Lemma~\ref{lemma:m2to0}]{dgauss1}.
The results of \cite[Section~3]{MR496191} (see also \cite[Proposition~1.2]{1711.04720}) then imply
that for $f: \Lambda\to\R$ with $\sum f=0$:
\begin{align}
  \label{eq:ginibre_1}
  \avg{e^{(f,\sigma)}}_{J,\beta}^{\Lambda} &\leq e^{\frac12 (f,(-\Delta_J)^{-1}_{\Lambda} f)},
  \\
        \label{eq:ginibre_3}
  \avg{(f,\sigma)^2}_{J,\beta}^{\Lambda} &\leq (f,(-\Delta_J)^{-1}_{\Lambda}f).
\end{align}
Moreover, \cite[Corollary~3.2 (1)]{MR496191} implies that
\begin{equation} \label{eq:ginibre_mon}
  \avg{e^{i(\varphi,f)}}_{C_1}
  \leq
  \avg{e^{i(\varphi,f)}}_{C_2}
  \qquad \text{if $C_2 \leq C_1$.}
\end{equation}

\begin{proposition} \label{prop:infvollimitDG}
  Let $L>1$ be an integer.
  For any finite-range step distribution $J$
  and any sequence of discrete tori $\Lambda_N$ with side lengths $L^N$, with $N \in \N$,
  the measures $\avg{\cdot}^{\Lambda_N}_{J,\beta}$ converge weakly as $N\to\infty$
  (when the field is identified up to constants).
  For any $f:\Z^d\to\R$ with compact support and $\sum f=0$, one also has
  $\avg{e^{(f,\sigma)}}^{\Lambda_N}_{J,\beta} \to    \avg{e^{(f,\sigma)}}$ where $\avg{\cdot}=\lim_{N\to\infty }\avg{\cdot}_{J,\beta}^{\Lambda_N}$ is the weak limit.
\end{proposition}

\begin{proof}
We consider the Laplacian $-\Delta^{\Lambda_N}$ as an operator on $\ell^2(\Z^d)$ with
domain
\begin{equation}
  D(-\Delta^{\Lambda_N}) = \{f \in \ell^2(\Z^d): f(0)=0,\; f(x)=f(x+L^Ny) \; \text{for any $y \in \Z^d$}\}.
\end{equation}
Then clearly $D(-\Delta^{\Lambda_{N}}) \subset D(-\Delta^{\Lambda_{N+1}})$
and $-\Delta^{\Lambda_N} = -\Delta^{\Lambda_{N+1}}$ on $D(-\Delta^{\Lambda_{N}})$.
This implies $-\Delta^{\Lambda_{N}} \geq -\Delta^{\Lambda_{N+1}}$
and hence $(-\Delta^{\Lambda_{N}})^{-1} \leq (-\Delta^{\Lambda_{N+1}})^{-1}$.
From \eqref{eq:ginibre_mon}, it follows that for any $f: \Z^d \to \R$ compactly supported and with $\sum f=0$,
$S_N(f) = \avg{e^{i(f,\varphi)}}^{\Lambda_N}_{J,\beta}$ is increasing in $N$.
In particular, since also $S_N(f) \leq 1$, the limit $S(f) = \lim_{N\to\infty} S_N(f)$ exists.
To show $S(f)$ is the characteristic function of a probability measure on $(2\pi\Z)^{\Z^2}/\text{constants}$ to which $\avg{\cdot}^{\Lambda_N}_{J,\beta}$ converges weakly, we will apply Minlos' theorem.
To this end, we consider $(2\pi\Z)^{\Z^2}/\text{constants}$
as a topological vector space with the topology defined by the condition that $\varphi_k \to \varphi$
in $(2\pi\Z)^{\Z^2}/\text{constants}$ if $(\varphi_k,g) \to (\varphi,g)$ for all compactly supported $g:\Z^d\to \R$ with $\sum g=0$.
In particular, $(2\pi\Z)^{\Z^2}/\text{constants}$ is the dual of a nuclear space.
To apply Minlos' theorem  we need to check that $S$ is continuous in this topology.
But this is immediate from the correlation inequality \eqref{eq:ginibre_3} which implies that
for any $g:\Z^2\to\R$ with compact support and $\sum g=0$,
\begin{equation}
  |S(f+g)-S(f)|
  =\lim_{N\to\infty}   |S_N(f+g)-S_N(f)|
  \leq \lim_{N\to\infty} (g,(-\Delta_J^{\Lambda_N})^{-1}g)
  = (g,(-\Delta_J)^{-1}g),
\end{equation}
from which the continuity is clear.

The final statement about the convergence of $\avg{e^{(f,\sigma)}}^{\Lambda_N}_{J,\beta}$
follows from the weak convergence and
\eqref{eq:ginibre_1} which implies that the random variables $e^{(f,\sigma)}$
are uniformly integrable.
\end{proof}

It is also standard, see \cite{MR2807681} and analogous extensions to the gradient Gibbs setting
as in \cite{MR2228384,MR1463032},
that any limit as in the previous proposition
is translation invariant and satisfies the gradient Gibbs property.
Moreover, the limit satisfies the analogous correlation inequalities.

\begin{proposition}
  The  measure $\avg{\cdot}^{\Z^2}_{J,\beta}$ has tilt $0$, i.e.,
  for each gradient Gibbs state in the ergodic decomposition of $\avg{\cdot}^{\Z^2}_{J,\beta}$ 
  the gradient field has mean $0$.
\end{proposition}

\begin{proof}
  The proof is analogous to that of \cite[Theorem~3.2]{MR1463032}.
The correlation decay can be replaced by the following application of the Riemann--Lebesgue  lemma.
For $g :\Z^2 \to \R^d$ with compact support, where now $\nabla \sigma:\Z^d \to \R^d$ denotes the vector of discrete
forward derivatives,     \eqref{eq:ginibre_3} implies
\begin{equation}
  \avg{(g,\nabla \sigma)^2}_{J,\beta}^{\Z^2} \leq C\int_{[-\pi,\pi]^2} \frac{|\hat g(p) \cdot p|^2}{|p|^2} \, dp.
\end{equation}
Thus the distributional Fourier transform of $\avg{\nabla_{e_i}\sigma(0)\nabla_{e_i}\sigma(x)}$
is integrable in the Fourier variable.
From this, the Riemann-Lebesgue lemma implies that
\begin{equation}
  \avg{\nabla_{e_i}\sigma(x)\nabla_{e_i}\sigma(y)}_{J,\beta}^{\Z^2} \to 0\qquad (|x-y|\to\infty).
\end{equation}
In particular, for every $i=1, \dots, d$,
with $Q_R = [-R,R]^2 \cap \Z^2$,
\begin{equation}\label{e:ergodic}
  \avgbb{\pbb{\liminf_{R\to\infty}\frac{1}{|Q_R|} \sum_{x\in Q_R} \nabla_{e_i}\sigma(x)}^2}_{J,\beta}
  \leq
  \liminf_{R\to\infty} \frac{1}{|Q_R|^2} \sum_{x,y\in Q_R}
  |\avg{\nabla_{e_i}\sigma(x)\nabla_{e_i}\sigma(y)}_{J,\beta}|
  = 0.
\end{equation}
This implies that every measure $\mu$ in the ergodic decomposition of $\avg{\cdot}^{\Z^2}_{J,\beta}$ has mean $0$ for $\nabla\sigma$ (see e.g.~\cite[Theorem~3.2]{MR1463032} for a similar argument): indeed, for any such $\mu$, by \eqref{e:ergodic} and ergodicity, one deduces that $|Q_R|^{-1} \sum_{x\in Q_R} \nabla_{e_i}\sigma(x)$ converges $\mu$-a.s.~and that the limit vanishes, whence $E_\mu[\nabla_{e_i}\sigma(x)]=0$.
\end{proof}

\section{Properties of the regulator with external field}
\label{app:prop_of_reg_with_ext_field}
\begin{proof}[Proof of Lemma~\ref{lemma:G_change_of_scale_external_field}]
In the proof, the notation
\begin{equation}
W_{j+s} (X, \nabla^a_j \varphi)^2 = \sum_{B\in \cB_{j+s} (X)} \norm{\nabla_{j+s}^a \varphi}^2_{L^{\infty}(B^*)}
\end{equation}
will be used.
For brevity, $s + M^{-1}$ will be denoted $s'$ and $X_{s'}$ will be denoted $X'$.
We will bound each term appearing in $\log G_{j+s} (X, \varphi + \xi_o)$.
First, $\norm{\nabla \varphi}^2_{L^2(X)}$ will be isolated from $\norm{\nabla (\varphi + \xi_o)}^2_{L^2 (X)}$. Let $B\in \mathcal{B}_{j+s} (X)$ and without loss of generality, let $B$, $l_i$ ($i=1,2,3,4$) be as above but $B = [1, L^{j+s}]^2$. Then by discrete integration by parts,
\begin{align}
\sum_{x\in B} \nabla^{e_1} \varphi(x) \nabla^{e_1} \xi_o (x) = -\sum_{x\in l_3} \xi_o ( x) \nabla^{-e_1} \varphi (x)
- \sum_{x\in l_1} \xi_o (x+ e_1) \nabla^{e_1} \varphi(x)
+ \sum_{x\in B} \xi_o (x) \nabla^{e_1} \nabla^{-e_1}  \varphi(x).
\end{align}
Hence in particular, summing this over each direction $\pm e_1, \pm e_2$, $B\in \mathcal{B}_{j+s} (X)$, and using the AM-GM inequality,
\begin{align}
t ( \nabla \varphi, \nabla \xi_o )_X & \leq  \tau t \norm{\xi_o }^2_{L^2_{j+s} (X)} +  \tau^{-1} t \norm{\nabla^2_{j+s} \varphi}^2_{L^2_{j+s} (X)} + \tau t \norm{\xi_o }^2_{L^2_{j+s} (\partial X)} + \tau^{-1} t \norm{\nabla_{j+s} \varphi}^2_{L^2_{j+s} (\partial X)}  \nnb
& \leq  2\tau W_{j+s} (X, \xi_o)^2 +  \tau^{-1} \big( \norm{\nabla_{j+s} \varphi}^2_{L^2_{j+s} (\partial X)} + W_{j+s}(X, \nabla^2_{j+s} \varphi)^2 \big)
\end{align}
for any $\tau >0$, and hence
\begin{equation}
\begin{split}
\norm{\nabla_{j+s} (\varphi + \xi_o)}^2_{L^2_{j+s} (X)} & \leq \norm{\nabla_{j+s'} \varphi}^2_{L^2_{j+s'} (X)} + \norm{\nabla_{j+s} \xi_o}^2_{L^2_{j+s} (X)} \\
& \quad + 2\tau W_{j+s} (X, \xi_o)^2 +  \tau^{-1} \big( \norm{\nabla_{j+s} \varphi}^2_{L^2_{j+s} (\partial X)} + W_{j+s}(X, \nabla^2_{j+s} \varphi)^2 \big).
\end{split} \label{eq:G_change_of_scale1_external_field}
\end{equation}
Next, we will use rather trivial bound on the other two terms of $\log G_{j+s}$ :
\begin{align}
& \norm{\nabla_{j+s} (\varphi + \xi_o)}_{L^2_{j+s} (\partial X)}^2 \leq 2\norm{\nabla_{j+s} \varphi}^2_{L^2_{j+s} (\partial X)} + 2 W_{j+s} ( X, \nabla_{j+s} \xi_o)^2 \label{eq:G_change_of_scale2_external_field} \\ 
& \norm{\nabla^2_j (\varphi +\xi_B)}^2_{L^{\infty} (B^*)} \leq 2 \norm{\nabla^2_j \varphi }^2_{L^{\infty} (B^*)} + 2 \norm{\nabla^2_j \xi_B }^2_{L^{\infty} (B^*)} \label{eq:G_change_of_scale3_external_field}
\end{align}
By \eqref{eq:G_change_of_scale1_external_field}, \eqref{eq:G_change_of_scale2_external_field}, \eqref{eq:G_change_of_scale3_external_field} and setting $c_4 = \max \{ 2c_1, 2\tau c_1, 2 c_2 \}$,
\begin{align}
\begin{split}
\frac{1}{\kappa_L} \log G_{j+s} (X, \varphi, \xi_o, (\xi_B)_B ) \leq c_1 \norm{\nabla_{j+s} \varphi}^2_{L^2_{j+s}(X)} + (2c_2 + c_1 \tau^{-1}) \norm{\nabla_{j+s} \varphi}^2_{L^2_{j+s} (\partial X)} \\
+ 2c_1 (1+\tau^{-1}) W_{j+s} (X, \nabla^2_{j+s} \varphi) + \frac{1}{\kappa_L} \log \max_{\mathfrak{a} \in \{ o \} \cup \cB_{j+s} (X)} g_{j+s} (X, \xi_{\mathfrak{a}}).
\end{split}
\end{align}
Now by repeated application of  the discrete Sobolev trace theorem \cite[\eqref{eq:discrete_int_by_parts_rec_step}]{dgauss1},
\begin{align}
\norm{\nabla_{j+s} \varphi}^2_{L^2_{j+s} (\partial X)} \leq \norm{\nabla_{j+s} \varphi}_{L^2_{j+s} (\partial X')}^2 + 10 \norm{\nabla_{j+s} \varphi}^2_{L^2_{j+s} (X' \backslash X)} + 10 W_{j+s} (\nabla^2_{j+s} \varphi, X' \backslash X)
\end{align}
hence by choosing $\tau = c_1 c_2^{-1}$ and $30c_2 \leq c_1$,
\begin{align}
& \frac{\log( G_{j+s} (X, \varphi, \xi ; t, (t_B))  / \max_{\mathfrak{a} } g_{j+s} (X, \xi_{\mathfrak{a}})  )}{\kappa_L} \nnb
& \leq c_1 \norm{\nabla_{j+s} \varphi}^2_{L^2_{j+s} (X')} +3c_2 \norm{\nabla_{j+s} \varphi}^2_{L^2_{j+s} (\partial X')} + 2c_1 (1+ \tau^{-1}) W_{j+s} (\nabla^2_{j+s} \varphi, X') \nnb
& \leq c_1 \norm{\nabla_{j+s'} \varphi}^2_{L^2_{j+s'} (X')} + 3 \, \ell^{-1} c_2 \norm{\nabla_{j+s'} \varphi}^2_{L^2_{j+s'} (\partial X')} + 2 \, \ell^{-2} c_1 (1+ \tau^{-1}) W_{j+s'} (\nabla^2_{j+s'} \varphi, X') . 
\end{align}
Hence the conclusion follows upon taking $\ell$ large enough.
\end{proof}

\section{Reblocking and fluctuation integral}
\label{app:reblocking}

\begin{proof}[Proof of Theorem~\ref{thm:general_RG_step_consistent_external_field}]
Throughout the proof, we write
\begin{equation}
  \varphi = \varphi' + \zeta
\end{equation}
with $\zeta \sim \Gamma_{j+1}$ and $\varphi',  \zeta$ independent,
and the fluctuation integral $\Eplus$ acts on the variable $\zeta$.
As explained in \cite[below \eqref{eq:Phi_j+1_definition}]{dgauss1}, we may assume that $E_j=0$ and $e_j=0$.
The first step is the reblocking
\begin{align}
  Z_j (\varphi, \Psi_j ; (\Psi_k)_{k<j} ) 
  =
  \sum_{X\in \mathcal{P}_{j}} e^{U_j (\Lambda \backslash X)} (K_j (X ; (\Psi_k)_{k<j} ) + \Psi_j (X) )
  =
  \sum_{X\in \mathcal{P}_{j+1}} e^{U_j (\Lambda \backslash X)} \bar{K}^{\Psi}_j (X)
\end{align}
where $\bar{K}^{\Psi}_j$ is defined in \eqref{eq:K_bar^Psi_definition}.
In the next step, $e^{-E_{j+1}|B| + 1_{0\in B} e_{j+1} + U_{j+1}}$ replaces $e^{U_j}$ using the identity
\begin{align}
e^{U_j (\Lambda \backslash X', \varphi)} = \prod_{B\in \mathcal{B}_{j+1} (\Lambda \backslash X')} \Big( \big( e^{U_j ( B, \varphi)} - e^{-E_{j+1}|B| + 1_{0\in B} e_{j+1} + U_{j+1} (B, \varphi')} \big) + e^{-E_{j+1} |B| + 1_{0\in B} e_{j+1}  +  U_{j+1} (B, \varphi')} \Big)  \nnb 
 = \sum_{Y \in \mathcal{P}_{j+1} (\Lambda \backslash X')} e^{-E_{j+1}|\Lambda \backslash (X'\cup Y)| + 1_{0\in \Lambda \backslash (X' \cup Y )} e_{j+1} + U_{j+1} (\Lambda \backslash (X'\cup Y), \varphi')}  \big( e^{U_j (\varphi)} - e^{E_{j+1}|B| + 1_{0\in B} e_{j+1}  +  U_{j+1} (\varphi')} \big)^{Y}
\end{align}
and similarly $\bar{K}^{\Psi}_j - \cE^{\Psi} K_j$ replaces $\bar{K}^{\Psi}_j$ (recall $\cE^{\Psi} K_j$ from \eqref{eq:cEK^Psi_definition}) using the identity
\begin{align}
\bar{K}^{\Psi}_j (X', \varphi) &= \prod_{Z' \in \operatorname{Comp}_{j+1} (X')} \big( \mathcal{E}^{\Psi} K_j (Z', \varphi') + ( \bar{K}^{\Psi}_j (Z', \varphi ) - \mathcal{E}^{\Psi} K_j (Z', \varphi') ) \big) \nnb 
&= \sum_{Z \in \mathcal{P}_{j+1} (X')}^{Z \not\sim X' \backslash Z} \mathcal{E}^{\Psi} K_j (\varphi')^{[Z]} (\bar{K}^{\Psi}_j (\varphi)  - \mathcal{E}^{\Psi} K_j (\varphi')  )^{[X' \backslash Z]} 
.
\end{align}
Using the specific form of $\cE^{\Psi} K_j$ given by \eqref{eq:cEK^Psi_definition} the last right-hand side can be rewritten as
\begin{equation}
  \mathcal{E}^{\Psi} K_j (\varphi')^{[Z]}
  = \sum_{(B_{Z''})_{Z''}} \prod_{Z''} J_j^{\Psi} (B, Z'')
\end{equation}
where the last sum $(B_{Z''})_{Z''}$ runs over collections of blocks $B_{Z''} \in \cB_{j+1} (Z'')$
and $Z'' \in \operatorname{Comp}_{j+1}(Z)$.
Rewriting $X'' = X'\cup Y$,
the expectation $\Eplus Z_j$ can now be written as
\begin{align}
  Z_{j+1}(\varphi' , 0 ; (\Psi_k)_{k\leq j})
  &=
    \Eplus Z_j (\varphi' + \zeta, \Psi_j ; (\Psi_k)_{k<j})
    \nnb
  &= e^{-E_{j+1} |\Lambda|}
    \Eplus \Bigg[
    \sum_{X'' \in \mathcal{P}_{j+1}} e^{1_{0\in \Lambda\backslash X''} e_{j+1} +  U_{j+1} (\Lambda \backslash X'')} e^{E_{j+1} |X''|}
    \nnb
  &\qquad\qquad\qquad \times \sum_{X'\subset X''} (e^{U_j} - e^{-E_{j+1} |B| + 1_{0\in B} e_{j+1} + U_{j+1}})^{X''\backslash X'} \\
  &\qquad\qquad\qquad   \times \sum_{Z \subset X'} (\bar{K}^{\Psi}_j - \mathcal{E}^{\Psi} K_j)^{[X' \backslash Z]} \sum_{(B_{Z''})} \prod_{Z'' \in \operatorname{Comp}_{j+1} (Z)} J^{\Psi}_j (B, Z'') \Bigg].
    \notag
\end{align}
The final result is obtained  after taking $e^{e_{j+1}}$ out and another resummation: we write
$X_0 = X'' \backslash X'$, $X_1 = X' \backslash Z$, $T=X_0 \cup X_1 \cup Z = X''$ and define
for $X = \cup_{Z''}B^*_{Z''} \cup X_0 \cup X_1$,
\begin{multline}
K_{j+1} (X, \varphi' ; (\Psi_k)_{k \leq j} ) = \sum_{X_0, X_1, Z, (B_{Z''})} e^{E_{j+1} |T| - 1_{0\in T} e_{j+1} } e^{U_{j+1} (X \backslash T)} \\
\times \Eplus \Big[ (e^{U_j} - e^{-E_{j+1} |B| + 1_{0\in B}e_{j+1} + U_{j+1}})^{X_0} (\bar{K}^{\Psi}_j - \mathcal{E}^{\Psi} K_j)^{[X_1]}  \Big] \prod_{Z'' \in \operatorname{Comp}_{j+1} (Z)} J^{\Psi}_j (B, Z'') . \label{eq:expression_for_K_j+1_in_appendix_external_field}
\end{multline}
Note that only $T\subset X$ contribute because, by definition of $\cE^{\Psi} K_{j}$, the whole expression vanishes when $Z \not\in \cS_{j+1}$.
Therefore
\begin{equation}
Z_{j+1} (\varphi' , 0 ; (\Psi_k)_{k\leq j} ) = e^{-E_{j+1} |\Lambda| + e_{j+1} } \sum_{Z \in \mathcal{P}_{j+1}} e^{U_{j+1} (\Lambda \backslash Z,\varphi')} K_{j+1} (Z, \varphi' ; (\Psi_k)_{k \leq j} )
\end{equation}
which is the desired form.  The factorisation property of $K_{j+1} (\cdot ; (\Psi_k)_{k\leq j})$ is inherited from that of $e^{U_j}$, $e^{U_{j+1}}$ and $K_j$.
\end{proof}

\section*{Acknowledgements}

R.B.\ was supported by the European Research Council under the European Union's Horizon 2020 research and innovation programme
(grant agreement No.~851682 SPINRG). He also acknowledges the hospitality of the Department of Mathematics at McGill University
where part of this work was carried out.
J.P.\ was supported by the Cambridge doctoral training centre Mathematics of Information.

\bibliography{all}
\bibliographystyle{plain}

\end{document}